\documentclass[hyperref, 11pt]{article}
\usepackage{pifont}
\usepackage{hyperref}
\hypersetup{colorlinks=true}
\usepackage{amsfonts,amscd}
\usepackage{amsmath,amsthm,amssymb,epsfig,mathrsfs}
\usepackage{latexsym}
\usepackage[all]{xy}
\usepackage{dsfont}
\usepackage{color}
\usepackage{comment}
\usepackage{citeref}

\newcommand{\K}{\mbox{$\mathcal K$}}
\newcommand{\U}{\mbox{$\mathcal O$}}

\usepackage[a4paper]{geometry}
 \makeatother
\newtheorem{thm}{Theorem}[section]

\newtheorem{lem}[thm]{Lemma}

\newtheorem{cor}[thm]{Corollary}

\newtheorem{prop}[thm]{Proposition}
\newtheorem{rmk}[thm]{Remark}
\newtheorem{order}[thm]{}

\numberwithin{equation}{thm}

\begin{document}
	\title{Hyperfocal subalgebras of hyperfocal abelian Frobenius blocks}
\author{Xueqin Hu$^{1,}$\footnote
	{E-mail address: {\it hxq@mail.ccnu.edu.cn}} , \,\  Kun Zhang$^2$
	and Yuanyang Zhou$^1$}
\date{\small
	$1.$ School of Mathematics and Statistics, Central China Normal University, Wuhan, 430079, China\\
	$2.$ Faculty of Mathematics and Statistics, Hubei University, Wuhan, 430062, China}
\maketitle


%
	
	\vskip 1cm
	
	\noindent{\small{{\bf Abstract}
	In this paper, we introduce a class of blocks
	which is called hyperfocal abelian Frobenius blocks.
	This class of blocks is an analogous version of blocks with an abelian defect group and Frobenius inertial quotient at the hyperfocal level
	and includes the blocks with a Klein four hyperfocal subgroup or a cyclic hyperfocal subgroup.
	We show that there is a stable equivalence of Morita type 
	between the hyperfocal subalgebra of a hyperfocal abelian Frobenius block
	and 
	a group algebra of a Frobenius group associated with the hyperfocal subgroup of the block. 	  
	As applications, 
	we can partially describe some structures of the blocks with a 	
	Klein four hyperfocal subgroup or a cyclic hyperfocal subgroup,
	such as the structures of their hyperfocal subalgebras in terms of derived categories
	and the structure of their characters.
	As consequences of these applications,
	we  show that Brou$\acute{\mathrm{e}}$'s abelian defect group conjecture holds for 
	blocks with a Klein four hyperfocal subgroup and the forward  direction of a conjecture
	proposed by Kessar and Linckelmann and Navarro, which can be viewed as a 
	`hyperfocal height zero' version of Brauer's height zero conjecture holds for
	blocks with a Klein four hyperfocal subgroup or a cyclic hyperfocal subgroup.
	}}

	\medskip\noindent{\small{{{\bf Keywords} 	
				hyperfocal abelian Frobenius block; hyperfocal subalgebra;
				Brou$\acute{\mathrm{e}}$'s abelian defect group conjecture;  
				Rouquier's conjecture; KLN conjecture
			} }
		
		
		\vskip 2cm\section{Introduction}
In this paper, we always assume that $p$ is a prime.
Let $(\mathcal{K},\U,k)$ be a complete $p$-modular system
consisting of a complete discrete valuation ring $\U$ 
with a residue field $k$ of prime characteristic $p$
and a quotient field $\mathcal{K}$ of characteristic $0$.
Moreover, we always assume that the residue field $k$ is algebraically closed 
and the quotient field $\K$ is large enough for all finite groups considered below.

Let $G$ be a finite group and $b$ a block of $G$ over $\U$ with defect group $P$. 		
This means that $b$ is a primitive idempotent of the center $Z(\U G)$ of the group algebra $\U G$.
In \cite{P00}, Puig introduced a particular normal subgroup $D$ of the defect group $P$,
where he called a hyperfocal subgroup of $P$.
Like the defect group $P$,
the hyperfocal subgroup $D$ also plays a key role in determining the algebra structure of
the block algebra $\U Gb$.
In \cite[A.2]{Rou98}, Rouquier made a conjecture,
now known as Rouquier's conjecture, that predicts that the block algebras
$\U Gb$ and $\U N_G(D)c$ are basically Rickard equivalent in the sense of
\cite[19.1]{Puig's book99} when the hyperfocal subgroup $D$ is abelian.
Here, the block $c$ denotes the Brauer correspondent of the block $b$ in $N_G(D)$.
Rouquier's conjecture is a generalization of 
Brou$\acute{\mathrm{e}}$'s abelian defect group conjecture  and 
would also offer, if true, an explanation for the blockwise Alperin weight conjecture
for the case where the defect group is nonabelian.

In the light of Rouquier's conjecture,
some recent work focuses on understanding the structure of the block algebra $\U Gb$
with some special abelian hyperfocal subgroup.
In \cite{W14} and \cite{HZ19},
the numbers of ordinary irreducible characters and irreducible Brauer characters 
were determined for blocks with a cyclic or a Klein four hyperfocal subgroup, respectively.
In \cite{HZ22}, the first and third authors constructed an isotypy,
which is a character-theoretic shadow of Rouquier's conjecture,
between the blocks $b$ and $c$ 
when the hyperfocal subgroup $D$ is cyclic and the quotient group $P/D$ is abelian.
These results reveal the influences of the hyperfocal subgroup
on the character structure of block algebras.
It may be more interesting and meaningful if the algebra structure of 
the block algebras with some special abelian hyperfocal subgroups can be determined.

Recently,  we  investigated the algebra structure of
the block algebra $\U Gb$ with cyclic hyperfocal subgroup $D$ in \cite{HZZ24}.
We showed that Brou$\acute{\mathrm{e}}$'s abelian defect group conjecture 
holds for blocks with cyclic hyperfocal subgroups (see \cite[Theorem 1.1]{HZZ24}).
One of the key steps of its proof is the description of the structure of
the hyperfocal subalgebra of the block $b$,
which is introduced in \cite{P00} as a unitary subalgebra of
the source algebra of the block $b$.
More explicitly, in \cite[Theorem 1.2]{HZZ24},
we showed that the hyperfocal subalgebra is Rickard equivalent to a symmetric serial algebra
associated with the hyperfocal subgroup $D$ 
when the hyperfocal subgroup $D$ is cyclic and the defect group $P$ is abelian.
This is an analogous result of the structure of blocks with cyclic defect groups 
at the hyperfocal level.
This work suggests that 
in order to successfully describe the structure of 
the block algebras with some special hyperfocal subgroups,
one should first determine the structure of their hyperfocal subalgebras.

Motivated by this, in this paper,
we will focus on the structure of the hyperfocal subalgebra of a block
with a special class of hyperfocal subgroups,
which includes the two classical cases where 
the hyperfocal subgroup is a Klein four group or cyclic, respectively. 
To make this explicit, we introduce some notation.
Let $\mathbb{A}$ be a hyperfocal subalgebra of the block $b$.
Let $D_\delta$ be a local pointed group of $D$ on the block algebra $\U Gb$
and denote by $N_G(D_\delta)$ its stabilizer under the conjugation action of $G$
(see the paragraphs 2.9-2.11 below).
As shown in Proposition \ref{p-nilpotent quot} below,
the quotient group $N_G(D_\delta)/DC_G(D)$ is $p$-nilpotent.
Its normal $p$-complement is denoted by $E_{\mathfrak{hyp}}$.
Since the order of $E_{\mathfrak{hyp}}$ is coprime to $p$,
we can view it as a subgroup of the automorphism group $\mathrm{Aut}(D)$ of $D$.
Then the block $b$ is called with a \emph{Frobenius hyperfocal inertial quotient}
if $E_{\mathfrak{hyp}}$ acts freely on $D-\{1\}$, namely
the set of all nontrivial elements of $D$ (see the paragraph 3.8 below).
This is equivalent to requiring that the semidirect product
$D\rtimes E_{\mathfrak{hyp}}$ is a Frobenius group.
Now we can state the main result of this paper as follows.
\vskip 0.5cm

\noindent {\bf Theorem 1.1.} (see also Theorem \ref{MT})
\emph{Keep the notation as above.
Suppose that the hyperfocal subgroup $D$ is abelian and 
the block $b$ is with a Frobenius hyperfocal inertial quotient.
Then there is a stable equivalence of Morita type between 
the hyperfocal subalgebra $\mathbb{A}$ and the group algebra $\U(D\rtimes E_{\mathfrak{hyp}})$.}
\vskip 0.5cm

Throughout this paper, a block with an abelian hyperfocal subgroup and  
a Frobenius hyperfocal inertial quotient is called a \emph{hyperfocal abelian Frobenius block} 
for short.
By definition, such block
is the hyperfocal analogue of the block with an abelian defect group and 
a Frobenius inertial quotient,
which is called an \emph{abelian Frobenius block} here for simplicity.
Thus, the main theorem above can be viewed as a hyperfocal version 
of \cite[Theorem 10.5.1]{L'book 2},
which is due to Puig (see \cite[6.8]{P91}).
Just as abelian Frobenius blocks include
the blocks with Klein four defect groups or cyclic defect groups,
the hyperfocal abelian Frobenius blocks include
those with Klein four hyperfocal subgroups or cyclic hyperfocal subgroups
(see Remark \ref{examples of hyp frob} below).
Thus we can apply this main theorem to the blocks with 
Klein four hyperfocal subgroups or cyclic hyperfocal subgroups 
and obtain the structure of the corresponding hyperfocal subalgebras as follows,
which are similar to the structure of the blocks with 
Klein four defect groups or cyclic defect groups, respectively.
\vskip 0.5cm

\noindent{\bf Theorem 1.2.} (see also Theorem \ref{MT of Klein four} 
and Proposition \ref{Broue conj for Klein 4})
\emph{Assume that the hyperfocal subgroup $D$ is a Klein four group.
	Then the hyperfocal subalgebra $\mathbb{A}$ is Morita equivalent to either $\U\mathrm{A}_4$
	or the principal block algebra of $\U\mathrm{A}_5$.
In particular, the hyperfocal subalgebra $\mathbb{A}$ is Rickard equivalent to $\U\mathrm{A}_4$,  and then
	Brou$\acute{e}$'s abelian defect group conjecture is true for the blocks 
	with Klein four hyperfocal subgroups.
	Here, $\mathrm{A}_n$ denotes the alternating group on $n$ letters for any positive integer $n$.}
\vskip 0.5 cm

\noindent{\bf Theorem 1.3.}  (see also Theorem \ref{MT for cyclic case})
\emph{Assume that the hyperfocal subgroup $D$ is a nontrivial cyclic group.
Then the hyperfocal subalgebra $\mathbb{A}$ is Rickard equivalent to the group algebra
$\U(D\rtimes E_{\mathfrak{hyp}})$.}
\vskip 0.5cm

Theorem 1.3 is a generalization of \cite[Theorem 1.2]{HZZ24},
as we no longer require the defect group to be abelian, and thus provides a complete description of the hyperfocal subalgebra with a nontrivial cyclic hyperfocal subgroup. This result also answers a question proposed by Professor Linckelmann some years ago, 
asking whether such a hyperfocal subalgebra is a Brauer tree algebra.

Furthermore, by using the Clifford theoretic relationship between the representation theory of 
the source algebra and its hyperfocal subalgebra,
we can also obtain some information about the characters of the source algebra of a block with 
a Klein four hyperfocal subgroup or a cyclic hyperfocal subgroup.
We refer to Propositions \ref{structure of characters for klein 4}
and \ref{structure of characters for cyclic} for detailed descriptions of these structures.
A key step towards these descriptions is to determine the relationship between 
the dimenisons of simple $\mathcal{K}\mathop\otimes\limits_\mathcal{O}\mathbb{A}$-modules and 
the prime $p$.
This is closely related to a conjecture proposed by Kessar and Linckelmann and Navarro in \cite{KLN}, which we will refer to as the KLN conjecture.
We recall this conjecture in the following.

Denote by $\mathrm{Irr}_\mathcal{K}(\mathbb{A})$ a set of representatives of isomorphism classes 
of simple $\mathcal{K}\mathop\otimes\limits_\mathcal{O}\mathbb{A}$-modules and
by $\mathrm{Irr}_{\mathcal{K}, 0}(\mathbb{A})$ 
the subset corresponding to simple modules of dimensions coprime to $p$.
In \cite{KLN}, Kessar and Linckelmann and Navarro proposed the following conjecture.
\vskip 0.25cm

\noindent {\bf KLN Conjecture.} (\cite[Conjecture 1.2]{KLN})
\emph{With the notation above, assume that $\mathcal{K}$ is a splitting field of $\mathbb{A}$. 
Then $\mathrm{Irr}_\mathcal{K}(\mathbb{A})=\mathrm{Irr}_{\mathcal{K}, 0}(\mathbb{A})$ if and only if
the hyperfocal subgroup $D$ is abelian.}
\vskip 0.25cm

In \cite{KLN}, the KLN conjecture is viewed as a 
`hyperfocal height zero' version of Brauer's height zero conjecture.
In this paper, we can verify the forward direction of the KLN conjecture in the cases 
where the hyperfocal subgroup $D$ is a Klein four group or a nontrivial cyclic group.
\vskip 0.5cm

\noindent {\bf Proposition 1.4.} (see also Proposition \ref{KLN conj for klein 4})
When the hyperfocal subgroup $D$ is a Klein four group,
then $\mathrm{Irr}_\mathcal{K}(\mathbb{A})=\mathrm{Irr}_{\mathcal{K}, 0}(\mathbb{A})$.
\vskip 0.5cm

\noindent {\bf Proposition 1.5.} (see also Proposition \ref{KLN conj cyclic})
When the hyperfocal subgroup $D$ is nontrivial cyclic,
then $\mathrm{Irr}_\mathcal{K}(\mathbb{A})=\mathrm{Irr}_{\mathcal{K}, 0}(\mathbb{A})$.
\vskip 0.5cm

Now let us now briefly outline the construction of the stable equivalence of Morita type 
stated in  Theorem 1.1.
This construction  is analogous to that in \cite[Theorem 10.5.1]{L'book 2} 
for abelian Frobenius blocks.
The main ingredients there are the descriptions of local structures 
of an abelian Frobenius block and Puig's Gluing Theorem 
for indecomposable endopermutation modules of abelian $p$-groups.
Let us make it more explicit.
For an abelian Frobenius block with defect group $P$,
its Brauer correspondents in the centralizers of nontrivial subgroups of $P$ 
are nilpotent in the sense of \cite{BP80}.
By the structure theorem of the source algebras of nilpotent blocks, due to Puig \cite{P-88},
these local structures can yield a family of indecomposable endopermutation modules 
(see the paragraph 2.14 below)
of quotient groups of the abelian $p$-group $P$ satisfying certain compatibility conditions.
Then by Puig's Gluing Theorem, these endopermutation modules can be glued together to 
an indecomposable endopermutation module of $P$,
which plays a crucial role in the construction of the 
stable equivalence of Morita type for abelian Frobenius blocks. 

We adapt this approach to the hyperfocal subalgebra of a hyperfocal abelian Frobenius block.
Instead of analyzing the Brauer correspondents 
in the centralizers of nontrivial subgroups of the defect group $P$,
we need to understand the Brauer correspondents in the centralizers of nontrivial subgroups of the hyperfocal subgroup $D$,
which are shown to be also nilpotent by Proposition \ref{nilpotent block of D at local} below.
Naturally, the next step should be applying the structure theorem of 
the source algebras of nilpotent blocks to these blocks 
to yield a family of indecomposable endopermutation modules of quotient groups of $D$.
However, an obstruction arises when the defect group $P$ is nonabelian
because the Brauer quotients of the source algebra at nontrivial subgroups of $D$ 
may not be the source algebras of the blocks of the corresponding centralizers.
To circumvent this,
instead of considering the source algebra and the hyperfocal subalgebra,
we will focus on a so-called embedded algebra $\mathbb{A}_\delta$ of the pointed group 
$D_\delta$ of the hyperfocal subalgebra $\mathbb{A}$ (see the paragraph 2.9 below),
which can be viewed as a subalgebra (not unitary in general) 
of the hyperfocal subalgebra $\mathbb{A}$.
The embedded algebra $\mathbb{A}_\delta$ shares several properties with the source algebra,
such as the existence of a Morita equivalence between 
$\mathbb{A}_\delta$ and the hyperfocal subalgebra $\mathbb{A}$
(see Corollary \ref{separable and Morita equi} below) and 
an $\U(D\times D)$-module structure similar to the $\U(P\times P)$-module structure
of the source algebra (see Proposition \ref{general mod stru} below).
These properties are essential for describing the Brauer quotients of $\mathbb{A}_\delta$ at nontrivial subgroups of $D$
and eventually for constructing the desired stable equivalence of Morita type.

The paper is organized as follows.
In Section $2$, some necessary notation and definitions are collected.
In Section $3$, the definition of the  hyperfocal abelian Frobenius block is introduced
after we show that the quotient group $N_G(D_\delta)/DC_G(D)$ is $p$-nilpotent.
Then the local structures of such block associated with 
nontrivial subgroups of the hyperfocal subgroup are determined.
Also, some basic properties of the embedded algebra $\mathbb{A}_\delta$ are established
in this section.
Section $4$ is devoted to detailed study of the structures of 
the embedded algebra $\mathbb{A}_\delta$,
including its $\U(D\times D)$-module structure and 
the algebra structure of its Brauer quotient 
at every nontrivial subgroup of the hyperfocal subgroup $D$.
These results are crucial for the construction of a stable equivalence of Morita type   
in Section $5$.
In addition to this construction, 
we also investigate in Section $5$ an equivariant property of the bimodule 
inducing this stable equivalence,
which is necessary to describe the structure of the source algebra.
Finally, in Section $6$, we apply the main result of Section 5 to blocks with a Klein four hyperfocal subgroup or a cyclic hyperfocal subgroup, determining the structures of their hyperfocal subalgebras and obtaining some character-theoretic information of their source algebras.

		\vskip2cm
		\section{Preliminaries}
		
In this section, we collect some necessary notation and definitions.

		\begin{order}\rm 
Let $S$ be a finite set.
We denote by $|S|$ the cardinality of $S$.
Let $H$ be a finite group and $\mathfrak{b}$ a block of $H$ 
			over $\mathcal{O}$.
			Denote by $\mathrm{Irr}(H)$ and $\mathrm{Irr}(\mathfrak{b})$
			the set of ordinary irreducible characters of $H$ and 
			the set of ordinary irreducible characters of $\mathfrak{b}$,
			respectively.
			Similary, we denote by $\mathrm{IBr}(H)$ and $\mathrm{IBr}(\mathfrak{b})$
			the set of modular irreducible characters of $H$ and 
			the set of modular irreducible characters of $\mathfrak{b}$,
			respectively.
Let $H_1$ and $H_2$ be two subgroups of $H$.
For any $\xi\in\mathrm{Irr}(H_1)$,
we denote by $\mathrm{Ind}_{H_1}^H(\xi)$ 
			the induced character of $\xi$.
			We use the similar notation for the induced module.
			Analogously,
			we denote the restriction by $\mathrm{Res}_{H_1}^H$.
For any $\mathcal{O}H_1$-module $M$ 
and any $h\in H$,
we define the twisted $\mathcal{O}(hH_1h^{-1})$-module 
$_{h}M$ by setting $_{h}M=M$ as an $\mathcal{O}$-module,
with $hh_1h^{-1}\in hH_1h^{-1}$ acting on $m\in{_{h}M}$ 
as $h_1$ on $m\in M$.
Furthermore, 
if $M$ is an $\mathcal{O}H_1$-$\mathcal{O}H_2$-bimodule,
then by preserving the right $\mathcal{O}H_2$-module structure,
$_{h}M$ is an $\mathcal{O}(hH_1h^{-1})$-$\mathcal{O}H_2$-bimodule.
At the same time, $M$ can be viewed as an $\U(H_1\times H_2)$-module
with the action of $(h_1,h_2)\in H_1\times H_2$ on $m\in M$ defined as
$h_1mh_2^{-1}$ and vice versa.
		\end{order}

		\begin{order}\rm
			Throughout this paper,
			an $\U$-algebra always satisfies the properties that
			it is  free and has finite rank as an $\U$-module.
An $\U$-algebra $C$ is called an \emph{indecomposable} algebra if it is indecomposable as a $C$-$C$-bimodule,
which is equivalent to saying that its center $Z(C)$ is local. 
			Let $\tilde{C}$ be a subalgebra of $C$.
			We denote by $\mathrm{Res}_{\tilde{C}}(M)$ the 
			restriction of $M$ to $\tilde{C}$ for any left $C$-module $M$.
			Denote by $1_C$ (or 1 if no confusion arises), $J(C)$ and $C^\times$ 
			the unit element of $C$,
			the Jacobson radical of $C$, 
			the set of invertible elements of $C$, respectively.
			We denote by $C^\circ$ the opposite algebra of $C$.
			Let $C^\prime$ be another $\mathcal{O}$-algebra.
						Let $\varphi$ be a homomorphism from $C$ to $C^\prime$
			and $M^\prime$ a left $C^\prime$-module.
			We use ${_{\varphi}}M^\prime$ to denote the left $C$-module 
			with being equal to $M^\prime$ as $\U$-modules and the left action of $c\in C$ 
			on $_{\varphi}M^\prime$
			being equal to the left action of $\varphi(c)$ on $M^\prime$.
			For any $C$-$C^\prime$-bimodule $\mathbb{M}$,
			we can define 
			$(x\mathop\otimes\limits_\mathcal{O}x^\prime){\rm m}=x{\rm m}x^\prime$
			for any $x\in C$,
			any $x^\prime\in(C^\prime)^\circ$
			and any ${\rm m}\in\mathbb{M}$.
			In this way,
			it can be viewed as a $C\mathop\otimes\limits_\mathcal{O}(C^\prime)^\circ$-module
			and vice versa.
			Throughout this paper, 
			we identify $C$-$C^\prime$-bimodules with 
			$C\mathop\otimes\limits_\mathcal{O}(C^\prime)^\circ$-modules through this way
			and vice versa.
		\end{order}
		

\begin{order}\rm
	We denote the Heller operator of $C$ by $\Omega_C$ for an $\U$-algebra $C$.
	This operator can induce an equivalence on the 
	$\mathcal{O}$-\emph{stable category of} $\mathrm{mod}(C)$,
	of which the definition we recall as follows.
	Recall that a $C$-module $U$
	is called \emph{relative} $\mathcal{O}$-\emph{projective},
	if it is isomorphic to a direct summand of 
	$C\mathop\otimes\limits_\mathcal{O}V$ for some $\mathcal{O}$-module $V$
	(see \cite[Definition 2.6.11]{L'book 1}).
	Note that if the $C$-module $U$ is relative $\mathcal{O}$-projective
	and free of finite rank,
	then $U$ is a projective $C$-module.
	Denote by $\mathrm{mod}(C)$ the category of finitely generated left
	$C$-modules.
	The quotient category of $\mathrm{mod}(C)$ by 
	the subcategory of relatively $\mathcal{O}$-projective modules 
	is called the $\mathcal{O}$-\emph{stable category of} $\mathrm{mod}(C)$
	(see \cite[11.1]{KZ98}).
	We denote it by $\overline{\mathrm{mod}}(C)$.			
\end{order}

\begin{order}\rm
	Let us recall the definitions of the stable equivalence of Morita type
	and Rickard equivalence.
	Let $C$ and $C^\prime$ be two $\mathcal{O}$-algebras.
	A $C$-$C^\prime$-bimodule $M$ and a $C^\prime$-$C$-bimodule $M^\prime$
	are said to induce a \emph{stable equivalence of Morita type between}
	$C$ and $C^\prime$
	(\cite{B94}),
	if $M$ and $M^\prime$ are projective both as left and right modules,  
	and if
	$M\mathop\otimes\limits_{C^\prime}M^\prime\cong C\oplus X$ as $C$-$C$-bimodules
	and 
	$M^{\prime}\mathop\otimes\limits_C M\cong C^\prime\oplus X^\prime$ 
	as $C^\prime$-$C^\prime$-bimodules,
	where $X$ is a projective $C$-$C$-bimodule and
	$X^\prime$ is a projective $C^\prime$-$C^\prime$-bimodule.
	In this case,
	the functors $M\mathop\otimes\limits_{C^\prime}-$ 
	and $M^\prime\mathop\otimes\limits_C-$ 
	induce mutually inverse equivalences between 
	the $\mathcal{O}$-stable categories 
	$\overline{\mathrm{mod}}(C^\prime)$ and 
	$\overline{\mathrm{mod}}(C)$.
\end{order}

\begin{order}\rm
	Keep the notation above.
	For any $C$-module $M$, we denote $\mathrm{Hom}_\mathcal{O}(M,\mathcal{O})$ by $M^*$.
	It is the dual module of the $C$-module $M$.
	The $\mathcal{O}$-algebra $C$ is called \emph{symmetric over} $\mathcal{O}$ 
	if $C$ is a finitely generated projective $\mathcal{O}$-module and  
	$C\cong C^*$ as $C$-$C$-bimodules
	(see \cite[Definition 2.11.1]{L'book 1}).
	Take two symmetric $\mathcal{O}$-algebras $C$ and $C^\prime$. 
	Let $M_\cdot$ be a bounded complex of finitely generated 
	$C$-$C^\prime$-bimodules
	which are projective as left $C$-modules and right $C^\prime$-modules.
	Set $M_\cdot^*$ to be the dual $\mathrm{Hom}_\mathcal{O}(M_\cdot,\mathcal{O})$
	of the complex $M_\cdot$.
	We say that the complex $M_\cdot$ induces  a \emph{Rickard equivalence between}
	$C^\prime$ and $C$,
	if there is a $C$-$C^\prime$-bimodule $F$ 
	and a $C^\prime$-$C$-bimodule $F^\prime$
	such that $F$ and $F^\prime$ are homotopy equivalent to $0$ as complexes of bimodules,
	and if
	$M_\cdot\mathop\otimes\limits_{C^\prime}M_\cdot^*\cong C\oplus F$ 
	as complexes of $C$-$C$-bimodules
	and
	$M_\cdot^*\mathop\otimes\limits_CM_\cdot\cong C^\prime\oplus F^\prime$
	as complexes of $C^\prime$-$C^\prime$-bimodules
	(see \cite[Definitions 2.2.1]{Rou06}).
	The complex $M_\cdot$ is called a \emph{Rickard complex}.
	It is clear that if $C$ and $C^\prime$ are Rickard equivalent,
	then they are derived equivalent, namely
	their bounded derived categories $\mathcal{D}^b(C)$ and $\mathcal{D}^b(C^\prime)$
	are equivalent to each other as triangulated categories.
\end{order}
\vskip 0.25cm
 
In the following two paragraphs, we will collect some notation from \cite[\S2-\S3]{KL10}
which will be used later in the descriptions of characters of 
the source algebras and the hyperfocal subalgebras in Section $6$.
Throughout the following $2$ paragraphs,
we fix a symmetric $\U$-algebra $C$ such that $\mathcal{K}\mathop\otimes_\mathcal{O}C$ 
is a split semi-simple $\K$-algebra,
which means that  $\mathcal{K}\mathop\otimes_\mathcal{O}C$ is isomorphic to 
a direct product of matrix algebras over $\K$.

\begin{order}\rm
	Denote $\K\mathop\otimes\limits_{\mathcal{O}}C$ and $k\mathop\otimes\limits_{\mathcal{O}}C$ 
	by $\hat{C}$ and $\bar{C}$, respectively.
	We have two Grothendieck groups $R_\mathcal{K}(C)$ and $R_k(C)$ associated 
	with these two algebras $\hat{C}$ and $\bar{C}$, respectively.
	This means that $R_\mathcal{K}(C)$ and $R_k(C)$ are
	the Grothendieck groups of finite dimensional $\hat{C}$-modules and $\bar{C}$-modules, respectively.
	Furthermore, denote by $\mathrm{Pr}_{\mathcal{O}}(C)$ the subgroup of $R_\mathcal{K}(C)$
	generated by the images of modules of the form $\hat{C}\mathfrak{i}$ 
	for some idempotent $\mathfrak{i}$ of $C$.	
	For any finite dimensional module $U$ over $\hat{C}$ or $\bar{C}$,
	we denote by $[U]$ the image of $U$ in $R_\mathcal{K}(C)$ or $R_k(C)$, respectively.
	It is clear that the set of images of simple $\hat{C}$-modules, 
	denoted by $\mathrm{Irr}_{\mathcal{K}}(C)$, is a $\mathbb{Z}$-basis of $R_\mathcal{K}(C)$.
	When $C$ is a group algebra $\U H$ for some finite group $H$,
	we identify ${\rm Irr}_\mathcal{K}(C)$ with ${\rm Irr}(H)$ 
	if no confusion arises throughout this paper.
	Set $\mathrm{k}(C)=|\mathrm{Irr}_{\mathcal{K}}(C)|$.
	Similarly, the set of images of simple $\bar{C}$-modules,
	denoted by $\mathrm{Irr}_k(C)$, is a $\mathbb{Z}$-basis of 	$R_k(C)$.
	Set $l(C)=|\mathrm{Irr}_k(C)|$.
\end{order}		
	
	\begin{order}\rm
		Furthermore, for any finite dimensional $\hat{C}$-module $X$,
		there exists a finitely generated  $C$-module $Y$
		which is free of finite rank as an $\U$-module 
		such that $\mathcal{K}\mathop\otimes\limits_{\mathcal{O}}Y$ is isomorphic to $X$ 
		as $\hat{C}$-modules.
		This induces the so-called \emph{decomposition map} $d_C$
		from $R_\mathcal{K}(C)$ to $R_k(C)$, 
		sending $[X]$ to $[k\mathop\otimes\limits_{\mathcal{O}}Y]$.
		Its kernel is denoted by $L^0(C)$.
		The nonnegative integers occurring in the matrix of $d_C$
		with respect to the canonical bases 
		$\mathrm{Irr}_\mathcal{K}(C)$ and $\mathrm{Irr}_k(C)$
		are called \emph{decomposition numbers}.		
		It is well-known that there is a bilinear form $\langle~,~\rangle_C$ on $R_\mathcal{K}(C)$ such that $\langle[U],[V]\rangle_C=\mathrm{dim}_\mathcal{K}{\rm Hom}_{\hat{C}}(U,V)$
	 for any two finite dimensional $\hat{C}$-modules $U$ and $V$.
	Set $L^0(C)^\bot$ to be the subgroup of $R_\mathcal{K}(C)$
	consisting of elements orthogonal to every elements of $L^0(C)$
	with respect to this bilinear form.
\end{order}
\vskip 0.25cm

Next, we will recall some terminology and results about pointed groups which are due to Puig. 
We refer to \cite{P81,P86,P88} for details.
With these notation, we will give the definitions of hyperfocal subgroups and hyperfocal subalgebras
introduced by Puig.

		\begin{order}\rm
			We fix a finite group $H$ in the remaining paragraphs of this section.
	An $\mathcal{O}$-algebra $\mathbb{B}$ is called an
	$H$-\emph{algebra} if there is a group homomorphism
	$H\longrightarrow\mathrm{Aut}_{\mathcal{O}}(\mathbb{B})$.
	If this group homomorphism is induced by a group homomorphism
	$\theta:H\longrightarrow\mathbb{B}^\times$,
	the $H$-algebra $\mathbb{B}$ is called an	
	$H$-\emph{interior algebra}.
	Let $\mathbb{B}$ be an $H$-interior algebra.
	Clearly, $\mathbb{B}$ becomes an $\U H$-$\U H$-bimodule via this group homomorphism.
	We simply write $\theta(h^{-1})(a)\theta(h)$ as $a^h$ and $\theta(x)a\theta(y)$ as $xay$
	for any $h,x,y\in H$ and any $a\in\mathbb{B}$.	
	Let $\mathbb{B}^\prime$ be another $H$-interior algebra.
	An $\mathcal{O}$-algebra homomorphism 
	$f:\mathbb{B}\longrightarrow\mathbb{B}^\prime$ is called
	an $H$-\emph{interior} \emph{algebra homomorphism} 
	if $f$ preserves the $H$-interior algebra structures 
	on $\mathbb{B}$ and $\mathbb{B}^\prime$. 
	Here, $f$ needs not be unitary. 
	We say it is an \emph{embedding} if $\text{Ker}(f)=\{0\}$ and $\text{Im}(f)=f(1)\mathbb{B}^\prime f(1)$.
	For any subgroup $Z$ of $H$,
	denote by $\mathbb{B}^Z$ the subalgebra of all elements of $\mathbb{B}$ fixed by the $Z$-conjugation.
	The $H$-interior algebra $\mathbb{B}$ is called \emph{primitive} 
	if $\mathbb{B}^H$ is a local $\U$-algebra.
		\end{order}
		
		\begin{order}\rm		
	A \emph{point} $\alpha$ of $Z$ on $\mathbb{B}$ is a
	$(\mathbb{B}^Z)^\times$-conjugacy class of primitive idempotents in	$\mathbb{B}^Z$.
	The pair $(Z,\alpha)$, denoted by $Z_\alpha$,
	is called a \emph{pointed group} of $Z$ on $\mathbb{B}$.
	We call a $Z$-interior algebra $\mathbb{B}_\alpha$ an \emph{embedded algebra} of $Z_\alpha$
	if there is an embedding of $Z$-interior algebras
	$f_\alpha:\mathbb{B}_\alpha\longrightarrow\mathbb{B}$ such that
	$f_\alpha(1_{\mathbb{B}_\alpha})\in\alpha$.
	By \cite[1.6]{P86}, embedded algebras of $Z_\alpha$ are unique up to isomorphism.
	A typical choice of embedded algebras of $Z_\alpha$ is $a\mathbb{B}a$
	and then the embedding $f_\alpha: a\mathbb{B}a\longrightarrow\mathbb{B}$ 
	is just induced by inclusion.
	Here, $a$ is an element of $\alpha$ and
	$\mathbb{B}_\alpha$ is a $Z$-interior algebra with the group homomorphism 
	$Z\rightarrow (\mathbb{B}_\alpha)^\times, z\mapsto za$.
	We denote its image by $Za$.
			For two pointed groups
			$Z_\alpha$ and $Z^\prime_{\alpha^\prime}$ on $\mathbb{B}$,
			we say that
			$Z_\alpha$ is \emph{contained in} $Z^\prime_{\alpha^\prime}$
			if $Z\leq Z^\prime$ and
			there are $e\in\alpha$ and $e^\prime\in\alpha^\prime$
			such that $ee^\prime=e=e^\prime e$.
			It is clear that the subgroup $N_H(Z)$ acts on
			the set
			$\{Z_\alpha~|~\alpha~\mathrm{runs~over~all~points~of~}Z~
			\mathrm{on~}\mathbb{B}\}$
			by conjugation.
			Denote by $N_H(Z_\alpha)$ the stabilizer of the pointed group $Z_\alpha$
			under this action.
		\end{order}
		
		\begin{order}\rm For any subgroup $K$ of $Z$,
			we have the \emph{relative trace map}
			$\mathrm{Tr}_K^Z:\mathbb{B}^K\rightarrow\mathbb{B}^Z,
			a\mapsto\sum\limits_za^z$,
			where $z$ runs over a set of representatives of the right cosets of
			$K$ in $Z$.
			Denote by $\mathbb{B}_K^Z$ the image of the map $\mathrm{Tr}_K^Z$.
			The sum
			$\sum\limits_{K<Z}\mathbb{B}_K^Z+J(\mathcal{O})\mathbb{B}^Z$
			is an ideal of $\mathbb{B}^Z$.
			Denote by $\mathbb{B}(Z)$ the quotient of $\mathbb{B}^Z$
			by this ideal and by
			$\mathrm{Br}_Z^\mathbb{B}$
			the canonical homomorphism $\mathbb{B}^Z\longrightarrow\mathbb{B}(Z)$.
The quotient $\mathbb{B}(Z)$ is called the $\emph{Brauer quotient}$ of $\mathbb{B}$ at $Z$
and the associated homomorphism $\mathrm{Br}_Z^\mathbb{B}$ is called
			the \emph{Brauer homomorphism}. If no confusion arises,
			we denote $\mathrm{Br}_Z^\mathbb{B}$
			by $\mathrm{Br}_Z$.
			By \cite[Lemma (11.7)]{Thevenaz},
			$\mathbb{B}(Z)\neq 0$ only if $Z$ is a $p$-group.
			When $\mathbb{B}$ is the group algebra $\mathcal{O}L$ for some finite group $L$, the inclusion
			$\mathcal{O}C_L(Z)\subseteq(\mathcal{O}L)^Z$
			induces a $k$-algebra isomorphism
			$(\mathcal{O}L)(Z)\cong kC_L(Z)$.
			We identify these two $k$-algebras through the isomorphism.
	More generally,
				there is a similar construction for any $\mathcal{O}Z$-module $M$.
				Set $M^K=\{m\in M\mid x\cdot m=m,\forall x\in M\}$.
				We can define a \emph{relative trace map} $\mathrm{Tr}_K^Z:M^K\longrightarrow M^Z$
				by a similar way and denote its image by $M_K^Z$.
				Similarly, we set $M(Z)$ to be the quotient $\mathcal{O}$-module
				$M^Z/(\sum\limits_{K<Z}M_K^Z+J(\mathcal{O})M^Z)$ and
				also denote the canonical homomorphism $M^Z\rightarrow M(Z)$ by $\mathrm{Br}_Z^M$,
				or $\mathrm{Br}_Z$ for simplicity if no confusion arises.
		\end{order}

\begin{order}\rm
	Recall that a pointed group $R_\gamma$ on $\mathbb{B}$ is \emph{local},
			if $\mathrm{Br}_R(\gamma)\neq 0$.
For any local pointed group $R_\gamma$,
it is clear that there is a unique primitive idempotent $e$ of $Z(\mathbb{B}(R))$
such that $e\cdot\mathrm{Br}_R(\gamma)=\mathrm{Br}_R(\gamma)$.
We call this primitive idempotent $e$ is \emph{associated} with $R_\gamma$.
			A pointed group $R_\gamma$ on $\mathbb{B}$ is
			a \emph{defect pointed group} of a pointed group $Z_\alpha$,
			if $R_\gamma$ is a maximal local pointed group contained in
			$Z_\alpha$.
			Let $R_\gamma$ be a defect pointed group of 
			$K_\beta$ on $\mathbb B$. 
 Then the embedded algebra $\mathbb{B}_\gamma$ of $R_\gamma$ is called
a \emph{source algebra} of the embedded algebra $\mathbb{B}_\beta$ of $K_\beta$ (see \cite{P81}).
	Note that source algebras of the embedded algebra $\mathbb{B}_\beta$ of $K_\beta$
	are unique up to $K$-conjugations and isomorphisms (see \cite[\S 18]{Thevenaz}).			
			\end{order}
		
\begin{order}\rm
Let $R_\gamma$ be a local point on $\mathbb{B}$.
Take an element $i\in\gamma$ and set $\mathbb{B}_\gamma=i\mathbb{B}i$.
Denote the subgroup 
$\{a\in\mathbb{B}_\gamma^\times\mid {^a}(Ri)=Ri\}$ of $\mathbb{B}_\gamma^\times$
by $N_{\mathbb{B}_\gamma}(R)$.
It has three normal subgroups 
$Ri$ and $(\mathbb{B}_\gamma^R)^\times$ and $1+J(\mathbb{B}_\gamma^R)$.
Denote the two quotient groups 
$N_{\mathbb{B}_\gamma}(R)/Ri\cdot(1+J(\mathbb{B}_\gamma^R))$ 
and $N_{\mathbb{B}_\gamma}(R)/Ri\cdot(\mathbb{B}_\gamma^R)^\times$ 
by $\bar{F}_\mathbb{B}(R)$ and $F_\mathbb{B}(R)$, respectively.
It is well-known that $(\mathbb{B}_\gamma^R)^\times$ is isomorphic to $k^\times\times(1+J(\mathbb{B}_\gamma^R))$.
We can get the following short exact sequence
    \begin{equation*}
	    	\xymatrix{
	    		1\ar[r]& k^\times\ar[r] & \bar{F}_\mathbb{B}(R) \ar[r] & F_\mathbb{B}(R)\ar[r] & 1.
	    	}
	    \end{equation*}
When $\mathbb{B}$ is the group algebra $\U L$ for some finite group $L$,
by \cite[Theorem 3.1]{P86}, 
there is a group isomorphism $\mathfrak{f}_{L,R}$
between $N_L(R_\gamma)/RC_L(R)$ and $F_{\mathcal{O}L}(R)$,
which sends $xRC_L(R)\in N_L(R_\gamma)/RC_L(R)$ to $a_x(Ri\cdot(\mathbb{B}_\gamma^R)^\times)\in F_{\mathbb{B}}(R)$ such that
$a_x$ can be chosen to satisfy $^{a_x}(ri)=({^x}r)i$ for any $r\in R$.
\end{order}

		\begin{order}\rm 
			For any element $x$ in $H$ and
			any subgroup $K$ of $H$,
			we denote by $[x,K]$ the subgroup of $H$
			generated by  elements $xux^{-1}u^{-1}$,
			where $u$ runs over $K$.
			Let $\mathfrak{b}$ be a block of $H$.
			Then $H_{\{\mathfrak{b}\}}$ is a pointed group on 
			the $H$-interior algebra $\mathcal{O}H$.
			A defect pointed group of the pointed group $H_{\{\mathfrak{b}\}}$
			is also called a defect pointed group of the block $\mathfrak{b}$.
			Let $Q_\delta$ be a defect pointed group of the block $\mathfrak{b}$ and
			$A$ a source algebra of $\mathcal{O}H\mathfrak{b}$.
			Then $Q$ can act on $A$ through the conjugation.
			Recall that the \emph{hyperfocal subgroup} $\mathfrak{hyp}(Q_\delta)$ of
			$Q_\delta$ is the subgroup of $Q$ generated by the subgroups $[x,R]$ 
			where $R_\epsilon$ runs over the set of local pointed groups on
			$\mathcal{O}H$ such that $R_\epsilon\subseteq Q_\delta$,
			and $x$ over the set of $p^\prime$-elements of $N_H(R_\epsilon)$
			(see \cite[1.7.2]{P00}).
			Sometimes,
			we also call $\mathfrak{hyp}(Q_\delta)$ a hyperfocal subgroup of $Q$
			for simplicity.
			It is easy to see that $\mathfrak{hyp}(Q_\delta)$ is a normal subgroup of $Q$.
			Denote by $\mathcal{S}$ a set of representatives for $Q/\mathfrak{hyp}(Q_\delta)$ in $Q$.
			By \cite[Theorem 1.8]{P00},
			there is  a $Q$-stable $\mathcal{O}$-subalgebra $\mathbb{A}$
			of $A$, unique up to $(A^Q)^\times$-conjugation,
			containing the image of $\mathfrak{hyp}(Q_\delta)$  and fulfilling 
			$A=\mathop\bigoplus\limits_{v\in\mathcal{S}}\mathbb{A}v=
			\mathop\bigoplus\limits_{v\in\mathcal{S}}v\mathbb{A}$.
			It is clear that $\mathbb{A}$ is an $\mathfrak{hyp}(Q_\delta)$-interior algebra
			inherited from the $Q$-interior algebra structure of $A$.
			The $\mathfrak{hyp}(Q_\delta)$-interior algebra $\mathbb{A}$ is called
			a \emph{hyperfocal subalgebra} of $A$.
			Sometimes,
			we also call it a hyperfocal subalgebra of the block $\mathfrak{b}$.
		\end{order}
		\vskip 0.25cm
		
	In the following paragraph, 	
		we will collect some notation and definitions about endopermutation modules 
		over a $p$-group which is introduced by Dade \cite{D78a, D78b}.
		
		\begin{order}\rm
			Let $P$ be a $p$-group.	
			An $\mathcal{O}P$-module $V$ is called an \emph{endopermutation} 
			$\mathcal{O}P$-module if the endomorphism algebra $\mathrm{End}_\mathcal{O}(V)$ has a $P$-stable $\mathcal{O}$-basis under the $P$-conjugation action.
			Let $V$ be an indecomposable endopermutation $\U P$-module with vertex $P$,
			which means that $S(P)\neq 0$. 
			Here, $S=\mathrm{End}_\mathcal{O}(V)$,
			which is called a \emph{Dade} $P$-\emph{interior algebra}.
			Then by \cite[Proposition 7.3.7]{L'book 2},
			for any subgroup $Q$ of $P$, 
			there is up to isomorphism a unique endopermutation $k(N_P(Q)/Q)$-module $W_Q$
			such that $S(Q)\cong\mathrm{End}_k(W_Q)$ as $N_P(Q)/Q$-algebras.
			It is clear that $W_Q$ is an indecomposable endopermutation $k(N_P(Q)/Q)$-module.
			Moreover, let $R$ be a subgroup of $P$ such that $Q$ is normal in $R$.
			We denote by $\emph{Defres}_{R/Q}^P(V)$ the endopermutation $k(R/Q)$-module
			$\mathrm{Res}_{R/Q}^{N_P(Q)/Q}(W_Q)$.
			In particular, when $Q$ is normal in $P$,
			$\mathrm{Defres}_{P/Q}^P(V)$ is just the indecomposable endopermutation 
			$k(P/Q)$-module $W_Q$.
		\end{order}

	\begin{order}\rm
	Let us recall the definitions of twisted group algebra and
	relative separability.
	Consider the trivial action of $H$ on $k^\times$.
	We denote by $Z^2(H; k^\times)$ and $H^2(H; k^\times)$ 
	the set of all $2$-cocycles of $H$ with coefficients in $k^\times$ 
	and the second cohomology group of $H$ with coefficients in $k^\times$, respectively.
		Let $\alpha\in Z^2(H; k^\times)$ be a $2$-cocycle of $H$.
		Since $k$ is perfect, we have a canonical isomorphism $\U^\times\cong k^\times\times(1+J(\U))$ (see \cite{Serre's book}).
		So we can view $k^\times$ as a subgroup of $\U^\times$.
		The twisted group algebra of $H$ by $\alpha$, denoted by $\U_\alpha H$,
		is the $\U$-algebra
		with an $\U$-basis consisting of all elements of $H$ and the multiplication induced by 
		the equality $\lambda_x x\cdot\lambda_y y=(\lambda_x\lambda_y\alpha(x,y))xy$  
		for all $x,y\in H$ and all $\lambda_x,\lambda_y\in\U$.
		Let $\mathbb{B}$ and $\mathbb{B}^\prime$ be two $\U$-algebras.
		Let $\varphi$ be an $\U$-algebra homomorphism from $\mathbb{B}^\prime$ to $\mathbb{B}$.
		Through this homomorphism $\varphi$,
		we can view $\mathbb{B}$ as a  $\mathbb{B}^\prime$-$\mathbb{B}^\prime$-bimodule.
	The $\U$-algebra $\mathbb{B}$ is called 
		 \emph{relatively} $\mathbb{B}^\prime$-\emph{separable} if $\mathbb{B}$ 
		is isomorphic to a direct summand of $\mathbb{B}\mathop\otimes\limits_{\mathbb{B}^\prime}\mathbb{B}$ as $\mathbb{B}$-$\mathbb{B}$-bimodules (see \cite[\S 4.8]{L'book 1}).
		This is equivalent to requiring that the canonical 
		$\mathbb{B}$-$\mathbb{B}$-bimodule	homomorphism 
		$\mathbb{B}\mathop\otimes\limits_{\mathbb{B}^\prime}\mathbb{B}\longrightarrow\mathbb{B}$
		induced by multiplication in $\mathbb{B}$ splits by \cite[Theorem 2.6.10]{L'book 1}.
		In the case where $\mathbb{B}^\prime$ is a subalgebra of $\mathbb{B}$
	the homomorphism $\varphi$ is always taken to be the inclusion map unless otherwise specified.
	\end{order}
		
		\vskip2cm
		\section{Hyperfocal abelian Frobenius blocks}
		\quad\, 
		In this section, we will give the definition of the hyperfocal abelian Frobenius blocks
and investigate their local structures.
		
		\begin{order}\rm
			Let $G$ be a finite group and $b$ a block of $G$ over $\mathcal{O}$
			with  defect pointed group $P_\gamma$.
We denote the hyperfocal subgroup of $P_\gamma$ by $D$.
Then by \cite[Propositions 4.2 and 3.3]{P00},
there is a unique local point $\delta$ of $D$ on $\U Gb$ such that $D_\delta\leq P_\gamma$.
Take an element $i$ of $\gamma$ and set $A=i\U Gi$.
Then $A$ is a source algebra of $\U Gb$ and let  $\mathbb{A}$ be a hyperfocal subalgebra of $A$.
We take an element $j$ of $\delta$ such that $ij=j=ji$.
By \cite[Proposition 2.8]{P00}, we can assume that $j$ belongs to $\mathbb{A}$.
We can view $D_\delta$ as a local pointed group of $D$ on both $A$ and $\mathbb{A}$.
Set $A_\delta=jAj$ and $\mathbb{A}_\delta=j\mathbb{A}j$.
Clearly, $A_\delta$ and $\mathbb{A}_\delta$ are embedded algebras of
the local pointed group $D_\delta$ on $A$ and $\mathbb{A}$, respectively.
By the uniqueness of $D_\delta$,
there is an element $a_u$ in $(\mathbb{A}^D)^\times$ such that
$a_uja_u^{-1}=uju^{-1}$ for any $u\in P$.
Denote $a_u^{-1}uj$ by $a_{u,\delta}$ for any $u\in P$.
Then $a_{u,\delta}$ belongs to $A_\delta^\times$ fulfilling  
$a_{u,\delta}(vj)a_{u,\delta}^{-1}=(uvu^{-1})j$
for any $v\in D$ and $j\mathbb{A}uj=\mathbb{A}_\delta a_{u,\delta}$.
Moreover, we have
$A_\delta=\mathop\bigoplus\limits_{uD\in P/D}\mathbb{A}_\delta a_{u,\delta}$
as $\U$-modules.
\end{order}

\begin{order}\rm
It is well-known that the source algebra $A$ is closely related to the
block algebra $\U Gb$.
For example, they are both relatively $\U P$-separable $P$-interior algebras. 
There is a canonical Morita equivalence between 
the source algebra $A$ and the block algebra $\U Gb$ induced by the $A$-$\U Gb$-bimodule $i\U G$.
Next, we will show that these properties also hold for 
the hyperfocal subalgebra $\mathbb{A}$ and its embedded subalgebra $\mathbb{A}_\delta$ by the following general result.
\end{order}
		
	\begin{lem}\label{a general lemma}
Let $Q$ be a $p$-subgroup and $C$ an indecomposable relatively $\U Q$-separable $Q$-interior algebra.
Assume that $C$ has a unique local pointed group $Q_\epsilon$ and 
as $\U(Q\times Q)$-modules, every indecomposable direct summand of $C$ has the form
$\mathrm{Ind}_{\Delta_\varphi(Z)}^{Q\times Q}(\U)$ for some subgroup $Z$ of $Q$ and 
some injective homomorphism $\varphi: Z\longrightarrow Q$.
Here, $\Delta_\varphi(Z)=\{(u,\varphi(u))\mid u\in Z\}$ is a subgroup of $Q\times Q$. 
Moreover, suppose that
$\frac{\mathrm{rank}_\mathcal{O}(C)}{|Q|}$ is coprime to $p$.
Take an element $l$ of $\epsilon$.
Then the $Q$-interior algebra $lCl$ is a relatively  $\U Q$-separable $Q$-interior algebra and 
the $lCl$-$C$-bimodule $lC$ induces a Morita equivalence between $lCl$ and $C$.
		\end{lem}
		
		\begin{proof}
Since every indecomposable direct summand of  the $\U(Q\times Q)$-module $C$ has the form
$\mathrm{Ind}_{\Delta_\varphi(Z)}^{Q\times Q}(\U)$,
$C$ is a direct sum of some copies of $\U Q$ as left $\U Q$ and right $\U Q$-modules.
In particular, $C$ is free as left and right $\U Q$-modules.
Hence, the structural map $Q\longrightarrow C^\times$ is injective and 
we can view $\U Q$ as a subalgebra of $C$.
Obviously, every finitely generated projective $C$-module has $\U$-rank divided by $|Q|$.

By the definition of relatively $\U Q$-separable $Q$-interior algebras,
$C$ is a direct summand of $C\mathop\otimes\limits_{\mathcal{O}Q}C$ as $C$-$C$-bimodules.
Since $C$ is an indecomposable algebra, 
there are some primitive idempotents $e, f$ of $C^Q$ such that 
$C$ is a direct summand of $Ce\mathop\otimes\limits_{\mathcal{O}Q}fC$ as $C$-$C$-bimodules.
Suppose that $\mathrm{Br}_Q(e)=0$.
By \cite[Theorem (23.1)]{Thevenaz},
there is a proper subgroup $Q_0$ of $Q$ and a primitive idempotent $e_0$ of $C^{Q_0}$ such that
$e=\mathrm{Tr}_{Q_0}^Q(e_0)$ and ${^u}e_0\cdot e_0=0$ for any $u\in Q-Q_0$.
Then the map sending $c\in Ce$ to $\sum\limits_{uQ_0\in Q/Q_0}cue_0\mathop\otimes\limits_{\mathcal{O}Q_0}u^{-1}$
induces a $C$-$\U Q$-bimodule isomorphism from $Ce$ to 
$Ce_0\mathop\otimes\limits_{\mathcal{O}Q_0}\U Q$.
Therefore, $C$ is a direct summand of $Ce_0\mathop\otimes\limits_{\mathcal{O}Q_0}fC$ as $C$-$C$-bimodules.
Denote by $\bar{C}$ the $k$-algebra $C/J(\U)C$.
Then for any simple $C$-module $W$, 
we have $W$ isomorphic to a direct summand of $\bar{C}\mathop\otimes\limits_{kQ_0}U$ 
for some indecomposable $kQ_0$-module $U$.
By the assumption and Mackey's formula, 
$\mathrm{Res}_{Q\times Q_0}^{Q\times Q}(\bar{C})$ is a direct sum of submodules isomorphic to
${\rm Ind}_{\Delta_\psi(Z)}^{Q\times Q_0}(k)$ for some proper subgroup $Z$ of $Q$ and 
injective homomorphism $\psi: Z\longrightarrow Q_0$.
By Green's indecomposability theorem,
this implies that the dimension of every simple $C$-module has dimension divided by $p$.
Let $P_W$ be the projective cover of the simple $C$-module $W$.
We have  
$\mathrm{dim}(\bar{C})=\sum\limits_W\mathrm{dim}(P_W)\mathrm{dim}(W)$,
which is divided by $p|Q|$.
This contradicts the assumption that $\frac{\mathrm{rank}_\mathcal{O}(C)}{|Q|}$ is coprime to $p$.

We conclude that $\mathrm{Br}_Q(e)\neq 0$.
Similarly, we can also get $\mathrm{Br}_Q(f)\neq 0$.
Since $\epsilon$ is the unique local point of $Q$ on $C$,
we have $e$ and $f$ both belonging to $\epsilon$.
In particular, we have $Ce$ is isomorphic to $Cl$ as $C$-$\U Q$-bimodules 
and $fC$ is isomorphic to $lC$ as $\U Q$-$C$-bimodules.
So $C$ is a direct summand of $Cl\mathop\otimes\limits_{\mathcal{O}Q}lC$ as $C$-$C$-bimodules and
then $lCl$ is a direct summand of $lCl\mathop\otimes\limits_{\mathcal{O}Q}lCl$ as $lCl$-$lCl$-bimodules.
This means that $lCl$ is a relatively $\U Q$-separable $Q$-interior algebra.

Since $C$ is a direct summand of $Cl\mathop\otimes\limits_{\mathcal{O}Q}lC$,
$W$ is a direct summand of $Cl\mathop\otimes\limits_{\mathcal{O}Q}l\cdot W$ for any simple $C$-module $W$.
In particular, $l\cdot W$ is not equal to $0$ for any simple $C$-module $W$.
By \cite[Theorem 2.8.7]{L'book 1},
the $lCl$-$C$-bimodule $lC$ induces a Morita equivalence between $lCl$ and  $C$.
		\end{proof}	

\begin{cor}\label{separable and Morita equi}
The embedded subalgebra $\mathbb{A}_\delta$ of $\mathbb{A}$ is a 
relatively $\U D$-separable $D$-interior algebra and
the $\mathbb{A}_\delta$-$\mathbb{A}$-bimodule $j\mathbb{A}$ induces 
a Morita equivalence between $\mathbb{A}_\delta$ and $\mathbb{A}$.
\end{cor}

\begin{proof}
By \cite[Theorem 1]{W09}, 
the hyperfocal subalgebra $\mathbb{A}$ is a relatively $\U D$-separable $D$-interior algebra.
By \cite[Lemma 3.2]{HZZ24},  $\mathbb{A}$ is an indecomposable $\U$-algebra.
From the paragraph 3.1, 
we can get that $D$ has a unique local point on $\mathbb{A}$.
Since $A=\mathop\bigoplus\limits_{uD\in P/D}\mathbb{A}u$,
we have 
$\frac{\mathrm{rank}_\mathcal{O}(\mathbb{A})}{|D|}=\frac{\mathrm{rank}_\mathcal{O}(A)}{|P|}$
which is coprime to $p$ by \cite[Corollary (44.8)]{Thevenaz}.
Since $\mathbb{A}$ is a direct summand of $A$ as $\U(D\times D)$-modules,
the statements in the corollary follow from 
\cite[Theorem 8.7.1]{L'book 2} and Lemma \ref{a general lemma} above.
\end{proof}

\begin{order}\rm
Let us consider the quotient group $N_G(D_\delta)/DC_G(D)$.
Generally, $N_G(D_\delta)/DC_G(D)$ may have a nontrivial Sylow $p$-subgroup $PC_G(D)/DC_G(D)$.
This property is different from the fact that the inertial quotient 
$N_G(P_\gamma)/PC_G(P)$ is a $p^\prime$-group.
However, we can show that it is a $p$-nilpotent group, 
which means that it has a normal $p$-complement.
We call this normal $p$-complement a \emph{hyperfocal inertial quotient} of the block $b$
and denote it by $E_\mathfrak{h}$ throughout this paper.
Hence, $E_\mathfrak{h}$ is a $p^\prime$-group and 
$N_G(D_\delta)/DC_G(D)=PC_G(D)/DC_G(D)\cdot E_\mathfrak{h}$.
To show this result, we need the following lemma  
about the structures of the Brauer quotients 
$A_\delta(D)$ and $\mathbb{A}_\delta(D)$ of $A_\delta$ and $\mathbb{A}_\delta$ at $D$, respectively.
\end{order}

\begin{lem}\label{Brauer quot in D}
We have $A_\delta(D)\cong kC_P(D)$ as $Z(D)$-interior algebras
such that its restriction to $\mathbb{A}_\delta(D)$ gives a $Z(D)$-interior algebra isomorphism
$\mathbb{A}_\delta(D)\cong kZ(D)$.
In particular, $(j\mathbb{A}uj)(D)\cong kZ(D)u$ when $u$ is in $C_P(D)$ and 
$(j\mathbb{A}uj)(D)=0$ when $u$ is not in $DC_P(D)$.
\end{lem}

\begin{proof}
Let $\bar{b}_D$ be the block of  $C_G(D)$ over $k$ associated with the local pointed group $D_\delta$.
It is clear that $\bar{b}_D$ is $P$-stable and nilpotent by \cite[Proposition 3.3]{P00}.
Then $\bar{b}_D$ can be viewed as a block of $PC_G(D)$
and we  denote it by $\bar{c}_D$ when it is viewed as a block of $PC_G(D)$.
So the block $\bar{c}_D$ is also nilpotent by \cite[Proposition 6.5]{KP90}.
It is easy to check that the block $\bar{c}_D$ has a defect group $P$ and 
then $C_P(D)$ is a defect group of the block $\bar{b}_D$.

By \cite[Lemma (40.2)]{Thevenaz},
we have $\mathrm{Br}_D(i)$ belongs to a local point of $P$ on $kC_G(D)$.
In particular,
$\mathrm{Br}_D(i)kPC_G(D)\mathrm{Br}_D(i)$ is a source algebra of the block $\bar{c}_D$.
Note that $A(D)=\mathrm{Br}_D(i)kC_G(D)\mathrm{Br}_D(i)$ is a primitive $P$-algebra.
Then by \cite[Theorem 8.12.3]{L'book 2},
there is a primitive Dade $P$-interior algebra $T_D$ over $k$ such that
we have an isomorphism of $P$-interior algebras
\begin{align*}
\mathrm{Br}_D(i)kPC_G(D)\mathrm{Br}_D(i)\cong T_D\mathop\otimes\limits_kkP,
\end{align*}
which restricts to a $C_P(D)$-interior algebra isomorphism
\begin{align}\label{A local  at D}
A(D)\cong T_D\mathop\otimes\limits_kkC_P(D).
\end{align}
Since $D$ is normal in $PC_G(D)$,
we have $v\cdot 1_{T_D}=1_{T_D}$ for any $v\in D$.
By the isomorphism (\ref{A local  at D}),
we have a $Z(D)$-interior algebra isomorphism $A_\delta(D)\cong kC_P(D)$.
Obviously, the image of $\mathbb{A}_\delta(D)$ under this isomorphism contains $kZ(D)$.
In particular, the dimension ${\rm dim}_k(\mathbb{A}_\delta(D))$ of the $k$-algebra $\mathbb{A}_\delta(D)$ is greater than or equal to $|Z(D)|$.

Borrowing the notation from the paragraph 3.1, for any $u\in P$,
there is an element $a_{u,\delta}$ in $A_\delta^\times$ 
such that 
${^{a_{u,\delta}}}(v\cdot j)=({^u}v)\cdot j$ for any $v\in D$
and $A_\delta=\mathop\bigoplus\limits_{uD\in P/D}\mathbb{A}_\delta a_{u,\delta}$.
It is clear that $\mathbb{A}_\delta a_{u,\delta}$ is a $D$-stable $\U$-module
for any $u\in P$
and then $A_\delta(D)=\mathop\bigoplus\limits_{uD\in P/D}(\mathbb{A}_\delta a_{u,\delta})(D)$.
When $u$ is in $C_P(D)$,
$a_{u,\delta}$ is in $(A_\delta^D)^\times$ and then
$(\mathbb{A}_\delta a_{u,\delta})(D)=\mathbb{A}_\delta(D)\mathrm{Br}_D(a_{u,\delta})$.
In particular, they have the same dimension.
Therefore, the $k$-algebra $A_\delta(D)=\mathop\bigoplus\limits_{uZ(D)\in C_P(D)/Z(D)}(\mathbb{A}_\delta a_{u,\delta})(D)$
has dimension greater than or equal to $|C_P(D)|$.
Then we are done.
\end{proof}
		
Recall from the paragraph 2.12 that
there is a group isomorphism $\mathfrak{f}_{G,D}$ 
from $N_G(D_\delta)/DC_G(D)$ and $F_{\mathcal{O}G}(D)$.
It is clear that $j\U Gj=jAj=A_\delta$.
Then $F_{\mathcal{O}G}(D)=F_{A}(D)$ and
$\mathfrak{f}_{G,D}$ is a group isomorphism from $N_G(D_\delta)/DC_G(D)$ to $F_A(D)$.
Now we can show that the quotient group $N_G(D_\delta)/DC_G(D)$ is a $p$-nilpotent group.

\begin{prop}\label{p-nilpotent quot}
The quotient group $N_G(D_\delta)/DC_G(D)$ is a $p$-nilpotent group.
Moreover, the hyperfocal inertial quotient $E_\mathfrak{h}$ of the block $b$
is the preimage of the subgroup $N_{\mathbb{A}_\delta}(D)\cdot(A_\delta^D)^\times/D(A_\delta^D)^\times$
of $F_A(D)$ under the isomorphism $\mathfrak{f}_{G,D}$.
\end{prop}

\begin{proof}
Set $P(A_\delta^D)^\times=(A_\delta^D)^\times\cdot\{a_{u,\delta}\mid u\in P\}$.
It is easy to check that
 $P(A_\delta^D)^\times/D(A_\delta^D)^\times$ is just the image of $PC_G(D)/DC_G(D)$,
which is a Sylow $p$-subgroup of $N_G(D_\delta)/DC_G(D)$, under the isomorphism $\mathfrak{f}_{G,D}$.
Hence, $P(A_\delta^D)^\times/D(A_\delta^D)^\times$ is a Sylow $p$-subgroup of $F_A(D)$.

	On the other hand, by \cite[Proposition 2.8]{P00}, 
	we have 
	$$F_A(D)=((P(A_\delta^D)^\times)\cdot N_{\mathbb{A}_\delta}(D))/D(A_\delta^D)^\times$$
	and 
	$N_{\mathbb{A}_\delta}(D)(A_\delta^D)^\times/D(A_\delta^D)^\times$ 
	is normal in $F_A(D)$.
	Take an element $\mathfrak{a}_\delta$ of 
	$P(A_\delta^D)^\times\cap N_{\mathbb{A}_\delta}(D)$.
	So there is some $u\in P$ such that ${^{\mathfrak{a}_\delta}}(vj)=({^u}v)j$ for any $v\in D$.
	Then $\mathfrak{a}_\delta^{-1}u$ is in $(j\mathbb{A}ju)^D$	
	and ${\rm Br}_D(\mathfrak{a}_\delta^{-1}u)\neq 0$
	since $\mathfrak{a}_\delta^{-1}u\cdot u^{-1}\mathfrak{a}_\delta=j$.
	By the paragraph 3.1, 
	there is some $a_u\in(\mathbb{A}^D)^\times$ such that
	$a_u^{-1}ja_u=u^{-1}ju$.
	Therefore, 
	$j\mathbb{A}ju=j\mathbb{A}ua_u^{-1}ja_u=j\mathbb{A}uja_u$.
	Combining these facts together, 
	we have $(j\mathbb{A}uj)(D)\neq 0$.
	By Lemma \ref{Brauer quot in D}, this forces $u\in DC_P(D)$.
	Then the element $\mathfrak{a}_\delta$ has to belong to $D(\mathbb{A}_\delta^D)^\times$.
	Hence,
	$(P(A_\delta^D)^\times/D(A_\delta^D)^\times)\cap (N_{\mathbb{A}_\delta}(D)(A_\delta^D)^\times/D(A_\delta^D)^\times)=1$.
	In particular, $F_A(D)$ is a $p$-nilpotent group.
	So is $N_G(D_\delta)/DC_G(D)$.
	The last statement can be easily obtained.
	 	\end{proof}
	
\begin{order}\rm
Now we can define the hyperfocal abelian Frobenius block.
Keep the notation above.
By the Schur-Zassenhaus Theorem, we can identify $E_\mathfrak{h}$ with 
a subgroup of $\mathrm{Aut}(D)$ and then it can act on $D$ by the conjugation action.
Then the block $b$ is called with \emph{Frobenius hyperfocal inertial quotient}
if $D$ and $E_\mathfrak{h}$ are both nontrivial and 
$E_\mathfrak{h}$ acts freely on $D-\{1\}$,
equivalently, $D$ is nontrivial and the corresponding semidirect product 
$D\rtimes E_\mathfrak{h}$ is a Frobenius group.
Moreover, a block $b$ with Frobenius hyperfocal inertial quotient is called a 
\emph{hyperfocal abelian Frobenius block} if a hyperfocal subgroup of the block $b$ is abelian.
\end{order}

\begin{rmk}\label{examples of hyp frob}\rm
There are some examples of hyperfocal abelian Frobenius blocks.
The first well-known example is the block with a cyclic hyperfocal subgroup.
The second example is the block with a Klein four hyperfocal subgroup.
In this case, by \cite[Proposition 2.3]{HZ19},
$N_G(D_\delta)/C_G(D)$ is isomorphic to a cyclic group of order $3$, 
or the symmetric group $\mathrm{S}_3$  of degree 3,
depending on $C_P(D)=P$ or not.
In both cases, they are  both hyperfocal abelian Frobenius blocks. 
\end{rmk}
\vskip 0.25cm

In the following, we investigate the local structures
of hyperfocal abelian Frobenius blocks.

\begin{order}\rm
For any element $x$ of $\U G$,
we use $\bar{x}$ to denote the image of $x$ in $kG$ 
under the natural surjective homomorphism.
There is an analogous notation for the subset of $\U G$.
First we fix some notation about $b$-Brauer pairs.
We refer to \cite[\S 40]{Thevenaz} for details.
Denote by $b_P$ the block of $C_G(P)$ over $\U$ such that 
the block $\bar{b}_P$ of $kC_G(P)$ is associated with the local pointed group $P_\gamma$.
Then the pair $(P, \bar{b}_P)$ is a maximal $b$-Brauer pair.
For any subgroup $R$ of $P$,
let  $b_R$ be the unique block of $C_G(R)$ over $\U$ such that $(R,\bar{b}_R)\leq(P,\bar{b}_P)$.
Denote by $\mathcal{B}_G(b)_{\leq(P,\bar{b}_P)}$ the Brauer category of the block $b$
whose objects consist of all $b$-Brauer pairs contained in $(P,\bar{b}_P)$ 
and morphisms from $(Q,\bar{b}_Q)$ to $(R,\bar{b}_R)$ are group homomorphisms
from $Q$ to $R$ induced by the conjugation action of an element $g$ in $G$
such that ${^g}(Q,\bar{b}_Q)\leq(R,\bar{b}_R)$.
Let $H$ be a subgroup of $G$.
$H$ is said to \emph{control fusion} in $\mathcal{B}_G(b)_{\leq(P,\bar{b}_P)}$,
or \emph{control fusion} of the block $b$ for simplicity, 
if the morphisms of the Brauer category $\mathcal{B}_G(b)_{\leq(P,\bar{b}_P)}$
can be induced by elements in $H$ (see \cite[\S 49]{Thevenaz}).
\end{order}

\begin{order}\rm
Let us recall the definitions of some special local pointed groups (see \cite[1.6]{P00}).
A \emph{self-centralizing} pointed group $Q_\epsilon$ on $\U G$ 
is a local pointed group on $\U G$ such that $Z(Q)$ is a defect group of 
the block $\bar{b}(\epsilon)$ of $C_G(Q)$ associated with $Q_\epsilon$.
Furthermore, a self-centralizing pointed group $Q_\epsilon$ is called an \emph{essential} pointed group
if the quotient $N_G(Q_\epsilon)/QC_G(Q)$ contains a proper subgroup $M$ such that
$p$ divides $|M|$ but does not divide 
$|M\cap M^x|$ for any $x\in N_G(Q_\epsilon)/QC_G(Q)-M$.
In particular, $N_G(Q_\epsilon)/QC_G(Q)$ has no nontrivial normal $p$-subgroup.
For an essential pointed group $Q_\epsilon$,
it is clear that $N_G(Q_\epsilon)=N_G(Q,\bar{b}(\epsilon))$ and
we also call the corresponding $b$-Brauer pair $(Q,\bar{b}(\epsilon))$ an
essential $b$-Brauer pair.
\end{order}

\begin{prop}\label{nilpotent block of D at local}
	Suppose that the block $b$ is a hyperfocal abelian Frobenius block.
	Then for any nontrivial subgroup $Q$ of $D$,
	the block $b_Q$ of $C_G(Q)$ over $\U$ is nilpotent.
\end{prop}

\begin{proof}
Let $Q$ be a nontrivial subgroup of $D$.
Suppose that $C_P(Q)$ is a defect group of the block $b_Q$ of $C_G(Q)$.
We denote $C_P(Q)$ by $Q_c$.
We first prove this proposition for this special case.
For any subgroup $R$ of $C_P(Q)$,
$(R,\bar{b}_{QR})$ is a  $b_Q$-Brauer pair.
In particular $(Q_c,\bar{b}_{Q_c})$ is a maximal $b_Q$-Brauer pair.
In order to show that the block $b_Q$ is nilpotent,
it suffices to prove that the hyperfocal subgroup of the block $b_Q$ is trivial.
By \cite[1.7.1]{P00}, it reduces to show that 
$N_{C_G(Q)}(R,\bar{b}_{QR})/C_G(QR)$ is a $p$-group
when $(R,\bar{b}_{QR})$ is either an essential $b_Q$-Brauer pair or
$(R,\bar{b}_{QR})=(Q_c,\bar{b}_{Q_c})$.

It is clear that $N_{C_G(Q)}(R,\bar{b}_{QR})=N_G(QR,\bar{b}_{QR})\cap C_G(Q)$.
By \cite[Theorem 2]{W14},
$N_G(QR,\bar{b}_{QR})=(N_G(D,\bar{b}_D)\cap N_G(QR,\bar{b}_{QR}))C_G(QR)$.
Therefore,
$$N_{C_G(Q)}(R,\bar{b}_{QR})=C_G(QR)
(N_G(D,\bar{b}_D)\cap N_G(QR,\bar{b}_{QR})\cap C_G(Q))$$
and then
$$N_{C_G(Q)}(R,\bar{b}_{QR})/C_G(QR)\cong
(N_G(D,\bar{b}_D)\cap N_G(QR,\bar{b}_{QR})\cap C_G(Q))
/(N_G(D,\bar{b}_D)\cap C_G(QR)).$$
Obviously, $N_G(D,\bar{b}_D)/C_G(D)$ is isomorphic to $N_G(D_\delta)/C_G(D)$.
Since $E_\mathfrak{h}$ acts freely on $D-\{1\}$ and $D$ is abelian,
we have $N_G(D,\bar{b}_D)\cap C_G(Q)=Q_cC_G(D)$.

Suppose that $D\leq R$. 
Then 
$$C_G(QR)=C_G(R)~\text{and}~N_{C_G(Q)}(R,\bar{b}_R)=N_{Q_cC_G(D)}(R,\bar{b}_R)~\text{and}~
N_G(D,\bar{b}_D)\cap C_G(R)=C_G(R).$$
In this case, the pair $(R,\bar{b}_R)$ is also a $b_D$-Brauer pair 
when $b_D$ is viewed as a block of $Q_cC_G(D)$ over $\U$.
Since by \cite[Proposition 3.3]{P00} the block $b_D$ of $Q_cC_G(D)$ is 
nilpotent with defect group $Q_c$,
we have 
$$N_{C_G(Q)}(R,\bar{b}_R)=N_{Q_cC_G(D)}(R,\bar{b}_R)=N_{Q_c}(R)C_G(R).$$
Hence,
$$N_{C_G(Q)}(R,\bar{b}_R)/C_G(R)\cong
N_{Q_c}(R)C_G(R)/C_G(R)$$
is a $p$-group.

Now let $(R,\bar{b}_{QR})$ be an essential $b_Q$-Brauer pair.
Clearly, $(R,\bar{b}_R)\leq(QR,\bar{b}_{QR})$ and $R\unlhd QR$.
We have $\mathrm{Br}_{QR}(\bar{b}_R)\bar{b}_{QR}=\bar{b}_{QR}$.
Since $Z(R)$ is the defect group of the block $b_R$ of $C_G(R)$,
$R$ is the defect group of the block $b_R$ viewed as a block of $RC_G(R)$.
Note that $QR\leq RC_G(R)$. 
We have $QR\leq R$. 
In particular, $Q\leq R$ and then $C_G(QR)=C_G(R)$.
On the other hand, by the argument above,
we have $N_D(R)C_G(R)/C_G(R)$ is a normal $p$-subgroup of 
$N_G(R,\bar{b}_R)/C_G(R)$.
Therefore, $N_D(R)$ is contained in $RC_G(R)$.
Set $T=RN_D(R)$ and so $R\unlhd T$.
We have $\mathrm{Br}_T(\bar{b}_R)\bar{b}_T=\bar{b}_T$.
Similarly, we can get $T\leq R$ which forces $D\leq R$.

In conclusion, when $(R,\bar{b}_{QR})$ is either essential or equal to $(Q_c,\bar{b}_{Q_c})$,
we both have $D\leq R$.
We are done by the fourth paragraph under the assumption that
$C_P(Q)$ is a defect group of the block $b_Q$.

In general, it is well-known that there is some element $x$ of $G$
such that $(Q,\bar{b}_Q)\leq(P,\bar{b}_P)^x$ and 
$C_{P^x}(Q)$ is a defect group of the block $b_Q$.	
Since the hyperfocal subgroup $D$ is abelian,	
we can assume that $x$ is in $N_G(D,\bar{b}_D)$	by \cite[Theorem 2]{W14}.
Therefore, $C_P({^x}Q)$ is a defect group of the block ${^x}b_Q$ of $C_G({^x}Q)$.
Since $x$ belongs to $N_G(D,\bar{b}_D)$,
${^x}Q$ is still a nontrivial subgroup of $D$ and 
$({^x}Q,{^x}\bar{b}_Q)$ is contained in $(D,\bar{b}_D)$.
By the uniquness of $\bar{b}_{{^x}Q}$,
we get ${^x}b_Q=b_{{^x}Q}$.
On the other hand,
by the argument above for the special case,
the block $b_{{^x}Q}$ of $C_G({^x}Q)={^x}C_G(Q)$ is nilpotent.
Obviously, the conjugation action of $x$ induces an isomorphism
between the Brauer categories of the block $b_Q$ and $b_{{^x}Q}$.
Hence, the block $b_Q$ of $C_G(Q)$ is also nilpotent.
\end{proof}
	\vskip2cm

\section{The bimodule structure and  the Brauer quotients of $\mathbb{A}_\delta$}
		\quad\, 
Keep the notation in the last section.	
In this sectoin, we will first describe the $\U(D\times D)$-module structure of $\mathbb{A}_\delta$
without assuming that the block $b$ is a hyperfocal abelian Frobenius block.
This is analogous to the $\U(P\times P)$-module structure of the source algebra $A$ 
(see \cite[Theorem (44.3)]{Thevenaz}).
Let us collect some basic results about the $\U(D\times D)$-module structure of $\mathbb{A}_\delta$.



\begin{lem}\label{basic result of mod stru}
Every indecomposable direct summand of the $\U(D\times D)$-module $\mathbb{A}_\delta$  
has the form $\U DgD$ for some $g\in G$, 
which is isomorphic to $\mathrm{Ind}_{\Delta_g(D)}^{D\times D}(\U)$.
Here, $\Delta_g(D)=\{(u,g^{-1}ug)\in D\times D\mid u\in D\cap{^g}D\}$ is a subgroup of $D\times D$.
Moreover, setting  $Q=D\cap{^g}D$,
then $(Q,\bar{b}_Q)^g\subseteq(D,\bar{b}_D)$.
\end{lem}

\begin{proof}
Obviously, as $\U(D\times D)$-modules,
$\mathbb{A}$ is a direct summand of $A$ and $\mathbb{A}_\delta$ is a direct summand of $\mathbb{A}$.
Therefore, $\mathbb{A}_\delta$ is isomorphic to a direct summand of $A$ as $\U(D\times D)$-modules.
The statements in this lemma can follow from \cite[Theorem 8.7.1]{L'book 2}.
\end{proof}

\begin{lem}\label{basis contain 1}
There exists a $(D\times D)$-invariant $\U$-basis $\mathcal{B}$ of $\mathbb{A}_\delta$ containing 
the unit element $j$ of $\mathbb{A}_\delta$.
\end{lem}

\begin{proof}
The existence of a $(D\times D)$-invariant $\U$-basis $\mathcal{B}$ follows from 
Lemma \ref{basic result of mod stru}.
In order to show that the $\U$-basis $\mathcal{B}$ can be chosen to contain $j$,
it suffices to demonstrate that this $\mathcal{B}$ contains an invertible element in $\mathbb{A}_\delta^D$.
So we can assume that $\U=k$.
Set $j=\sum\limits_{x\in\mathcal{B}}\lambda_xx$ with $\lambda_x\in k$.
Then $\mathrm{Br}_D(j)=\sum\limits_{x\in\mathcal{B}^D}\lambda_x\mathrm{Br}_D(x)$ is the unit element of $\mathbb{A}_\delta(D)$ which is isomorphic to the local algebra $kZ(D)$ 
by Lemma \ref{Brauer quot in D}.
In particular, there is at least one element $x_0\in\mathcal{B}^D$ with 
$\lambda_{x_0}\neq 0$ and $\mathrm{Br}_D(x_0)$ being invertible in $\mathbb{A}_\delta(D)$.
Therefore, $x_0$ is invertible in $\mathbb{A}_\delta^D$ by \cite[Theorem (3.2)]{Thevenaz}.
\end{proof}

\begin{lem}\label{OD mult 1}
The multiplicity of $\U D$ in any decomposition of $\mathbb{A}_\delta$ 
as a direct sum of indecomposable $\U(D\times D)$-modules is $1$.
\end{lem}

\begin{proof}
Denote by $\Delta(D)$ the diagonal group of $D\times D$.
Then by Mackey's formula,
the direct summand of $\mathrm{Res}^{D\times D}_{\Delta(D)}(\U DgD)$ has the form
$\mathrm{Ind}_{\Delta(D)\cap{^{(u,v)}}\Delta_g(D)}^{\Delta(D)}(\U)$
with $(u,v)\in D\times D$.
It is easy to check that $\Delta(D)\cap{^{(u,v)}}\Delta_g(D)=\Delta(D)$
for some $(u,v)\in D\times D$ if and only if 
$g\in DC_G(D)$ and when $g\in DC_G(D)$,
$\U DgD=\U Dg$ is isomorphic to $\U D$ as $\U(D\times D)$-modules.
Hence, $(\U DgD)(D)$ is isomorphic to $kZ(D)$ when $g$ belongs to $DC_G(D)$.
Then by Lemma \ref{Brauer quot in D},
there is only one indecomposable direct summand of 
the $\U(D\times D)$-module $\mathbb{A}_\delta$ isomorphic to $\U D$.
\end{proof}

\begin{order}\rm
Recall that $N_G(D_\delta)/DC_G(D)$ is a $p$-nilpotent group 
with normal $p$-complement $E_\mathfrak{h}$.
By Proposition \ref{p-nilpotent quot},
the isomorphism $\mathfrak{f}_{G,D}$ from $N_G(D_\delta)/DC_G(D)$ to $F_A(D)$ 
maps $E_\mathfrak{h}$ onto $N_{\mathbb{A}_\delta}(D)(A_\delta^D)^\times/D(A_\delta^D)^\times$,
which is isomorphic to $N_{\mathbb{A}_\delta}(D)/D(\mathbb{A}_\delta^D)^\times$.
For any $g\in N_G(D_\delta)$ with $gDC_G(D)\in E_\mathfrak{h}$,
we denote by $\mathfrak{a}_g$ an element in $\mathbb{A}_\delta^\times$ such that
$\mathfrak{a}_g(uj)\mathfrak{a}_g^{-1}=(gug^{-1})j$ for any $u\in D$.
\end{order}

\begin{lem}\label{ODg mult 1}
As $\U(D\times D)$-modules,
$\U D\mathfrak{a}_g$ is isomorphic to $\U Dg$ and 
it is a direct summand of $\mathbb{A}_\delta$ with multiplicity $1$.
\end{lem}

\begin{proof}
It is easy to check that mapping $u\mathfrak{a}_g$ to $ug$ gives an $\U(D\times D)$-module isomorphism
from $\U D\mathfrak{a}_g$ to $\U Dg$.
By Lemma \ref{basis contain 1},
we can write $\mathbb{A}_\delta=\U Dj\oplus Y$ as $\U(D\times D)$-modules
for some $\U(D\times D)$-submodule $Y$.
Then we have 
$\mathbb{A}_\delta=\mathbb{A}_\delta\mathfrak{a}_g=\U D\mathfrak{a}_g\oplus Y\mathfrak{a}_g$
which is still a decomposition of $\mathbb{A}_\delta$ as $\U(D\times D)$-modules.
Hence, $\U D\mathfrak{a}_g$ is a direct summand of $\mathbb{A}_\delta$.
By Lemma \ref{OD mult 1},
we can get that the multiplicity of $\U D\mathfrak{a}_g$ is $1$.
\end{proof}

\begin{prop}\label{general mod stru}
There is a decomposition of $\mathbb{A}_\delta$ as an $\U(D\times D)$-module
\begin{align}\label{general dec as bimod}
\mathbb{A}_\delta=(\mathop\bigoplus\limits_{gDC_G(D)\in E_\mathfrak{h}}\U D\mathfrak{a}_g)
\oplus X,
\end{align}
where $X$ is isomorphic to a direct sum of $\U(D\times D)$-modules 
of the form $\U DhD$ for some $h\in G-N_G(D)$.
The $\U(D\times D)$-module $\U D\mathfrak{a}_g$ is independent of 
the choice of $g$ up to isomorphism and
has multiplicity $1$ for any $gDC_G(D)\in E_\mathfrak{h}$.
Moreover, there is a unitary subalgebra $\mathbb{B}$ of $\mathbb{A}_\delta$ 
isomorphic to the twisted group algebra $\U_\alpha(D\rtimes E_\mathfrak{h})$
for some $\alpha$ of $Z^2(E_\mathfrak{h};k^\times)$
such that it contains $\U Dj$ and 
$\mathbb{A}_\delta=\mathbb{B}\oplus Y$ as $\mathbb{B}$-$\mathbb{B}$-bimodules.
In particular, $\mathbb{B}$ and $Y$ are isomorphic to 
$\mathop\bigoplus\limits_{gDC_G(D)\in E_\mathfrak{h}}\U D\mathfrak{a}_g$ and 
$X$ as $\U(D\times D)$-modules, respectively.
\end{prop}

\begin{proof}
By Lemma \ref{ODg mult 1} and \cite[Lemma (44.7)]{Thevenaz},
we have 
$\mathbb{A}_\delta=
(\mathop\bigoplus\limits_{gDC_G(D)\in E_\mathfrak{h}}\U D\mathfrak{a}_g)\oplus X$
such that $\mathop\bigoplus\limits_{gDC_G(D)\in E_\mathfrak{h}}\U D\mathfrak{a}_g$ 
and $X$ have no common direct summand up to isomorphism.
Suppose that $X$ has a direct summand isomorphic to $\U DhD$ with $h\in N_G(D)$.
This implies that there is an element $\mathfrak{a}_h$ in $\mathbb{A}_\delta$
satisfying that $u\mathfrak{a}_hw^{-1}=\mathfrak{a}_h$ if and only if 
$w=h^{-1}uh$ for any $(u,w)\in D\times D$ and 
the $\U(D\times D)$-module $\U D\mathfrak{a}_h$ is a direct summand of $\mathbb{A}_\delta$.

On the other hand, 
by Lemma \ref{basic result of mod stru},
the element $h$ has to belong to $N_G(D_\delta)$.
Then $h=vx$ for some $v\in P$ and $xDC_G(D)\in E_\mathfrak{h}$.
Since there is an element $\mathfrak{a}_x$ in $\mathbb{A}_\delta^\times$ 
such that $\mathfrak{a}_x(uj)\mathfrak{a}_x^{-1}=(xux^{-1}j)$ for any $u\in D$,
$\U D\mathfrak{a}_h\mathfrak{a}_x^{-1}$ is also a direct summand of $\mathbb{A}_\delta$
which is isomorphic to $\U Dv$ as $\U(D\times D)$-modules.
Recall from the proof of Lemma \ref{Brauer quot in D},
for $v\in P$, there is an element $a_{v,\delta}$ of $A_\delta^\times\cap j\mathbb{A}vj$ such that 
$a_{v,\delta}(uj)a_{v,\delta}^{-1}=(vuv^{-1})j$ for any $u\in D$.
Therefore, $\mathbb{A}_\delta a_{v,\delta}^{-1}$ is contained in $j\mathbb{A}v^{-1}j$ and
has a direct summand isomorphic to $\U D$.
In particular, it implies that $(j\mathbb{A}v^{-1}j)(D)\neq 0$.
By Lemma \ref{Brauer quot in D},
we have $v$ is in $DC_P(D)$.
So we can get that $X$ has a direct summand isomorphic to 
$\U D\mathfrak{a}_x$, which is impossible.
So the direct summand of $X$ has the form $\U DhD$ with $h\in G-N_G(D)$.
The second statement follows from Lemma \ref{ODg mult 1}.

Let us consider the remaining statements of this proposition.
Since $E_\mathfrak{h}$ is a $p^\prime$-group,
we can identify it with a subgroup of $N_G(D_\delta)/C_G(D)$.
For any $g\in N_G(D_\delta)$,
we denote $gC_G(D)\in N_G(D_\delta)/C_G(D)$ by $\tilde{g}$.
Recall from the paragraph 4.4 that
for any $\tilde{g}\in E_\mathfrak{h}$,
there is an element $\mathfrak{a}_g$ in $\mathbb{A}_\delta^\times$
such that $\mathfrak{a}_g(uj)\mathfrak{a}_g^{-1}=(gug^{-1})j$ for any $u\in D$.
Denote the subset $\{\mathfrak{a}_g\mathfrak{a}_\delta\mid\tilde{g}\in E_\mathfrak{h},\mathfrak{a}_\delta\in(\mathbb{A}_\delta^D)^\times\}$ by $\hat{E}_\mathfrak{h}$.
Clearly, $\hat{E}_\mathfrak{h}$ is a subgroup of $N_{\mathbb{A}_\delta}(D)$
containing $(\mathbb{A}_\delta^D)^\times$ such that 
$\hat{E}_\mathfrak{h}/(\mathbb{A}_\delta^D)^\times$ is isomorphic to $E_\mathfrak{h}$.
Since $(\mathbb{A}_\delta^D)^\times$ is isomorphic to 
$(j+J(\mathbb{A}_\delta^D))\times k^\times$,
we can get the following short exact sequence
\begin{equation*}
	\xymatrix{
		1\ar[r]& j+J(\mathbb{A}_\delta^D)\ar[r] & \hat{E}_\mathfrak{h}/k^\times \ar[r] & 
		E_\mathfrak{h}\ar[r] & 1.
	}
\end{equation*}
By \cite[Lemma (45.6)]{Thevenaz}, this sequence  splits since $E_\mathfrak{h}$ is a $p^\prime$-group.
Now let $\mathbb{M}$ be a subgroup  of $\hat{E}_\mathfrak{h}$ containing $k^\times$
such that $\mathbb{M}/k^\times$ is a complement of $j+J(\mathbb{A}_\delta^D)$ in $\hat{E}_\mathfrak{h}/k^\times$.
This yields another short exact sequence
\begin{equation*}
	\xymatrix{
		1\ar[r]& k^\times\ar[r] &\mathbb{M} \ar[r] & 
		E_\mathfrak{h}\ar[r] & 1,
	}
\end{equation*}
which determines an element $\alpha$ in $Z^2(E_\mathfrak{h};k^\times)$.
Set $\mathbb{B}$ to be the $\U$-submodule of $\mathbb{A}_\delta$ 
generated by $Dj\rtimes\mathbb{M}$.
It is routine to check that $\mathbb{B}$ is a unitary subalgebra containing $\U Dj$ of $\mathbb{A}_\delta$ which is 
isomorphic to $\U_\alpha(D\rtimes E_\mathfrak{h})$ as $D$-interior algebras and
isomorphic to $\mathop\bigoplus\limits_{gDC_G(D)\in E_\mathfrak{h}}\U D\mathfrak{a}_g$
as $\U(D\times D)$-modules.
Therefore, $\mathrm{Br}_D(\mathbb{B}^D)$ is isomorphic to $kZ(D)$.
By \cite[Lemma 5.5 and Proposition 5.15]{P88},
there is a finite subgroup $\mathbb{M}^\prime$ of $\mathbb{M}$ such that
$\mathbb{B}$ is isomorphic to $\U(Dj\rtimes\mathbb{M}^\prime)\mathfrak{b}^\prime$ 
for some central idempotent $\mathfrak{b}^\prime$ of $\U(Dj\rtimes\mathbb{M}^\prime)$
as $D$-interior algebras.
Then $\mathrm{Br}_D((\U(Dj\rtimes\mathbb{M}^\prime)\mathfrak{b}^\prime)^D)$ is isomorphic to
the local $k$-algebra $kZ(D)$ which implies
$\mathfrak{b}^\prime$ is a block idempotent since $Dj$ is normal in $Dj\rtimes\mathbb{M}^\prime$.
Hence, by \cite[Proposition 6.7.15]{L'book 2},
the inclusion map $\mathbb{B}$ to $\mathbb{A}_\delta$ is split injective as 
a homomorphism of $\mathbb{B}$-$\mathbb{B}$-bimodules.
We are done.
\end{proof}

Before giving a refinement of this $\U(D\times D)$-module structure for the hyperfocal subalgebra
of a hyperfocal abelian Frobenius block,
we state three consequences of Proposition \ref{general mod stru}.
First, we can get the information about the rank of $\mathbb{A}_\delta$ and 
the dimension of simple modules as follows.

\begin{cor}\label{dim of bbA and simple mod}
	We have 
	$\frac{\mathrm{rank}_\mathcal{O}(\mathbb{A}_\delta)}{|D|}\equiv |E_\mathfrak{h}|
	 (\mathrm{mod} ~|D|)$.
	In particular,
	there is at least one simple $k\mathop\otimes\limits_\mathcal{O}\mathbb{A}_\delta$-module
	of dimension coprime to $p$.
\end{cor}

As the second consequence of Proposition \ref{general mod stru},
we can show that a special automorphism of $\mathbb{A}_\delta$ is an inner automorphism.
This result is similar to \cite[Proposition 14.9]{P88} and may be of independent interest.

\begin{cor}\label{auto is inner}
With the notation above, let $f$ be an $\U$-algebra automorphism of $\mathbb{A}_\delta$
such that its restriction to the subalgebra $\mathbb{B}$ is trivial.
Then $f$ is an inner automorphism induced by an element in 
$C_{\mathbb{A}_\delta}(\mathbb{B})^\times$.
Here, $C_{\mathbb{A}_\delta}(\mathbb{B})^\times$ denotes the centralizer of $\mathbb{B}$ in $\mathbb{A}_\delta^\times$.
\end{cor}

\begin{proof}
Since $\mathbb{A}_\delta=\mathbb{B}\oplus Y$ as $\mathbb{B}$-$\mathbb{B}$-bimodules,
we have $\mathbb{A}_\delta\mathop\otimes\limits_{\mathbb{B}}\mathbb{A}_\delta$ being isomorphic to
$\mathbb{A}_\delta\oplus Y\oplus(Y\mathop\otimes\limits_{\mathbb{B}}Y)$ as $\mathbb{B}$-$\mathbb{B}$-bimodules.	
Since $Y$ is isomorphic to $X$ as $\U(D\times D)$-modules,
we have $\mathrm{Br}_D(Y)\cong\mathrm{Br}_D(X)=0$.
For $Y\mathop\otimes\limits_{\mathbb{B}}Y$,
it is clear that it is isomorphic to a direct summand of $Y\mathop\otimes\limits_{\mathcal{O}D}Y$ as 
$\mathbb{B}$-$\mathbb{B}$-bimodules and then
it is isomorphic to $X\mathop\otimes\limits_{\mathcal{O}D}X$ as $\U(D\times D)$-modules.
By the same argument in the proof of \cite[Lemma 3.3]{HZ25},
we have $(X\mathop\otimes\limits_{\mathcal{O}D}X)(D)=0$.
So is $(Y\mathop\otimes\limits_{\mathbb{B}}Y)(D)$.
Since $\mathbb{B}$ contains $\U D$ and $\mathbb{A}_\delta$ is relatively $\U D$-separable,
it is easy to check that $\mathbb{A}_\delta$ is also relatively $\mathbb{B}$-separable.
In particular,
$\mathbb{A}_\delta$ is isomorphic to a direct summand of 
$\mathbb{A}_\delta\mathop\otimes\limits_{\mathbb{B}}\mathbb{A}_\delta$ as 
$\mathbb{A}_\delta$-$\mathbb{A}_\delta$-bimodules.
Then $\mathbb{A}_\delta$ is the unique direct summand, up to isomorphism, of the $\mathbb{A}_\delta$-$\mathbb{A}_\delta$-bimodule
$\mathbb{A}_\delta\mathop\otimes\limits_{\mathbb{B}}\mathbb{A}_\delta$,
 with respect to the Brauer quotients at $D$ being nonzero.

On the other hand, since $\mathbb{A}_\delta$ is relatively $\mathbb{B}$-separable,
then the multiplication map 
$\mathbb{A}_\delta\mathop\otimes\limits_{\mathbb{B}}\mathbb{A}_\delta\longrightarrow
\mathbb{A}_\delta$ splits.
This means that there is an element 
$\sum\limits_l\mathfrak{a}_l\mathop\otimes\limits_{\mathbb{B}}
\mathfrak{a}^\prime_l$ in $\mathbb{A}_\delta\mathop\otimes\limits_{\mathbb{B}}\mathbb{A}_\delta$
such that $\sum\limits_l\mathfrak{a}_l\mathfrak{a}^\prime_l=j$ and
$\sum\limits_l\mathfrak{a}\mathfrak{a}_l\mathop\otimes\limits_{\mathbb{B}}
\mathfrak{a}^\prime_l=
\sum\limits_l\mathfrak{a}_l\mathop\otimes\limits_{\mathbb{B}}
\mathfrak{a}^\prime_l\mathfrak{a}$
for any $\mathfrak{a}\in\mathbb{A}_\delta$. 
Let $(\mathbb{A}_\delta)_f$ be the $\mathbb{A}_\delta$-$\mathbb{A}_\delta$-bimodule
such that it equals $\mathbb{A}_\delta$ as left $\mathbb{A}_\delta$-modules
with $f(\mathfrak{a})$ acting by right multiplication on 
$\mathfrak{a}^\prime\in(\mathbb{A}_\delta)_f$ 
as $\mathfrak{a}$ on $\mathfrak{a}^\prime\in\mathbb{A}_\delta$.
With the same arguments in the proof of \cite[Proposition 3.11]{HZ25},
$(\mathbb{A}_\delta)_f$ is isomorphic to a direct summand of 
$(\mathbb{A}_\delta)_f\mathop\otimes\limits_{\mathbb{B}}\mathbb{A}_\delta$
as $\mathbb{A}_\delta$-$\mathbb{A}_\delta$-bimodules.
Since the restriction of $f$ to $\mathbb{B}$ is trivial,
it is easy to check that $(\mathbb{A}_\delta)_f\mathop\otimes\limits_{\mathbb{B}}\mathbb{A}_\delta$
is isomorphic to $\mathbb{A}_\delta\mathop\otimes\limits_{\mathbb{B}}\mathbb{A}_\delta$
as $\mathbb{A}_\delta$-$\mathbb{A}_\delta$-bimodules.
Obviously, $(\mathbb{A}_\delta)_f$ is isomorphic to $\mathbb{A}_\delta$ 
as $\U(D\times D)$-modules.
By the uniqueness of $\mathbb{A}_\delta$,
$(\mathbb{A}_\delta)_f$ is also isomorphic to $\mathbb{A}_\delta$ 
as $\mathbb{A}_\delta$-$\mathbb{A}_\delta$-bimodules.
The image of $j$ under this isomorphism is the element we want.
\end{proof}

The third consequence is about the equivariant property of $\mathbb{A}$-modules.
To state this property, we need some notation for twisted modules.

\begin{order}\rm
Let $M$ be an $\mathbb{A}$-module.
Fix an element $u$ in $P$.
We can define the twisted $\mathbb{A}$-module ${_u}M$	
by setting ${_u}M=M$ as an $\U$-module,
with $u\mathfrak{a}u^{-1}$ acting on $m\in{_u}M$ as $\mathfrak{a}$ on $m\in M$
for any $\mathfrak{a}\in\mathbb{A}$ and any $m\in M$. 
Furthermore, recall from the paragraph 3.1 that
there is an element $a_{u,\delta}$ in $A_\delta^\times$ with
$\mathbb{A}_\delta^{a_{u,\delta}}=\mathbb{A}_\delta$.
Then we can similarly define the twisted $\mathbb{A}_\delta$-module ${_u}M_\delta$
for any $\mathbb{A}_\delta$-module $M_\delta$
through the conjugation action of the element $a_{u,\delta}$.
Since the choice of the element $a_{u,\delta}$ is unique up to a multiplication by 
an element in $(\mathbb{A}_\delta^D)^\times$,
so the twisted module ${_u}M_\delta$ is independent of the choice of the element $a_{u,\delta}$
up to isomorphism.
At the same time, letting $C$ be a $P$-algebra,
we have the analogous notation ${_{(u,u)}}{\bf M}$ and ${_{(u,u)}}{\bf M}^\prime$
for $\mathbb{A}$-$C$-bimodule ${\bf M}$
and $\mathbb{A}_\delta$-$C$-bimodule ${\bf M}^\prime$, respectively.
\end{order}

\begin{cor}\label{equivariant for bbA-module}
Let $u$ be an element in $C_P(D)$.
Then for any $\mathbb{A}$-module $M$,
it is isomorphic to its twisted module ${_u}M$ as $\mathbb{A}$-modules.
In particular, if $D$ is in the center $Z(P)$ of $P$,
every $\mathbb{A}$-module $M$ is isomorphic to its twisted module ${_v}M$ as $\mathbb{A}$-modules
for any $v\in P$.
\end{cor}

\begin{proof}
We adopt the notation in the paragraph above and 
the proof of Proposition \ref{general mod stru}.	
For the sake of clarity, we denote the element $m$ in $M$ by ${_u}m$
when we view it as an element in the twisted module ${_u}M$.
We use the analogous notation for $\mathbb{A}_\delta$-modules.

First, we show that $j({_v}M)$ is isomorphic to ${_v}(jM)$ as $\mathbb{A}_\delta$-modules
for any $v\in P$.
Recall that $a_{v,\delta}=a_v^{-1}vj$ for some $a_v\in(\mathbb{A}^D)^\times$
with $a_vja_v^{-1}=vjv^{-1}$.
We define a map ${\bf f}$ from $j({_v}M)$ to ${_v}(jM)$ by sending 
$j\cdot({_v}m)\in j({_v}M)$ to ${_v}(ja_v\cdot m)\in{_v}(jM)$.
Suppose that there is an element $m^\prime$ in $M$ such that
$j\cdot({_v}m)=j\cdot({_v}m^\prime)$.
Then we have the equation $v^{-1}jv\cdot m=v^{-1}jv\cdot m^\prime$ in $M$, by which
we can get the equation $ja_v\cdot m=ja_v\cdot m^\prime$ in $M$
since $a_v$ is invertible in $\mathbb{A}^D$.
So the map ${\bf f}$ defined above from $j({_v}M)$ to ${_v}(jM)$ is well-defined and
obviously bijective.
Now for any $\mathfrak{a}_\delta\in\mathbb{A}_\delta$, we have
\begin{align*}
{\bf f}(\mathfrak{a}_\delta\cdot({_v}m))
	=~&{\bf f}(j\cdot({_v}(v^{-1}\mathfrak{a}_\delta v\cdot m)))\\
	=~&{_v}(ja_vv^{-1}\mathfrak{a}_\delta v\cdot m)\\
	=~&{_v}(ja_vv^{-1}\mathfrak{a}_\delta va_v^{-1}\cdot(a_v\cdot m))\\
	=~&\mathfrak{a}_\delta\cdot{_v}(ja_v\cdot m)\\
	=~&\mathfrak{a}_\delta\cdot{\bf f}(j\cdot({_v}m)),
\end{align*}
which means that ${\bf f}$ is an isomorphism of $\mathbb{A}_\delta$-modules.

Next, we show that the element $a_{u,\delta}$ can be chosen such that 
it acts trivially on the subalgebra $\mathbb{B}$ through the conjugation action.
Recall that
 $\hat{E}_\mathfrak{h}=\{\mathfrak{a}_g\mathfrak{a}_\delta\mid\tilde{g}\in E_\mathfrak{h},\mathfrak{a}_\delta\in(\mathbb{A}_\delta^D)^\times\}$ is a subgroup of $N_{\mathbb{A}_\delta}(D)$ containing $(\mathbb{A}_\delta^D)^\times$.
Since $u$ belongs to $C_P(D)$, we have
$\mathfrak{a}_g^{-1}{^{a_{u,\delta}}}\mathfrak{a}_g$ belonging to $(\mathbb{A}_\delta^D)^\times$.
In particular, we have ${^{a_{u,\delta}}}\hat{E}_\mathfrak{h}=\hat{E}_\mathfrak{h}$.
Therefore, ${^{a_{u,\delta}}}\mathbb{M}/k^\times$ is also a complement of 
$j+J(\mathbb{A}_\delta^D)$ in $\hat{E}_\mathfrak{h}/k^\times$.
Then by \cite[Lemma (45.6)]{Thevenaz},
${^{a_{u,\delta}}}\mathbb{M}$ is conjugate to $\mathbb{M}$ 
by an element in $j+J(\mathbb{A}_\delta^D)$.
We can adjust the element $a_{u,\delta}$ by multiplying this element in $j+J(\mathbb{A}_\delta^D)$
such that ${^{a_{u,\delta}}}\mathbb{M}=\mathbb{M}$.
Without loss of generality,
we can set $\mathbb{M}=\{\mathfrak{a}_g\mu\mid
\tilde{g}\in E_\mathfrak{h},\mu\in k^\times\}$.
So for any $\tilde{g}\in E_\mathfrak{h}$,
${^{a_{u,\delta}}}\mathfrak{a}_g=\mu_g\mathfrak{a}_g$ for some $\mu_g\in k^\times$.
Since $u$ is a $p$-element,
there is some positive integer $a$ such that 
$a_{u,\delta}^{p^a}$ belongs to $(\mathbb{A}_\delta^D)^\times$.
Then $a_{u,\delta}^{p^a}=\mathfrak{a}_\delta\lambda$ for 
a unique element $\mathfrak{a}_\delta$ in $j+J(\mathbb{A}_\delta^D)$ and a unique element
$\lambda$ in $k^\times$.
Hence,
$${^{\mathfrak{a}_\delta}}\mathfrak{a}_g={^{a_{u,\delta}^{p^a}}}\mathfrak{a}_g=
\mu_g^{p^a}\mathfrak{a}_g.$$
Multiplying by $\mathfrak{a}_g^{-1}$ on the right yields an equation
${^{\mathfrak{a}_\delta}}\mathfrak{a}_g\mathfrak{a}_g^{-1}=\mu_g^{p^a}$.
Note that the element in the left side belongs to $j+J(\mathbb{A}_\delta^D)$
and $\mu_g^{p^a}$ is an element in $k^\times$.
So $\mu_g$ has to be $1$.
This completes the proof of the argument in this paragraph.

Moreover, by Corollary \ref{auto is inner},
there is an element $z_{u,\delta}$ in $C_{\mathbb{A}_\delta}(\mathbb{B})^\times$ such that
${^{a_{u,\delta}}}\mathfrak{c}_\delta={^{z_{u,\delta}}}\mathfrak{c}_\delta$
for any $\mathfrak{c}\in\mathbb{A}_\delta$.
Then it is easy to check that for any $\mathbb{A}_\delta$-module $M_\delta$,
it is isomorphic to its twisted module ${_u}M_\delta$.
Applying this fact to the $\mathbb{A}$-module $jM$,
we can get the following isomorphisms of $\mathbb{A}_\delta$-modules
$$j({_u}M)\cong{_u}(jM)\cong jM.$$
By Corollary \ref{separable and Morita equi},
the $\mathbb{A}$-$\mathbb{A}_\delta$-bimodule $\mathbb{A}j$ induces a Morita equivalence 
between $\mathbb{A}$ and $\mathbb{A}_\delta$.
Hence, we have ${_u}M$ isomorphic to $M$ as $\mathbb{A}$-modules.
\end{proof}

\begin{order}\rm
For the remainder of this section we assume that 
the block $b$ is a hyperfocal abelian Frobenius block.
We will give a more explicit description of 
the $\U(D\times D)$-module structure of $X$ in the decomposition
(\ref{general dec as bimod})
under this assumption.
To do this, we need to investigate the Brauer quotients of $\mathbb{A}_\delta$  
at all nontrivial subgroups of $D$.
We start with a special case.
\end{order}

\begin{lem}\label{local structure at Q not 1 special}
Keep the assumption as above.
Let $Q$ be a nontrivial subgroup of $D$.
Assume that $C_P(Q)$ is a defect group of the block $b_Q$ of $C_G(Q)$.
Then there is an indecomposable endopermutation $k(D/Q)$-module $V_Q$ with vertex $D/Q$
such that setting $\bar{S}_Q=\mathrm{End}_k(V_Q)$,
we have the following $D$-interior algebra isomorphism
\begin{align}\label{A Brauer quot at Q not 1 special}
	A_\delta(Q)\cong\bar{S}_Q\mathop\otimes\limits_kkC_P(Q)
\end{align}
with its restriction to $\mathbb{A}_\delta(Q)$ giving the following $D$-interior algebra isomorphism
\begin{align}\label{bbA Brauer quot at Q not 1 special}
\mathbb{A}_\delta(Q)\cong\bar{S}_Q\mathop\otimes\limits_kkD.
\end{align}
In particular,
$(j\mathbb{A}vj)(Q)\neq 0$ if and only if $v$ is in $C_P(Q)$.
\end{lem}

\begin{proof}
Denote $N_P(Q)$ by $P_Q$.
It is clear that $b_Q$ is still a block of $P_QC_G(Q)$ over $\U$.
We denote it by $c_Q$ when the block $b_Q$ is viewed as a block of $P_QC_G(Q)$.
By the assumption,
$P_Q$ is a defect group of the block $c_Q$.
Since the block $b_Q$ is nilpotent by Proposition \ref{nilpotent block of D at local},
the block $c_Q$ of $P_QC_G(Q)$ is also nilpotent.
Since $P_Q$ contains $D$,
there is a unique local pointe $\varepsilon$ of $P_Q$ on $\U Gb$
such that $(P_Q)_\varepsilon\leq P_\gamma$ 
by \cite[Proposition 4.2]{P00}.
Now we can take an element $l$ of $\varepsilon$ such that
$l$ is in $\mathbb{A}$ by \cite[Proposition 2.8]{P00} and $lj=jl=j$.

By \cite[Lemma (40.2)]{Thevenaz},
$\mathrm{Br}_Q(l)$ belongs to a local point of $P_Q$ on $kC_G(Q)$.
Since $P_Q$ is a defect group of the block $c_Q$,
$\mathrm{Br}_Q(l)kP_QC_G(Q)\mathrm{Br}_Q(l)$ is a source algebra of the block $\bar{c}_Q$ over $k$.
Since the block $c_Q$ is nilpotent, by \cite[Theorem 8.12.3]{L'book 2},
there is a an indecomposable endopermutation $kP_Q$-module $V_{P_Q}$ with vertex $P_Q$
such that, setting $\bar{T}_Q=\mathrm{End}_k(V_{P_Q})$,
we have an isomorphism of $P_Q$-interior algebras
\begin{align*}
\mathrm{Br}_Q(l)kP_QC_G(Q)\mathrm{Br}_Q(l)\cong\bar{T}_Q\mathop\otimes\limits_kkP_Q,
\end{align*}
which restricts to a $C_P(Q)$-interior algebra isomorphism
\begin{align}\label{source alg of C_G(Q)}
\mathrm{Br}_Q(l)kC_G(Q)\mathrm{Br}_Q(l)\cong\bar{T}_Q\mathop\otimes\limits_kkC_P(Q).
\end{align}
Since $Q$ is normal in $P_QC_G(Q)$,
it is clear that the action of $Q$ on $V_{P_Q}$ is trivial.
Then $V_{P_Q}$ can be viewed as a $k(P_Q/Q)$-module.
Note that $\mathrm{Br}_Q(l)kC_G(Q)\mathrm{Br}_Q(l)=(lAl)(Q)$ and
$(l\mathbb{A}l)(Q)$ is a unitary subalgebra of $(lAl)(Q)$.
By the proof of  \cite[Proposition 4.3]{P00}, 
the restriction of the isomorphism (\ref{source alg of C_G(Q)})  to $(l\mathbb{A}l)(Q)$ gives
a $D$-interior algebra isomorphism
\begin{align}\label{hyperfocal alg of C_G(Q)}
(l\mathbb{A}l)(Q)\cong\bar{T}_Q\mathop\otimes\limits_kkD.
\end{align}

It is well-known that there is a unique local point, denoted by $\tilde{\delta}$, 
of $D$ on $\bar{T}_Q$.
Take an element $f$ in $\tilde{\delta}$ and set $V_Q=fV_{P_Q}$.
Then $V_Q$ is an indecomposable endopermutation $kD$-module with vertex $D$.
Clearly, the action of $Q$ on $V_Q$ is also trivial.
So $V_Q$ can be viewed as a $k(D/Q)$-module with vertex $D/Q$.
Set $\bar{S}_Q=\mathrm{End}_k(V_Q)\cong f\bar{T}_Qf$.
By \cite[Proposition 7.4.4]{L'book 2},
the $D$-interior algebras $\bar{S}_Q\mathop\otimes\limits_k kC_P(Q)$ and $\bar{S}_Q\mathop\otimes\limits_kkD$ 
are both primitive. 
Note that $\mathrm{Br}_Q:\mathbb{A}^D\rightarrow\mathbb{A}(Q)^{D/Q}$ is surjective
and $\mathrm{Br}_{D/Q}\circ\mathrm{Br}_Q$ is just the Brauer homomorphism
$\mathrm{Br}_D$ from $\mathbb{A}^D$ to $\mathbb{A}(D)$.
Then $\mathrm{Br}_Q(j)$ belongs to a local point of $D$ on $((l\mathbb{A}l)(Q))^D$.
Comparing the isomorphism (\ref{source alg of C_G(Q)}) 
with the isomorphism ({\ref{hyperfocal alg of C_G(Q)}}),
we get 
$(lAl)(Q)=\mathop\bigoplus\limits_{uD\in C_P(Q)/D}(l\mathbb{A}l)(Q)\cdot u$.
Therefore, by \cite[Proposition 2.8]{P00},
$\mathrm{Br}_Q(j)$ also belongs to a local point of $D$ on $((lAl)(Q))^D$.
Note that by \cite[5.3]{P-88}, local points of $D$ on $(lAl)(Q)$ and $(l\mathbb{A}l)(Q)$ are both unique.
Therefore, we can assume that the isomorphism (\ref{source alg of C_G(Q)}) 
maps $\mathrm{Br}_Q(j)$ to $f\mathop\otimes\limits_k1_{C_P(Q)}$.
Then the isomorphisms (\ref{A Brauer quot at Q not 1 special}) and 
(\ref{bbA Brauer quot at Q not 1 special}) are obtained.

Borrowing the notation from the paragraph 3.1,
we have 
$A_\delta=\mathop\bigoplus\limits_{uD\in P/D}\mathbb{A}_\delta a_{u,\delta}=
\mathop\bigoplus\limits_{uD\in P/D}j\mathbb{A}uj$
with $a_{u,\delta}\in A_\delta^\times\cap\mathbb{A}u$ satisfying
$a_{u,\delta}(wj)a_{u,\delta}^{-1}=(uwu^{-1})j$ for any $w\in D$.
So $A_\delta(Q)=\mathop\bigoplus\limits_{uD\in P/D}(j\mathbb{A}uj)(Q)$
and 
$(j\mathbb{A}uj)(Q)=\mathbb{A}_\delta(Q)\cdot\mathrm{Br}_Q(a_{u,\delta})$
when $u$ is in $C_P(Q)$.
In particular, $\mathrm{dim}_k(\mathbb{A}_\delta(Q))=\mathrm{dim}_k((j\mathbb{A}uj)(Q))$
when $u$ is in $C_P(Q)$.
By isomorphisms (\ref{A Brauer quot at Q not 1 special}) and 
(\ref{bbA Brauer quot at Q not 1 special}),
we have $\mathrm{dim}_k(A_\delta(Q))=\frac{|C_P(Q)|}{|D|}\mathrm{dim}_k(\mathbb{A}_\delta(Q))$.
We are done.
\end{proof}

By using Lemma \ref{local structure at Q not 1 special},
we can obtain the strucutre of $\mathbb{A}_\delta(Q)$ for arbitrary nontrivial subgroup $Q$ of $D$.
And this general structure will help us construct a stable equivalence of Morita type
between $\mathbb{A}_\delta$ and $\U(D\rtimes E_\mathfrak{h})$ in the next section.

\begin{prop}\label{local structure at Q not 1 general}
For any nontrivial subgroup $Q$	of $D$,
there is an indecomposable endopermutation $k(D/Q)$-module $V_Q$ with vertex $D/Q$
such that, setting $\bar{S}_Q=\mathrm{End}_k(V_Q)$,
we have a $D$-interior algebra isomorphism
\begin{align}\label{bbA Brauer quot at Q not 1 general}
	\mathbb{A}_\delta(Q)\cong\bar{S}_Q\mathop\otimes\limits_kkD.
\end{align}
Moreover,
if $V_Q^\prime$ is another indecomposable endopermutation $k(D/Q)$-module with vertex $D/Q$
such that  setting $\bar{S}_Q^\prime=\mathrm{End}_k(V_Q^\prime)$,
$\mathbb{A}_\delta(Q)\cong\bar{S}_Q^\prime\mathop\otimes\limits_kkD$ as $D$-interior algebras,
then $V_Q$ is isomorphic to $V_Q^\prime$ as $k(D/Q)$-modules.
\end{prop}

\begin{proof}
It is well-known that there is some $x\in G$ such that
${^x}(Q,\bar{b}_Q)\leq(P,\bar{b}_P)$ and $C_P({^x}Q)$ is a defect group of 
the block ${^x}b_Q$ of $C_G({^x}Q)$ over $\U$.
Equivalently, $C_{P^x}(Q)$ is a defect group of the block $b_Q$.
Since $D$ is abelian, by \cite[Theomre 2]{W14}, we can assume that $x$ is in $N_G(D_\delta)$.
Furthermore, since $N_G(D_\delta)/C_G(D)=PC_G(D)/C_G(D)\cdot E_\mathfrak{h}$,
we can assume that $xC_G(D)$ is in $E_\mathfrak{h}$.
Then we have $\mathfrak{a}_x\in\mathbb{A}_\delta^\times$ such that
$\mathfrak{a}_x(wj)\mathfrak{a}_x^{-1}=(xwx^{-1})j$ for any $w\in D$.	

Clearly, $A^x$ is also a source algebra of the block $b$
determined by the defect pointed group $(P_\gamma)^x$.
Since $x$ is in $N_G(D_\delta)$,
we have $\mathbb{A}^x$ is a hyperfocal subalgebra of $A^x$
with hyperfocal subgroup $D$.
The conjugation action of $x$ induces an $\U$-algebra isomorphism
$\mathbb{A}_\delta^{{^x}Q}\longrightarrow((\mathbb{A}^x)_{\delta^x})^Q$.
At the same time, the conjugation action of $\mathfrak{a}_x$ induces 
an $\U$-algebra isomorphism $\mathbb{A}_\delta^Q\longrightarrow\mathbb{A}_\delta^{{^x}Q}$.
Composing these two isomorphisms together,
we get an $\U$-algebra isomorphism 
$\mathbb{A}_\delta^Q\longrightarrow(\mathbb{A}^x)_{\delta^x}^Q$,
which is in fact a $D$-interior algebra isomorphism.
By Lemma \ref{local structure at Q not 1 special}, 
we get this general structure of $\mathbb{A}_\delta(Q)$.	
	
Let $V_Q^\prime$ be another indecomposable endopermutation 
$k(D/Q)$-module with vertex $D/Q$ fulfilling the property in this lemma.
	Then we have $\bar{S}_Q\mathop\otimes\limits_kkD\cong\bar{S}_Q^\prime\mathop\otimes\limits_kkD$ 
	as $D$-interior algebras.
	Tensoring with $\bar{S}_Q^{\circ}\mathop\otimes\limits_k-$ on the both sides of this isomorphism 
	yields a $D$-interior algebra isomorphism
	$\bar{S}_Q^\circ\mathop\otimes\limits_k\bar{S}_Q\mathop\otimes\limits_kkD\cong
	\bar{S}_Q^\circ\mathop\otimes\limits_k\bar{S}_Q^\prime\mathop\otimes\limits_kkD$.
	By \cite[5.3]{P-88},
	local points of $D$ on $\bar{S}_Q^\circ\mathop\otimes\limits_k\bar{S}_Q^\prime$ and 
	$\bar{S}_Q^\circ\mathop\otimes\limits_k\bar{S}_Q\mathop\otimes\limits_kkD$ are both unique.
	We take two elements $e^\prime$ and $e$ belonging to these two unqiue local points, respectively.
	Then by \cite[Theorem 7.4.2]{L'book 2},
	$e(\bar{S}_Q^\circ\mathop\otimes\limits_k\bar{S}_Q\mathop\otimes\limits_kkD)e$ is isomorphic to $kD$ 
	as $D$-interior algebras.
	It is clear that $e^\prime(\bar{S}_Q^\circ\mathop\otimes\limits_k\bar{S}_Q^\prime)e^\prime$
	is also a primitive Dade $D$-interior algebra.
	By \cite[Proposition 7.4.4]{L'book 2},
	we have $e^\prime(\bar{S}_Q^\circ\mathop\otimes\limits_k\bar{S}_Q^\prime)e^\prime\mathop\otimes\limits_kkD$
	is a primitive $D$-interior algebra.
	Hence, we have 
	$e^\prime(\bar{S}_Q^\circ\mathop\otimes\limits_k\bar{S}_Q^\prime)e^\prime\mathop\otimes\limits_kkD\cong kD$
	as $D$-interior algebras.
	This implies that $e^\prime(\bar{S}_Q^\circ\mathop\otimes\limits_k\bar{S}_Q^\prime)e^\prime$
	is a trivial $D$-interior algebra.
	Then the elements in the Dade group (see \cite[Definition 7.3.11]{L'book 2}) 
	of $D$ corresponding to $V_Q$ and $V_Q^\prime$ are inverse to each other.
	In particular, $V_Q$ is isomorphic to $V_Q^\prime$ as $k(D/Q)$-modules.
\end{proof}

\begin{rmk}\rm
With the notation of Proposition \ref{local structure at Q not 1 general}, 
$A_\delta(Q)$ may not be isomorphic to $\bar{S}_Q\mathop\otimes\limits_kkC_P(Q)$ in general.
It behaves different from the special case in Lemma \ref{local structure at Q not 1 special}.
For example, let $G$ be the symmetric group ${\rm S}_4$ on $4$ letters
and $p$ equal to $2$.
Let $b$ be the principal block of $G$ over $\U$.
In fact, $b$ is just the unit element $1_{\mathcal{O}G}$ and 
the source algebra $A$ is equal to the group algebra $\U G$.
In this case, the hyperfocal subgroup $D$ is a Klein four group and
the hyperfocal subalgebra $\mathbb{A}$ is just the group algebra $\U{\rm A}_4$,
where ${\rm A}_4$ is the alternating group.
Set $P$ to be $D\rtimes(12)$.
Let $Q$ be the subgroup $\{(1),(13)(24)\}$ of $D$.
Then $C_P(Q)$ is just equal to $D$ and $C_G(Q)$ is equal to another Sylow $2$-subgroup
$P^\prime=D\rtimes(13)$.
In this case, $A_\delta$ is isomorphic to $kC_G(Q)$, which is $kP^\prime$.

Due to this, in general $(j\mathbb{A}vj)(Q)$ may be not zero when $v$ does not belong to $C_P(Q)$
since in the concrete example above $(j\mathbb{A}vj)(Q)$ is equal to 
$(kD)(13)$ 
when $v$ equals $(12)$ not in $C_P(Q)$.
\end{rmk}

\begin{order}\rm
	Let $Q$ be a subgroup of $D$ and $\varphi$ a homomorphism from $Q$ to $D$.
	We denote the subgroup $\{(u,\varphi(u))\mid u\in Q\}$ of $D\times D$ by $\Delta_\varphi(Q)$.	
	For any $g\in G$ with $Q^g\leq D$, we denote by $\varphi_g$ 
	the homomorphism from $Q$ to $D$ induced by the conjugation action of $g$. 
	Now we can refine the $\U(D\times D)$-structure of $\mathbb{A}_\delta$
	in the case where the block $b$ is a hyperfocal abelian Frobenius block as follows.
\end{order}

\begin{prop}\label{special mod stru}
Assume that the block $b$ is a hyperfocal abelian Frobenius block.
Then every direct summand of the $\U(D\times D)$-module $X$ 
appearing in the decomposition (\ref{general dec as bimod}) has the form
$\mathrm{Ind}_{\Delta_\varphi(Q)}^{D\times D}(\U)$
with $Q$ being a proper subgroup of $D$ and 
$\varphi:Q\longrightarrow D$ induced by an element in $E_\mathfrak{h}$.
Moreover, there is an injective group homomorphism 
from $E_\mathfrak{h}$ to $N_{\mathbb{A}_\delta}(D)$ sending
$\tilde{g}=gC_G(D)$ to $\mathfrak{e}_{\tilde{g}}$ satisfying that  $(gwg^{-1})j=\mathfrak{e}_{\tilde{g}}(wj)\mathfrak{e}_{\tilde{g}}^{-1}$ for any $w\in D$.
Identifying $E_\mathfrak{h}$ with a subgroup of $N_{\mathbb{A}_\delta}(D)$ 
through this homomorphism,
then $\U(D\rtimes E_\mathfrak{h})$ is a unitary subalgebra of $\mathbb{A}_\delta$
such that $\mathbb{A}_\delta=\U(D\rtimes E_\mathfrak{h})\oplus Y$ as 
$\U(D\rtimes E_\mathfrak{h})$-$\U(D\rtimes E_\mathfrak{h})$-bimodules
with $Y$ being isomorphic to $X$ as $\U D$-$\U D$-bimodules.
\end{prop}

\begin{proof}
Let $X^\prime$ be an indecomposable direct summand of $X$.	
By Proposition \ref{general mod stru},
we can assume that $X^\prime$ is isomorphic to $\U DhD$ for some $h\in G-N_G(D)$
as $\U(D\times D)$-modules.
Set $Q={^h}D\cap D$. 
So $Q$ is a proper subgroup of $D$.
Moreover, we can assume that $Q$ is nontrivial.
For in this case, we can set $\varphi$ to be $\mathrm{id}_D$.
As $\U(D\times D)$-modules,
$X^\prime$ is isomorphism to $\mathrm{Ind}_{\Delta_{\varphi_h}(Q)}^{D\times D}(\U)$,
which is isomorphic to $\U D\mathop\otimes\limits_{\mathcal{O}Q}{_{\varphi_h}}\U D$.
By \cite[Theorem 8.7.1]{L'book 2},
we have $(Q,\bar{b}_Q)^h\leq(D,\bar{b}_D)$.
Since $D$ is abelian, by \cite[Theorem 2]{W14},
we can assume that $h=zg$ for some $g\in N_G(D_\delta)=N_G(D,b_D)$ and some $z\in C_G(Q)$.
Therefore, $\varphi_h=\varphi_g: Q\longrightarrow D,~~u\mapsto g^{-1}ug$.
Furthermore, since $N_G(D_\delta)/C_G(D)=E_\mathfrak{h}\rtimes(PC_G(D)/C_G(D))$,
we can write $g$ as $vx$ for some $v\in P$ and 
some $x\in N_G(D_\delta)$ with $xC_G(D)\in E_\mathfrak{h}$.
Recall from the paragraph 4.4 that there is an element $\mathfrak{a}_x$ in $\mathbb{A}_\delta^\times$
such that $\mathfrak{a}_x(uj)\mathfrak{a}_x^{-1}=(xux^{-1})j$ for any $u\in D$.
Then $X^\prime\mathfrak{a}_x^{-1}$ is still an indecomposable direct summand of $\mathbb{A}_\delta$
as $\U(D\times D)$-modules.
It is routine to check that $X^\prime\mathfrak{a}_x^{-1}$ is isomorphic to 
${\rm Ind}_{\Delta_{\varphi_v}(Q)}^{D\times D}(\U)$.
Equivalently, there is an element $\mathfrak{a}_v$ in $\mathbb{A}_\delta$ such that
for any $(w,w^\prime)\in D\times D$,
$w\mathfrak{a}_v(w^\prime)^{-1}=\mathfrak{a}_v$ if and only if 
$w^\prime=v^{-1}wv$ and $w\in Q$.
At the same time,
by the proof of Proposition \ref{local structure at Q not 1 general},
there is an element $y$ in $N_G(D_\delta)$ such that $yC_G(D)$ is in $E_\mathfrak{h}$
and $C_P({^y}Q)$ is a defect group of the block $b_{{^y}Q}$ of $C_G({^y}Q)$.
For this element $y$,
there is an element $\mathfrak{a}_y\in\mathbb{A}_\delta^\times$ 
such that $\mathfrak{a}_y(wj)\mathfrak{a}_y^{-1}=(ywy^{-1})j$ for any $w\in D$.
Since $E_\mathfrak{h}$ is normal in $N_G(D_\delta)/C_G(D)$,
$(v^{-1}y^{-1}v)C_G(D)$ still belongs to $E_\mathfrak{h}$.
Hence, we have an element $\mathfrak{a}_{v^{-1}y^{-1}v}\in\mathbb{A}_\delta^\times$
such that 
$\mathfrak{a}_{v^{-1}y^{-1}v}(wj)\mathfrak{a}_{v^{-1}y^{-1}v}^{-1}=
(v^{-1}y^{-1}vwv^{-1}yv)j$ for any $w\in D$.
Then we can check that the $\U(D\times D)$-module generated by 
$\mathfrak{a}_y\mathfrak{a}_v\mathfrak{a}_{v^{-1}y^{-1}v}$ is isomorphic to 
$\mathrm{Ind}_{\Delta_{\varphi_v}({^y}Q)}^{D\times D}(\U)$.
This implies that the $\U(D\times D)$-module $\mathbb{A}_\delta$ has a direct summand 
isomorphic to $\mathrm{Ind}_{\Delta_{\varphi_v}({^y}Q)}^{D\times D}(\U)$.
Without loss of generality, we can assume that $C_P(Q)$ is a defect group of 
the block $b_Q$ of $C_G(Q)$.

Borrowing the notation from the paragraph 3.1,
there is an element $a_{v,\delta}$ in $A_\delta^\times\cap\mathbb{A}v$ such that
$a_{v,\delta}(uj)a_{v,\delta}^{-1}=(vuv^{-1})\cdot j$.
Then $\mathfrak{a}_va_{v,\delta}^{-1}$ is in $j\mathbb{A}v^{-1}j$ such that
$w\mathfrak{a}_va_{v,\delta}^{-1}(w^\prime)^{-1}=\mathfrak{a}_va_{v,\delta}^{-1}$
if and only if $w^\prime=w$ and $w\in Q$ for any $(w,w^\prime)\in D\times D$.
In particular, there is an $\U(D\times D)$-submodule of $j\mathbb{A}v^{-1}j$ isomorphic to 
$\mathrm{Ind}_{\Delta(Q)}^{D\times D}(\U)$.
This implies that $(j\mathbb{A}v^{-1}j)(Q)\neq 0$.
By Lemma \ref{local structure at Q not 1 special},
we have $v$ being in $C_P(Q)$.
Therefore, the subgroup $\Delta_{\varphi_h}(Q)$ is equal to
the subgroup $\Delta_{\varphi_x}(Q)$ with $xC_G(D)\in E_\mathfrak{h}$.
This completes the proof of the first statement.

By Proposition \ref{p-nilpotent quot},
$N_{\mathbb{A}_\delta}(D)/(\mathbb{A}_\delta^D)^\times$ is isomorphic to $E_\mathfrak{h}$.
We identify these two groups with each other.
At the same time, we have the following short exact sequence
    \begin{equation*}
	\xymatrix{
		1\ar[r]& k^\times\ar[r] & N_{\mathbb{A}_\delta}(D)/(j+J(\mathbb{A}_\delta^D)) \ar[r] & 
E_\mathfrak{h}\ar[r] & 1.
	}
\end{equation*}
Since $D\rtimes E_\mathfrak{h}$ is a Frobenius group,
we have $H^2(E_\mathfrak{h}; k^\times)=1$.
In particular, the short exact sequence above splits.
So there is a subgroup $\hat{E}_\mathfrak{h}$ of $N_{\mathbb{A}_\delta}(D)$ 
containing $j+J(\mathbb{A}_\delta^D)$ such that
$k^\times\cap\hat{E}_\mathfrak{h}=\{j\}$ and 
$\hat{E}_\mathfrak{h}/(j+J(\mathbb{A}_\delta)^D)\cong E_\mathfrak{h}$.
In particular, we get another short exact sequence
   \begin{equation*}
	\xymatrix{
		1\ar[r]& j+J(\mathbb{A}_\delta^D)\ar[r] &\hat{E}_\mathfrak{h} \ar[r] & 
		E_\mathfrak{h}\ar[r] & 1.
	}
\end{equation*}
Since $E_\mathfrak{h}$ is a $p^\prime$-group,
by \cite[Lemma (45.6)]{Thevenaz},
this sequence also splits.
Therefore, 
we can view $E_\mathfrak{h}$ as a subgroup 
$\{\mathfrak{e}_{\tilde{g}}\mid\tilde{g}\in E_\mathfrak{h}\}$
of $N_{\mathbb{A}_\delta}(D)$ intersecting with 
$(\mathbb{A}_\delta^D)^\times$ trivially such that it and $E_\mathfrak{h}$ act in the same way on $D$. 

As $\U(D\times D)$-modules, by Lemma \ref{ODg mult 1},
$\mathbb{A}_\delta=(\mathop\bigoplus\limits_{\tilde{g}\in E_\mathfrak{h}}
\U D\mathfrak{e}_{\tilde{g}})\oplus X$ for some $\U(D\times D)$-submodule $X$.
Clearly, $\mathop\bigoplus\limits_{\tilde{g}\in E_\mathfrak{h}}\U D\mathfrak{e}_{\tilde{g}}$
is a unitary subalgebra of $\mathbb{A}_\delta$ isomorphic to $\U(D\rtimes E_\mathfrak{h})$.
Then by \cite[Proposition 6.15.2]{L'book 2},
the inclusion map $\U(D\rtimes E_\mathfrak{h})\hookrightarrow\mathbb{A}_\delta$ 
splits as a homomorphism of 
$\U(D\rtimes E_\mathfrak{h})$-$\U(D\rtimes E_\mathfrak{h})$-bimodules.
Therefore, $\mathbb{A}_\delta=\U(D\rtimes E_\mathfrak{h})\oplus Y$ 
for some $\U(D\rtimes E_\mathfrak{h})$-$\U(D\rtimes E_\mathfrak{h})$-bimodule $Y$.
By  the Krull-Schmidt Theorem, 
$Y$ is isomorphic to $X$ as $\U D$-$\U D$-bimodules.
\end{proof}

\vskip2cm

\section{A stable equivalence of Morita type}
\quad\, 
We still use the notation in the last two sections.
Throughout this section, 
we assume that the block is a hyperfocal abelian Frobenius block.
The main result in this section is the construction of 
 a stable equivalence of Morita type between 
$\mathbb{A}_\delta$ and $\U(D\rtimes E_\mathfrak{h})$ under this assumption.
To accomplish this, we need to glue the indecomposable endopermutation $k(D/Q)$-modules $V_Q$ for any $1\neq Q\leq D$ appearing in Proposition \ref{local structure at Q not 1 general} together.
The main technique is the result due to Puig \cite[Proposition 3.6]{P91} 
(see also \cite[Theorem 7.8.2]{L'book 2}).
For convenience, we present this result here.

\begin{lem}{\rm(\cite[Proposition 3.6]{P91})}\label{Puig's result}
Let $Z$ be a finite abelian $p$-group and $E$ a subgroup of $\mathrm{Aut}(Z)$.
Let $\mathcal{C}$ be a nonempty upwardly closed $E$-stable set of subgroups of $Z$,
which means that if $Q\leq R\leq Z$ with $Q$ in $\mathcal{C}$,
then $R$ and $y(Q)$ belong to $\mathcal{C}$ for all $y\in E$.
For any subgroup $Q$ in $\mathcal{C}$,
let $V_Q$ be an indecomposable endopermutation
$k(Z/Q)$-module with vertex $Z/Q$.
Suppose that if $Q,R$ belong to $\mathcal{C}$ such that $Q\leq R$, 
then $\mathrm{Defres}_{Z/R}^{Z/Q}(V_Q)\cong V_R$,
where $Z/R$ is identified with $(Z/Q)/(R/Q)$.
Then there is an indecomposable endopermutation $kZ$-module $V$ with vertex $Z$
such that $\mathrm{Defres}_{Z/Q}^Z(V)\cong V_Q$ for any subgroup $Q$ in $\mathcal{C}$.
In addition, if ${_y}(V_Q)\cong V_{y(Q)}$ as $k(Z/y(Q))$-modules
for all $y\in E$ and all $Q$ in $\mathcal{C}$,
then ${_y}V\cong V$ as $kZ$-modules for all $y\in E$.
\end{lem}

\begin{order}\rm
	In our setting, we can identify $N_G(D_\delta)/C_G(D)$, denoted by $\tilde{N}$, with 
	a subgroup of $\mathrm{Aut}(D)$.
For any $g\in N_G(D_\delta)$,
we denote its image in $\tilde{N}$ by $\tilde{g}$.
	Let $\mathcal{C}$ be the set of subgroups of $D$ consisting of all nontrivial subgroups of $D$.
	Obviously, $\mathcal{C}$ is a nonempty upwardly $\tilde{N}$-stable set.
	Take two subgroups $Q$ and $R$ of $D$ with $Q\leq R$.
	We denote $\mathrm{Defres}_{D/R}^{D/Q}(V_Q)$ by $W_{Q,R}$.
	By definition, $W_{Q,R}$ is up to isomorphism the unique 
	indecomposable endopermutation $k(D/R)$-module 
	with vertex $D/R$ such that $\bar{S}_Q(R)$ is isomorphic to $\mathrm{End}_k(W_{Q,R})$
	as $D/R$-algebras. 
	Here, $\bar{S}_Q$ is $\mathrm{End}_k(V_Q)$.
\end{order}

\begin{lem}\label{compatible of endo mod}
	Keep the notation as above.
	We have $W_{Q,R}\cong V_{R}$ as $k(D/R)$-modules.
	Moreover, for all $\tilde{g}\in\tilde{N}$,
	we have ${_g}V_Q\cong V_{{^g}Q}$ as $kD$-modules. 
\end{lem}

\begin{proof}
	By Proposition \ref{local structure at Q not 1 general},
	we have $\mathbb{A}_\delta(Q)\cong\bar{S}_Q\mathop\otimes\limits_kkD$ and
	$\mathbb{A}_\delta(R)\cong\bar{S}_R\mathop\otimes\limits_kkD$ as $D$-interior algebras.
	Since $D$ is abelian and $Q\leq R$,
	we have $\mathbb{A}_\delta(R)\cong(\mathbb{A}_\delta(Q))(R)$ as $D$-interior algebras.
	It is clear that 
	$(\bar{S}_Q\mathop\otimes\limits_kkD)(R)\cong\bar{S}_Q(R)\mathop\otimes\limits_k(kD)(R)
	\cong\bar{S}_Q(R)\mathop\otimes\limits_kkD$.
	Then we have 
	$\bar{S}_Q(R)\mathop\otimes\limits_kkD\cong\bar{S}_R\mathop\otimes\limits_kkD$ as $D$-interior algebras.
	By the uniqueness of $V_R$ shown in Proposition \ref{local structure at Q not 1 general},
	we have $W_{Q,R}\cong V_R$ as $k(D/R)$-modules 
	since $\bar{S}_Q(R)\cong\mathrm{End}_k(W_{Q,R})$.
	
Take an element $\tilde{g}$ of $\tilde{N}$.
Clearly, $\mathrm{End}_k(_{g}V_Q)\cong{_g}\bar{S}_Q$.
Here, ${_g}\bar{S}_Q$ is the $D/{^g}Q$-interior algebra with 
being equal to $\bar{S}_Q$ as $k$-algebras
and the group homomorphism from $D/{^g}Q$ to $({_g}\bar{S}_Q)^\times$ being just the composition 
of the group homomorphism from $D/{^g}Q$ to $D/Q$ induced by the conjugation action 
of $g$ 
and the structural homomorphism from $D/Q$ to $\bar{S}_Q^\times$.
We use the similar notation for any $D$-interior algebra.
Since $\mathbb{A}_\delta(Q)\cong\bar{S}_Q\mathop\otimes\limits_kkD$ as $D$-interior algebras,
we have 
\begin{align}\label{endoperm g conj 1}
{_g}(\mathbb{A}_\delta(Q))\cong{_g}\bar{S}_Q\mathop\otimes\limits_k{_g}kD
\end{align}
as $D$-interior algebras.
On the other hand, obviously,
we have
 \begin{align}\label{endoperm g conj 2}
{_g}(\mathbb{A}_\delta(Q))\cong({_g}\mathbb{A}_\delta)({^g}Q)
\end{align} 
as $D$-interior algebras.
For $\tilde{g}\in\tilde{N}$,
it can be written as $\tilde{u}\tilde{x}$ for some $u\in P$ and 
some $x\in N_G(D_\delta)$ with $\tilde{x}\in E_\mathfrak{h}$.
Borrowing the notation from the paragraph 3.1 and 
Proposition \ref{special mod stru},
there are elements $a_{u,\delta}\in A_\delta^\times\cap j\mathbb{A}uj$ and
$\mathfrak{e}_x\in N_{\mathbb{A}_\delta}(D)$ such that
$a_{u,\delta}(wj)a_{u,\delta}^{-1}=(uwu^{-1})j$ and 
$\mathfrak{e}_x(wj)\mathfrak{e}_x^{-1}=(xwx^{-1})j$ for any $w\in D$.
Set $a_{g,\delta}=a_{u,\delta}\mathfrak{e}_x$.
Then $a_{g,\delta}$ is in $N_{A_\delta}(D)\cap j\mathbb{A}uj$ and
$a_{g,\delta}(wj)a_{g,\delta}^{-1}=(gwg^{-1})j$ for any $w\in D$.
In particular, the conjugation action of $a_{g,\delta}^{-1}$ on $\mathbb{A}_\delta$ induces
a $D$-interior algebra isomorphism $\mathbb{A}_\delta\cong{_g}\mathbb{A}_\delta$.
Therefore, 
\begin{align}\label{endoperm g conj 3}
({_g}\mathbb{A}_\delta)({^g}Q)\cong\mathbb{A}_\delta({^g}Q)\cong\bar{S}_{{^g}Q}\mathop\otimes\limits_kkD
\end{align}
as $D$-interior algebras.
By the isomorphisms (\ref{endoperm g conj 1}) and (\ref{endoperm g conj 2}) 
and (\ref{endoperm g conj 3}), we can get
\begin{align}\label{endoperm g conj 4}
\bar{S}_{{^g}Q}\mathop\otimes\limits_kkD\cong{_g}\bar{S}_Q\mathop\otimes\limits_k{_g}kD
\end{align}
as $D$-interior algebras.
Clearly, the conjugation action of $g^{-1}$ induces a $D$-interior algebra isomorphism $kD\cong{_g}kD$.
In conclusion, we obtain a $D$-interior algebra isomorphism
$\bar{S}_{{^g}Q}\mathop\otimes\limits_kkD\cong{_g}\bar{S}_Q\mathop\otimes\limits_kkD$.
By the uniqueness of $V_{{^g}Q}$ and the isomorphism $\mathrm{End}_k({_g}V_Q)\cong{_g}\bar{S}_Q$,
we can get ${_g}V_Q\cong V_{{^g}Q}$ as $kD$-modules.
\end{proof}

\begin{order}\rm
By Lemmas \ref{Puig's result} and \ref{compatible of endo mod},
we can get an indecomposable endopermutation $kD$-module $\bar{V}$ with vertex $D$ such that
$(\mathrm{End}_k(\bar{V}))(Q)\cong\bar{S}_Q$ as $D$-interior algebras
for any nontrivial subgroup $Q$ of $D$ and ${_g}\bar{V}\cong\bar{V}$ as $kD$-modules 
for any $\tilde{g}\in\tilde{N}$.
Furthermore, by \cite[Theorem 14.2]{Thevenaz07}, 
there is an indecomposable endopermutation $\U D$-module $V$ with vertex $D$ such that 
$k\mathop\otimes\limits_\mathcal{O}V\cong\bar{V}$.
Setting $S=\mathrm{End}_\mathcal{O}(V)$, 
then $S(Q)\cong\bar{S}_Q$ as $D$-interior algebras for any nontrivial subgroup $Q$ of $D$.
In the following, we will show that $V$ can be chose to be $\tilde{N}$-stable, namely,
${_g}V\cong V$ as $\U D$-modules for any $\tilde{g}\in\tilde{N}$.
\end{order}

\begin{lem}\label{N-stable of V}
Keep the notation as above.
Then the indecomposable endopermutation $\U D$-module $V$ can be chosen such that
$\mathrm{det}(w\cdot 1_S)=1$ for any $w\in D$.
Here, $\mathrm{det}$ means the determinant map from $S$ to $\U$.
Moreover, for this module $V$,
we have ${_g}V\cong V$ as $\U D$-modules for any $\tilde{g}\in\tilde{N}$.
In particular, ${_u}V\cong V$ as $kD$-modules for any $u\in P$.
\end{lem}

\begin{proof}
Take an element $\tilde{g}$ in $\tilde{N}$.
Denote the structural homomorphism from $D$ to $S^\times$ by $\phi$.
Clearly, $\mathrm{End}_\mathcal{O}({_g}V)\cong{_g}S$ and
the the structural homomorphism from $D$ to $({_g}S)^\times$, denoted by ${_g}\phi$,
mapping $w\in D$ to $\phi(gwg^{-1})$.
Since $k\mathop\otimes\limits_\mathcal{O}{_g}V\cong{_g}\bar{V}$,
we get $k\mathop\otimes\limits_\mathcal{O}{_g}V\cong k\mathop\otimes\limits_\mathcal{O}V$ as $kD$-modules.
By \cite[Proposition 7.3.12]{L'book 2},
$S$ is isomorphic to ${_g}S$ as $D$-algebras.
Denote this isomorphism by $\psi$.
Obviously, $\psi$ is an $\U$-automorphism of $S$.
By the Skolem-Noether Theorem,
$\psi$ is in fact an inner automorphism. 
In particular, $\psi$ preserves the determinants.
Set $\mathrm{rank}_\mathcal{O}(V)=n$. 
By \cite[Proposition 7.3.10 (i)]{L'book 2}, $n$ is coprime to $p$.
Then the statements in this lemma follows from \cite[Proposition (21.5)]{Thevenaz}.
\end{proof}

\begin{order}\rm
Now we fix an indecomposable endopermutation $\U D$-module $V$ with vertex $D$ such that
setting $S=\mathrm{End}_\mathcal{O}(V)$, $S(Q)\cong\bar{S}_Q$ 
for any nontrivial subgroup $Q$ of $D$ and
$\mathrm{det}(w\cdot 1_S)=1$ for any $w\in D$.
Then by Lemma \ref{N-stable of V}, 
we have ${_g}V\cong V$ as $\U D$-modules for any $\tilde{g}\in\tilde{N}$.
In particular, $V$ is $E_\mathfrak{h}$-stable.
By \cite[Corollary 7.8.4]{L'book 2}, 
$V$ can be extended to an indecomposable $\U(D\rtimes E_\mathfrak{h})$-module $U$
such that $U$ has a vertex $D$ and a source $V$ 
and $\mathrm{Res}_D^{D\rtimes E_\mathfrak{h}}(U)\cong V$.
\end{order}

\begin{order}\rm
Set $\tilde{G}=D\rtimes\tilde{N}$ and $L=D\rtimes E_\mathfrak{h}$.
Then $L\unlhd\tilde{G}$ and $D$ is a Sylow $p$-subgroup of $L$.
Therefore, $P$ can act on $L$ by conjugation through the canonical map
$P\longrightarrow P/C_P(D)$.
In this setting, ${_v}U$ makes sense and is still an indecomposable $\U L$-module
for any $v\in P$.
In fact, the module $U$ can be chosen in such a way that $U$ is $P$-stable.
\end{order}

\begin{lem}\label{P-stable}
Keep the notation as above.
There is an indecomposable  $P$-stable $\U L$-module $W$
such that $\mathrm{Res}_D^L(W)\cong V$
and  has a vertex $D$ and a source $V$.
\end{lem}

\begin{proof}
Since $V$ is $\tilde{N}$-stable,
by \cite[Theorem 7.8.3]{L'book 2},
there is an indecomposable $\mathcal{O}\tilde{G}$-module $\tilde{U}$
such that $\mathrm{Res}^{\tilde{G}}_D(\tilde{U})\cong V\oplus V^\prime$,
where every indecomposable direct summand of $V^\prime$ has a proper subgroup of $D$ as a vertex.
So there is at least one direct summand of $\mathrm{Res}^{\tilde{G}}_L(\tilde{U})$ 
of which $D$ is the vertex since $D$ is a normal Sylow $p$-subgroup of $L$.
Suppose that there are two direct summands $W_1$ and $W_2$ of 
${\rm Res}_L^{\tilde{G}}(\tilde{U})$ with vertex $D$.
Then ${\rm Res}_D^L(W_1)\oplus{\rm Res}_D^L(W_2)$ is a direct summand of 
${\rm Res}_D^{\tilde{G}}(\tilde{U})$.
Therefore, one of ${\rm Res}_D^L(W_1)$ and ${\rm Res}_D^L(W_2)$ has to be isomorphic 
to a direct summand of $V^\prime$.
We can assume that ${\rm Res}_D^L(W_1)$ is isomorphic to a direct summand of $V^\prime$.
Since $D$ is a Sylow $p$-subgroup of $L$,
$W_1$ is isomorphic to a direct summand of ${\rm Ind}_D^L{\rm Res}_D^L(W_1)$,
which implies that $W_1$ is isomorphic to a direct summand of ${\rm Ind}_D^L(V^\prime)$.
This is impossible since $W_1$ has a vertex $D$.
Hence, there is exactly one direct summand of ${\rm Res}_L^{\tilde{G}}(\tilde{U})$.
We denote this direct summand by $W$.
So $V$ is isomorphic to a direct summand of ${\rm Res}_D^L(W)$.
Since $W$ is isomorphic to a direct summand of ${\rm Ind}_D^L{\rm Res}_D^L(W)$,
we can get that $W$ is isomorphic to a direct summand of ${\rm Ind}_D^L(V)$.
Since $V$ is $\tilde{G}$-stable, by Mackey's formula,
$\mathrm{Res}_D^L(\mathrm{Ind}_D^L(V))$ is isomorphic to a direct sum of 
$|E_\mathfrak{h}|$-copies of $V$.
In conclusion, 
${\rm Res}_D^L(W)$ is isomorphic to a direct sum of some copies of $V$.
But the multiplicity of $V$ in the decomposition of ${\rm Res}_D^{\tilde{G}}(\tilde{U})$ 
is just one.
Hence,  ${\rm Res}_D^L(W)$ isomorphic to $V$.
The fact that $W$ is $P$-stable is due to the uniqueness of $W$.
\end{proof}

\begin{order}\rm
We still denote $W$ by $V$ whenever this does not introduce any confusion.
Now $S=\mathrm{End}_\mathcal{O}(V)$ is a $P$-stable primitive $L$-interior algebra.
This means that for any $u\in P$ there exists an element $s_u$ in $S^\times$ such that
$s_u(x\cdot\mathrm{v})=(uxu^{-1})\cdot s_u(\mathrm{v})$ for any $x\in L$ and any $\mathrm{v}\in V$.
In particular $s_u$ belongs to $S^{L}$ when $u$ is in $C_P(D)$.
Since taking the inverse gives a group isomorphism from $S^\times$ to $(S^\circ)^\times$,
the opposite algebra $S^\circ$ of $S$ is also a $P$-stable primitive $L$-algebra. 
\end{order}

\begin{order}\rm
Recall from Proposition \ref{special mod stru} that $\U L$ can be viewed as
a subalgebra of $\mathbb{A}_\delta$.
So we can consider the $L$-interior algebra $S^\circ\mathop\otimes\limits_\mathcal{O}\mathbb{A}_\delta$.
The following result gives the $\U(D\times D)$-module structure of 
$S^\circ\mathop\otimes\limits_\mathcal{O}\mathbb{A}_\delta$,
which is similar to the one of $\mathbb{A}_\delta$ in Proposition \ref{special mod stru}.
\end{order}

\begin{lem}\label{mod stru of S^oXbbA}
There is a decomposition of $S^\circ\mathop\otimes\limits_\mathcal{O}\mathbb{A}_\delta$ 
as an $\U(D\times D)$-module
\begin{align}\label{dec of S^oXbbA}
S^\circ\mathop\otimes\limits_\mathcal{O}\mathbb{A}_\delta\cong
(\mathop\bigoplus\limits_{\tilde{g}\in E_\mathfrak{h}}\U Dg)\oplus\mathbb{X},
\end{align}
where $\mathbb{X}$ is isomorphic to a direct sum of modules of the form 
$\mathrm{Ind}_{\Delta_{\tilde{h}}(Q)}^{D\times D}(\U)$ 
for some proper subgroup $Q$ of $D$ and some $\tilde{h}\in E_\mathfrak{h}$.
Here, we denote the subgroup $\{(u,g^{-1}ug)\mid u\in Q\}$ of $D\times D$ by $\Delta_{\tilde{g}}(Q)$
for any $\tilde{g}\in E_\mathfrak{h}$ and any subgroup $Q$ of $D$.
In particular, 
$S^\circ\mathop\otimes\limits_\mathcal{O}\mathbb{A}_\delta$ has a finite $D\times D$-stable $\U$-basis
such that $S^\circ\mathop\otimes\limits_\mathcal{O}\mathbb{A}_\delta$ is projective as a left $\U D$-module
and as a right $\U D$-module.
\end{lem}

\begin{proof}
By Proposition \ref{special mod stru}, as $\U(D\times D)$-modules,
$$S^\circ\mathop\otimes\limits_\mathcal{O}\mathbb{A}_\delta\cong
(\mathop\bigoplus\limits_{\tilde{g}\in E_\mathfrak{h}}
\mathrm{Ind}_{\Delta_{\tilde{g}}(D)}^{D\times D}
(\mathrm{Res}_{\Delta_{\tilde{g}}(D)}^{D\times D}(S^\circ)))\oplus 
(S^\circ\mathop\otimes\limits_\mathcal{O}X),$$
where $S^\circ\mathop\otimes\limits_\mathcal{O}X$ is isomorphic to 
a direct sum of indecomposable modules of the form 
$\mathrm{Ind}_{\Delta_{\tilde{h}}(Q)}^{D\times D}
(\mathrm{Res}_{\Delta_{\tilde{h}}(Q)}^{D\times D}(S^\circ))$ 
for some proper subgroup $Q$ of $D$ and some $\tilde{h}\in E_\mathfrak{h}$.
For any $\tilde{g}\in E_\mathfrak{h}$,
we denote by $s^\circ_{\tilde{g}}$ its image in $(S^\circ)^\times$.
Let $\Omega^\circ$ be a $D$-stable $\U$-basis of $S^\circ$ 
containing $1_{S^\circ}$ as the unique $D$-invariant element.
Then $\Omega^\circ\cdot(s^\circ_{\tilde{g}})^{-1}$ is a 
$\Delta_{\tilde{g}}(D)$-stable $\U$-basis of $S^\circ$
containing $(s^\circ_{\tilde{g}})^{-1}$ as the unique $\Delta_{\tilde{g}}(D)$-invariant element.
Then the  statements in this lemma follow from the existence of this $\U$-basis
of $S^\circ$ and the decomposition above.
\end{proof}

\begin{order}\rm
Let $e$ be a primitive idempotent of $(S^\circ\mathop\otimes\limits_\mathcal{O}\mathbb{A}_\delta)^L$
with $\mathrm{Br}_D(e)\neq 0$.
Set $\mathbb{A}^\prime=e(S^\circ\mathop\otimes\limits_\mathcal{O}\mathbb{A}_\delta)e$.
So $\mathbb{A}^\prime$ is a primitive $L$-interior algebra.
In fact, $e$ is still primitive in $(\mathbb{A}^\prime)^D$.
Indeed, by \cite[Corollary 5.5.9]{L'book 1},
every point of $D$ on $\mathbb{A}^\prime$ is local.
Since $\mathbb{A}_\delta$ is a primitive $D$-interior algebra with $\mathbb{A}_\delta(D)\neq 0$,
by \cite[Theorem 7.4.1 (i)]{L'book 2},
there is a unique local point of $D$ on 
$S^\circ\mathop\otimes\limits_\mathcal{O}\mathbb{A}_\delta$ with multiplicity $1$.
Therefore, $e$ remains primitive in $(\mathbb{A}^\prime)^D$.
In the following, we collect some properties of the $L$-interior algebra $\mathbb{A}^\prime$.
\end{order}

\begin{prop}\label{property of bbA'}
For the algebra $\mathbb{A}^\prime$,
the following hold.

(i) $\mathbb{A}^\prime$ is a primitive relatively $\U D$-separable $D$-interior algebra.

(ii) The $\mathbb{A}^\prime$-$\mathbb{A}_\delta$-bimodule 
$e\cdot(V^*\mathop\otimes\limits_\mathcal{O}\mathbb{A}_\delta)$ induces a Morita equivalence 
between $\mathbb{A}^\prime$ and $\mathbb{A}_\delta$.
Here, the right $\mathbb{A}_\delta$-module structure of $e\cdot(V^*\mathop\otimes\limits_\mathcal{O}\mathbb{A}_\delta)$ is given by
the right multiplication in $\mathbb{A}_\delta$.

(iii) For any nontrivial subgroup $Q$ of $D$,
$\mathbb{A}^\prime(Q)\cong kD$ as $D$-interior algebras.
\end{prop}

\begin{proof}
Since the rank of $V$ is coprime to $p$,
by \cite[Proposition 5.1.18]{L'book 1} and Corollary \ref{separable and Morita equi},
$S^\circ\mathop\otimes\limits_\mathcal{O}\mathbb{A}_\delta$ is 
a relatively $\U D$-separable $D$-interior algebra.
Clearly, the $\U$-rank of $S^\circ\mathop\otimes\limits_\mathcal{O}\mathbb{A}_\delta$ equals
$\mathrm{rank}_\mathcal{O}(S^\circ)\cdot\mathrm{rank}_\mathcal{O}(\mathbb{A}_\delta)$.
By Corollary \ref{dim of bbA and simple mod},
$\frac{\mathrm{rank}_\mathcal{O}(S^\circ\mathop\otimes\limits_\mathcal{O}\mathbb{A}_\delta)}{|D|}$ 
is coprime to $p$.
By Lemma \ref{mod stru of S^oXbbA} and the argument in the paragraph 5.12,
applying Lemma \ref{a general lemma} to $S^\circ\mathop\otimes\limits_\mathcal{O}\mathbb{A}_\delta$,
we can get the statements (i) and (ii).

For any nontrivial subgroup $Q$ of $D$,
by \cite[Proposition 5.6]{P-88},
there is a $D$-interior algebra isomorphism
$(S^\circ\mathop\otimes\limits_\mathcal{O}\mathbb{A}_\delta)(Q)\cong S^\circ(Q)\mathop\otimes\limits_k\mathbb{A}_\delta(Q)$,
which is isomorphic to $(S(Q))^\circ\mathop\otimes\limits_k\mathbb{A}_\delta(Q)$.
By Proposition \ref{local structure at Q not 1 general},
$\mathbb{A}_\delta(Q)$ is isomorphic to $\bar{S}_Q\mathop\otimes\limits_kkD$ as $D$-interior algebras.
Note that $S(Q)\cong\bar{S}_Q$ as $D$-interior algebras.
Therefore, $(S^\circ\mathop\otimes\limits_\mathcal{O}\mathbb{A}_\delta)(Q)$
is isomorphic to $\bar{S}_Q^\circ\mathop\otimes\limits_k\bar{S}_Q\mathop\otimes\limits_kkD$ as $D$-interior algebras.
Since $S^\circ\mathop\otimes\limits_\mathcal{O}\mathbb{A}_\delta$ has a finite $D\times D$-stable $\U$-basis
by Lemma \ref{mod stru of S^oXbbA},
we have $((S^\circ\mathop\otimes\limits_\mathcal{O}\mathbb{A}_\delta)(Q))(D)\cong
(S^\circ\mathop\otimes\limits_\mathcal{O}\mathbb{A}_\delta)(D)$.
This implies that $\mathrm{Br}_Q(e)$ belongs to a local point of $D$ on
$(S^\circ\mathop\otimes\limits_\mathcal{O}\mathbb{A}_\delta)(Q)$,
which is isomorphic to $\bar{S}_Q^\circ\mathop\otimes\limits_k\bar{S}_Q\mathop\otimes\limits_kkD$.
Then the statement (iii) follows from \cite[Proposition 7.3.10 (iv)]{L'book 2}.
\end{proof}

Now we can state the main result in this section and give a proof of it.

\begin{thm}\label{MT}
Assume that the block $b$ is a hyperfocal abelian Frobenius block.
Then there is a stable equivalence of Morita type between 
$\mathbb{A}_\delta$ and $\U L$.
Consequently,  there is a stable equivalence of Morita type between 
$\mathbb{A}$ and $\U L$.
\end{thm}

\begin{proof}
By \cite[Proposition 6.15.2]{L'book 2},
the structural map $\U L\longrightarrow\mathbb{A}^\prime$
splits as $\U L$-$\U L$-bimodules.
We identify $L$ with a subgroup of $(\mathbb{A}^\prime)^\times$.
Therefore, as $\U L$-$\U L$-bimodules,
 $\mathbb{A}^\prime=\U L\oplus\mathbb{Y}$ for some 
 $\U L$-$\U L$-subbimodules $\mathbb{Y}$ of $\mathbb{A}^\prime$.
 In particular, $\mathbb{Y}$ is a direct summand of 
 $\mathbb{A}^\prime$ as $\U(D\times D)$-modules.
 It is easy to check that $C_L(Q)=D$ for any nontrivial subgroup $Q$ of $D$.
 Then $(\U L)(Q)\cong kD$ for any nontrivial subgroup $Q$ of $D$.
 By Proposition \ref{property of bbA'} (iii),
 we have $\mathbb{Y}(Q)=0$ for any nontrivial subgroup $Q$ of $D$.
 
 On the other hand, as $\U(D\times D)$-modules,
 $\mathbb{A}^\prime$ is a direct summand of $S^\circ\mathop\otimes\limits_{\mathcal{O}}\mathbb{A}_\delta$.
 Hence, $\mathbb{Y}$ is also a direct summand of $S^\circ\mathop\otimes\limits_{\mathcal{O}}\mathbb{A}_\delta$ as $\U(D\times D)$-modules.
 By the $\U(D\times D)$-module decomposition (\ref{dec of S^oXbbA}),
$\mathbb{Y}$ has to be isomorphic to a direct summand of $\mathbb{X}$
since $\U L$ is isomorphic to $\mathop\bigoplus\limits_{\tilde{g}\in E_\mathfrak{h}}\U Dg$
as $\U(D\times D)$-modules.
Therefore, every direct summand of $\mathbb{Y}$ has the form
$\mathrm{Ind}_{\Delta_{\tilde{h}}(Q)}^{D\times D}(\U)$ 
for some proper subgroup $Q$ of $D$ and some $\tilde{h}\in E_\mathfrak{h}$.
Suppose that $\mathbb{Y}$ has a direct summand $U$ isomorphic to 
$\mathrm{Ind}_{\Delta_{\tilde{h}}(Q)}^{D\times D}(\U)$ with a nontrivial proper subgroup $Q$ of $D$.
Then there is an element  $y^\prime_U$ in $\mathbb{Y}$ such that
for any $(w_1,w_2)\in D\times D$,
$w_1y^\prime_Uw_2^{-1}=y^\prime_U$ if and only if $w_1\in Q$ and $w_2=h^{-1}w_1h$ 
and $U=\U Dy^\prime_UD$.
Then we have for any $(w_1,w_2)\in D\times D$,
$w_1y^\prime_U\tilde{h}^{-1}w_2^{-1}=y^\prime_U\tilde{h}^{-1}$ 
if and only if $w_1\in Q$ and $w_1=w_2$ and
then $\U Dy^\prime_U\tilde{h}D$ is isomorphic to $\mathrm{Ind}_{\Delta(Q)}^{D\times D}(\U)$. 
Since $\mathbb{Y}$ is an $\U L$-$\U L$-subbimodule of $\mathbb{A}^\prime$,
$\U Dy^\prime_U\tilde{h}D$ is contained in $\mathbb{Y}$ and 
also a direct summand of $\mathbb{Y}$ as $\U(D\times D)$-modules.
But it is easy to check that $(\mathrm{Ind}_{\Delta(Q)}^{D\times D}(\U))(Q)\neq 0$.
This is impossible since $\mathbb{Y}(Q)=0$ when $Q$ is nontrivial.
In particular, as an $\U(D\times D)$-module, $\mathbb{Y}$ is projective.

Clearly, $\U L$ is a relatively $\U D$-separable $D$-interior algebra.
So $\U L$ is isomorphic to a direct summand of $\U L\mathop\otimes\limits_{\mathcal{O}D}\U L$
as $\U L$-$\U L$-bimodules.
Obviously, as  $\U L$-$\U L$-bimodules,
$\mathbb{Y}$ is isomorphic to $\U L\mathop\otimes\limits_{\mathcal{O}L}\mathbb{Y}\mathop\otimes\limits_{\mathcal{O}L}\U L$.
Therefore, as  $\U L$-$\U L$-bimodules,
$\mathbb{Y}$ is isomorphic to a direct summand of 
$\U L\mathop\otimes\limits_{\mathcal{O}D}\mathbb{Y}\mathop\otimes\limits_{\mathcal{O}D}\U L$.
Since $\mathbb{Y}$ is a projective $\U D$-$\U D$-bimodule,
$\mathbb{Y}$ is also a projective $\U L$-$\U L$-bimodule.
Denote the $\mathbb{A}^\prime$-$\U L$ bimodule $\mathbb{A}^\prime$ by $\mathbb{A}^\prime_{\mathcal{O}L}$ and
the $\U L$-$\mathbb{A}^\prime$-bimodule $\mathbb{A}^\prime$ by
${_{\mathcal{O}L}}\mathbb{A}^\prime$.
Then by \cite[Proposition 4.14.12]{L'book 1},
the bimodules $\mathbb{A}^\prime_{\mathcal{O}L}$ and ${_{\mathcal{O}L}}\mathbb{A}^\prime$
induce a stable equivalence of Mortia type between $\mathbb{A}^\prime$ and $\U L$.
By Proposition \ref{property of bbA'} (ii),
the $\mathbb{A}_\delta$-$\U L$-bimodule 
$(V\mathop\otimes\limits_\mathcal{O}\mathbb{A}_\delta)\cdot e\cong
((V\mathop\otimes\limits_\mathcal{O}\mathbb{A}_\delta)\cdot e) \mathop\otimes\limits_{\mathbb{A}^\prime}\mathbb{A}^\prime_{\mathcal{O}L}$
and the $\U L$-$\mathbb{A}_\delta$-bimodule $e\cdot(V^*\mathop\otimes\limits_\mathcal{O}\mathbb{A}_\delta)\cong
{_{\mathcal{O}L}}\mathbb{A}^\prime\mathop\otimes\limits_{\mathbb{A}^\prime}
(e\cdot(V^*\mathop\otimes\limits_\mathcal{O}\mathbb{A}_\delta))$
induce a stable equivalence of Mortia type between $\mathbb{A}_\delta$ and $\U L$.
Moreover, by Corollary \ref{separable and Morita equi},
the $\mathbb{A}$-$\U L$-bimodule
$\mathbb{A}j\mathop\otimes\limits_{\mathbb{A}_\delta}((V\mathop\otimes\limits_\mathcal{O}\mathbb{A}_\delta)\cdot e)$ and
the $\U L$-$\mathbb{A}$-bimodule
$e\cdot(V^*\mathop\otimes\limits_\mathcal{O}\mathbb{A}_\delta)\mathop\otimes\limits_{\mathbb{A}_\delta}j\mathbb{A}$
induce a stable equivalence of Mortia type between $\mathbb{A}$ and $\U L$.
\end{proof}

\begin{cor}\label{nonsingular of cartan matrix}
The Cartan matrix of the hyperfocal subalgebra $\mathbb{A}$ is nonsingular,
the absolute value of which determinant equals $|D|$
when the block $b$ is a hyperfocal abelian Frobenius block.
\end{cor}

\begin{proof}
This can be deduced easily from \cite[Proposition 4.14.13]{L'book 1} and
the fact that the Cartan matrix of $\U(D\rtimes E_{\mathfrak{h}})$ has determinant $|D|$.
\end{proof}

\begin{order}\rm
Keep the notation as above.
Denote the $\mathbb{A}$-$\U L$-bimodule 
$\mathbb{A}j\mathop\otimes\limits_{\mathbb{A}_\delta}((V\mathop\otimes\limits_{\mathcal{O}}\mathbb{A}_\delta)\cdot e)$
by $M_{\mathrm{st}}$.
Then the $\mathbb{A}$-$\U L$-bimodule $M_{\mathrm{st}}$ is indecomposable.
Indeed, it suffices to show that the $\mathbb{A}_\delta$-$\U L$-bimodule
$(V\mathop\otimes\limits_{\mathcal{O}}\mathbb{A}_\delta)\cdot e$ is indecomposable
since the indecomposabilitiy is preserved under a Mortia equivalence.
It is easy to check that the endomorphism algebra of the $\mathbb{A}_\delta$-$\U L$-bimodule
$(V\mathop\otimes\limits_{\mathcal{O}}\mathbb{A}_\delta)\cdot e$ is isomorphic to
$(\mathbb{A}^\prime)^L$, which is a local algebra.
Since $P$ can act on $E_\mathfrak{h}$ through the canonical homomorphism $P\longrightarrow P/C_P(D)$,
it makes sense to define the twisted $\mathbb{A}$-$\U L$-module ${_{(u,u)}}M_{{\rm st}}$ 
of $M_{{\rm st}}$ (see the paragraph 4.9) and then
consider the $\Delta(P)$-equivariant property of the bimodule $M_{\mathrm{st}}$
which is a key property to get the information about source algebras by Clifford Theory.
Here, $\Delta(P)=\{(u,u)\mid u\in P\}$ is the diagonal subgroup of $P\times P$
and the $\Delta(P)$-equivariant property means that
$M_{{\rm st}}$ is isomorphic to ${_{(u,u)}}M_{{\rm st}}$ as $\mathbb{A}$-$\U L$-bimodule
for any $(u,u)\in\Delta(P)$.
The following is an investigation of this issue.
\end{order}

\begin{order}\rm
	Borrowing the notation from Proposition \ref{special mod stru},
there is a unitary subalgebra of $\mathbb{A}_\delta$ isomorphic to $\U L$ as $D$-interior algebras.	
More explicitly, there is a subgroup 
$\{\mathfrak{e}_{\tilde{g}}\mid\tilde{g}\in E_\mathfrak{h}\}$ of $N_{\mathbb{A}_\delta}(D)$
such that it intersects with $(\mathbb{A}_\delta^D)^\times$ trivially and
sending $\mathfrak{e}_{\tilde{g}}$ to $\tilde{g}$ gives an isomorphism
between it and $E_\mathfrak{h}$, which induces the same conjugation actions on $D$.
We denote this subgroup of $N_{\mathbb{A}_\delta}(D)$ by $\mathfrak{E}_\mathfrak{h}$.
At the same time, there is an element $a_{u,\delta}$ in $N_{A_\delta}(D)$ such that
$a_{u,\delta}(vj)a_{u,\delta}^{-1}=(uvu^{-1})j$ 
and $\mathbb{A}_\delta^{a_{u,\delta}}=\mathbb{A}_\delta$ for any $u\in P$ and any $v\in D$.
Since the $L$-interior algebra structure of $\mathbb{A}_\delta$ is essentially given by
the group isomorphism between $\mathfrak{E}_\mathfrak{h}$ and $E_\mathfrak{h}$,
the element $a_{u,\delta}$ should be chosen carefully such that
$\mathfrak{E}_\mathfrak{h}$ is stabilized under its conjugation action on $\mathbb{A}_\delta$ 
and then the induced action on $\mathfrak{E}_\mathfrak{h}$ is compatible with 
the conjugation action of $u$ on $E_\mathfrak{h}$ for any $u\in P$.
Explicitly, we have the following lemma.
\end{order}

\begin{lem}\label{action P on E_h}
For any $u\in P$,
we can choose an element $a_{u,\delta}$ in $N_{A_\delta}(D)$ such that
${^{a_{u,\delta}}}\mathfrak{E}_\mathfrak{h}=\mathfrak{E}_\mathfrak{h}$.
Furthermore, for any $\tilde{g}\in E_\mathfrak{h}$,
${^{a_{u,\delta}}}\mathfrak{e}_{\tilde{g}}=\mathfrak{e}_{{^u}\tilde{g}}$.
In particular, when the conjugation action of $u$ on $E_\mathfrak{h}$ 
induced by the canonical map $P\longrightarrow P/C_P(D)$ is trivial,
we have ${^{a_{u,\delta}}}\mathfrak{e}_{\tilde{g}}=\mathfrak{e}_{\tilde{g}}$
for any $\tilde{g}\in E_\mathfrak{h}$.
\end{lem}

\begin{proof}
Recall that we denote $N_G(D_\delta)/C_G(D)$ by $\tilde{N}$.
By the paragraph 2.12,
we can get a short exact sequence
\begin{equation*}
	\xymatrix{
		1\ar[r]& k^\times \ar[r] &N_{A_\delta}(D)/(j+J(A_\delta^D))\ar[r] & 
		\tilde{N}\ar[r] & 1.
	}
\end{equation*}
since $D$ is abelian.
Denote by $\pi$ the sujective homomorphism of the above sequence.
Note that $\tilde{N}$ is a $p$-nilpotent group 
with Sylow $p$-subgroup $\tilde{P}$ and normal $p$-complement $E_\mathfrak{h}$.
Here, $\tilde{P}$ denotes $PC_G(D)/C_G(D)$ which is isomorphic to $P/C_P(D)$.
Obviously, the second cohomology group $H^2(\tilde{P};k^\times)$ of $\tilde{P}$ is trivial.
Since $D\rtimes E_\mathfrak{h}$ is a Frobenius group,
the second cohomology group $H^2(E_\mathfrak{h};k^\times)$ of $E_\mathfrak{h}$ is also trivial.
In conclusion, the second cohomology group $H^2(\tilde{N};k^\times)$ of $\tilde{N}$ is trivial.
This means that the short exact sequence above splits.
So there is a subgroup $\hat{F}$ of $N_{A_\delta}(D)$ such that
it contains $j+J(A_\delta^D)$ and intersects with $k^\times$ trivially
and $\hat{F}/(j+J(A_\delta^D))$ is isomorphic to $\tilde{N}$.
So for any $v\in P$,
the element $a_{v,\delta}$ can be chosen such that
the set $P(j+J(A_\delta^D))=\{a_{v,\delta}x\mid v\in P,x\in j+J(A_\delta^D)\}$ 
is a Sylow $p$-subgroup of $\hat{F}/(j+J(A_\delta^D))$.
Denote by $\hat{F}_\mathfrak{h}$ a normal subgroup of $\hat{F}$ containing $j+J(A_\delta^D)$
such that $\hat{F}_\mathfrak{h}/(j+J(A_\delta^D))$ is a normal $p$-complement of 
$\hat{F}/(j+J(A_\delta^D))$.
In particular, ${^{a_{v,\delta}}}\hat{F}_\mathfrak{h}=\hat{F}_\mathfrak{h}$ for any $v\in P$
and $\hat{F}_\mathfrak{h}/(j+J(A_\delta^D))$ is isomorphic to $E_\mathfrak{h}$.
Hence, this yields another short exact sequence
  \begin{equation*}
	\xymatrix{
		1\ar[r]& j+J(A_\delta^D)\ar[r] &\hat{F}_\mathfrak{h} \ar[r] & 
		E_\mathfrak{h}\ar[r] & 1,
	}
\end{equation*}
which also splits by \cite[Lemma (45.6)]{Thevenaz} since $E_\mathfrak{h}$ is a $p^\prime$-group.
So there is a subgroup 
$\mathfrak{F}_\mathfrak{h}=\{{\rm f}_{\tilde{g}}\mid\tilde{g}\in E_\mathfrak{h}\}$ 
of $\hat{F}_\mathfrak{h}$ such that
it is isomorphic to $E_\mathfrak{h}$ and intersects with $j+J(A_\delta^D)$ trivially.

At the same time, it is clear that
$\pi(\mathfrak{E}_\mathfrak{h}\cdot(j+J(A_\delta^D)))=E_\mathfrak{h}$.
Then for any $\tilde{g}\in E_\mathfrak{h}$, we have
$\mathfrak{e}_{\tilde{g}}=\lambda(\tilde{g}){\rm f}_{\tilde{g}}x_{\tilde{g}}$
for a unique element $\lambda(\tilde{g})$ in $k^\times$ and 
a unique element $x_{\tilde{g}}$ in $j+J(A_\delta^D)$.
This defines a group homomorphism $\lambda$ from $E_\mathfrak{h}$ to $k^\times$
by sending $\tilde{g}$ to $\lambda(\tilde{g})$.
Adjusting by this group homomorphism $\lambda$,
we can assume that $E_\mathfrak{h}$ is contained in $\hat{F}_\mathfrak{h}$.
Therefore, $\mathfrak{E}_\mathfrak{h}$ is a complement of $j+J(A_\delta^D)$ 
in $\hat{F}_\mathfrak{h}$. 
Now fix an element $u$ in $P$.
Since ${^{a_{u,\delta}}}\hat{F}_\mathfrak{h}=\hat{F}_\mathfrak{h}$,
${^{a_{u,\delta}}}\mathfrak{E}_\mathfrak{h}$ is another complement of $j+J(A_\delta^D)$
in $\hat{F}_\mathfrak{h}$.
Then by \cite[Lemma (45.6)]{Thevenaz} again,
$\mathfrak{E}_\mathfrak{h}$ and ${^{a_{u,\delta}}}\mathfrak{E}_\mathfrak{h}$ are conjugate 
in $\hat{F}_\mathfrak{h}$ by an element in $j+J(A_\delta^D)$.
So the element $a_{u,\delta}$ can be chosen such that it stabilizes $\mathfrak{E}_\mathfrak{h}$
by the conjugation action.
The second statement in this lemma follows from the equation
${^{a_{u,\delta}}}\mathfrak{e}_{\tilde{g}}(wj)({^{a_{u,\delta}}}\mathfrak{e}_{\tilde{g}})^{-1}=
({^u}gw(^{u}g)^{-1})j$ for any $\tilde{g}\in E_\mathfrak{h}$ and any $w\in D$.
\end{proof}

For the remainder of this paper,
we can assume that the element $a_{u,\delta}$ in $N_{A_\delta}(D)$ satisfies the properties
in Lemma \ref{action P on E_h} and $a_{u,\delta}(wj)a_{u,\delta}^{-1}=(uwu^{-1})j$
for any $u\in P$ and $w\in D$.
Now we can state the $\Delta(P)$-equivariant property of the bimodule $M_{{\rm st}}$.

\begin{prop}\label{equivariant of M_st}
	Keep the notation.
	Then the two modules $M_{\mathrm{st}}$ and ${_{(u,u)}}M_{\mathrm{st}}$ 
	are isomorphic to each other as $\mathbb{A}$-$\U L$-bimodules.
\end{prop}

\begin{proof}
Recall that from the paragraph 5.9 that 
we have an invertible element $s_u$ of $S=\mathrm{End}_\mathcal{O}(V)$ such that
$s_u(x\cdot\mathrm{v})=(uxu^{-1})\cdot s_u(\mathrm{v})$ for any $x\in L$ and any $\mathrm{v}\in V$.
Note that $M_{\mathrm{st}}=
\mathbb{A}j\mathop\otimes\limits_{\mathbb{A}_\delta}((V\mathop\otimes\limits_{\mathcal{O}}\mathbb{A}_\delta)\cdot e)$.
We can define the following map from $M_{\mathrm{st}}$ to ${_{(u,u)}}M_{\mathrm{st}}$ 
$$\tau:M_{\mathrm{st}}\longrightarrow{_{(u,u)}}M_{\mathrm{st}},~~~~
\mathfrak{a}j\mathop\otimes\limits_{\mathbb{A}_\delta}
((\mathrm{v}\mathop\otimes\limits_\mathcal{O}\mathfrak{a}_\delta)\cdot e)\mapsto
u^{-1}\mathfrak{a}a_{u,\delta}\mathop\otimes\limits_{\mathbb{A}_\delta}
((s_u^{-1}(\mathrm{v})\mathop\otimes\limits_\mathcal{O}a_{u,\delta}^{-1}\mathfrak{a}_\delta a_{u,\delta})\cdot e),$$
where $\mathfrak{a}$ is in $\mathbb{A}$ and $\mathrm{v}$ is in $V$
and $\mathfrak{a}_\delta$ is in $\mathbb{A}_\delta$.
It is easy to check that $\tau$ is well-defined and a bijection.
Let $\mathfrak{a}^\prime$ and $w$ and $\tilde{g}$ be in
$\mathbb{A}$ and $D$ and $E_\mathfrak{h}$, respectively.

\begin{align*}
	&\tau(\mathfrak{a}^\prime\cdot(\mathfrak{a}j\mathop\otimes\limits_{\mathbb{A}_\delta}
	((\mathrm{v}\mathop\otimes\limits_\mathcal{O}\mathfrak{a}_\delta)\cdot e))\cdot(w\tilde{g}))\\
	=~&\tau(\mathfrak{a}^\prime\mathfrak{a}j\mathop\otimes\limits_{\mathbb{A}_\delta}
	(((w\tilde{g})^{-1}(\mathrm{v})\mathop\otimes\limits_\mathcal{O}\mathfrak{a}_\delta\cdot w\mathfrak{e}_{\tilde{g}})\cdot e))\\
	=~&u^{-1}\mathfrak{a}^\prime\mathfrak{a}a_{u,\delta}\mathop\otimes\limits_{\mathbb{A}_\delta}
	((s_u^{-1}((w\tilde{g})^{-1}(\mathrm{v}))\mathop\otimes\limits_\mathcal{O}
	a_{u,\delta}^{-1}\mathfrak{a}_\delta\cdot w\mathfrak{e}_{\tilde{g}}a_{u,\delta})\cdot e)\\
	=~&(u^{-1}\mathfrak{a}^\prime u)\cdot u^{-1}\mathfrak{a}a_{u,\delta}\mathop\otimes\limits_{\mathbb{A}_\delta}
	((u^{-1}w\tilde{g}u)^{-1}\cdot s_u^{-1}(\mathrm{v})\mathop\otimes\limits_\mathcal{O}
	a_{u,\delta}^{-1}\mathfrak{a}_\delta a_{u,\delta}\cdot(a_{u,\delta}^{-1}w\mathfrak{e}_{\tilde{g}}
	a_{u,\delta})\cdot e)\\
	=~&\mathfrak{a}^\prime\cdot\tau((\mathfrak{a}j\mathop\otimes\limits_{\mathbb{A}_\delta}
	((\mathrm{v}\mathop\otimes\limits_\mathcal{O}\mathfrak{a}_\delta)\cdot e)))\cdot(w\tilde{g}),
\end{align*}
where the multiplications in the last equation are taken in 
${_{(u,u)}}M_{\mathrm{st}}$.
We complete the proof of this proposition.
\end{proof}

\vskip2cm

\section{Applications}
\subsection{Klein four group case}
\quad\, 
Keep the notation from the last three sections.
In this subsection, we will assume that the prime  $p$ is equal to $2$ and
the hyperfocal subgroup $D$ is a Klein four group.
In this case the hyperfocal quotient inertial $E_\mathfrak{h}$ is a cyclic group of order $3$
and acts freely on $D-\{1\}$.
Then we can apply Theorem \ref{MT} to this case and get a stable equivalence of Morita type
between $\mathbb{A}$ and $\U L$. 
Here $L=D\rtimes E_\mathfrak{h}$ 
which is just $\mathrm{A}_4$ in this section.
Here, $\mathrm{A}_n$ denotes the alternating group on $n$ letters for any positive integer $n$.
In fact, we can classify the Morita equivalence classes and the derived equivalence classes of
the hyperfocal subalgebras in this case,
which is similar to the case where the defect group is a Klein four group.

First, we need to calculate the numbers of simple $\K\mathop\otimes\limits_{\mathcal{O}}\mathbb{A}$-modules
and simple $k\mathop\otimes\limits_{\mathcal{O}}\mathbb{A}$-modules.

\begin{order}\rm
Denote $\K\mathop\otimes\limits_{\mathcal{O}}\mathbb{A}$ and $k\mathop\otimes\limits_{\mathcal{O}}\mathbb{A}$ 
by $\hat{\mathbb{A}}$ and $\bar{\mathbb{A}}$, respectively.
We use the analogous notation for the source algebra $A$.
It is clear that $\hat{\mathbb{A}}$ is a semi-simple $\K$-algebra (see \cite{KLN}).
By \cite[(29.20)]{CR's book},
there is a finite extension $\K^\prime$ of $\K$ such that the semi-simple $\K^\prime$-algebra
$\K^\prime\mathop\otimes\limits_{\mathcal{K}}\hat{\mathbb{A}}$ is split.
Since $k$ is algebraically closed,
the field $\K^\prime$ determines a complete $p$-modular system $(\K^\prime,\U^\prime,k)$
such that $\U^\prime$ is an extension of $\U$ by \cite[Chapter 2, \S2, Proposition 3]{Serre's book}.
Replacing $\U$ by $\U^\prime$, 
we can assume that the $\K$-algebra $\hat{\mathbb{A}}$ is a split semi-simple $\K$-algebra.
Then we can borrow the notation from the paragraphs 2.6-2.7 for these three $\U$-algebras
$A$ and $\mathbb{A}$ and $\mathbb{A}_\delta$.
\end{order}

\begin{lem}\label{k(A) and l(A)}
We have $\mathrm{k}(\mathbb{A})=4$ and $l(\mathbb{A})=3$.
\end{lem}

\begin{proof}
By Theorem \ref{MT}, 
there is a stable equivalence of Morita type between $\mathbb{A}$ and $\U L$.
By Corollary \ref{nonsingular of cartan matrix},
the Cartan matrix of $\bar{\mathbb{A}}$ is nonsingular.
Note that the hyperfocal subalgebra $\mathbb{A}$ is a symmetric $\U$-algebra
(see \cite[Theorem 1.1]{HZ25}).
Then by \cite[Proposition 3.1]{KL10}
the stable equivalence between $\mathbb{A}$ and $\U L$ induces an isometry
between $L^0(\mathbb{A})$ and $L^0(\U L)$.
It is clear that $L^0(\U L)$ is a free $\mathbb{Z}$-module of rank one.
Set $\zeta$ to be a generator.
It is well-known that $\langle\zeta,\zeta\rangle_{\mathcal{O}L}=4$.
Then $L^0(\mathbb{A})$ is also a free $\mathbb{Z}$-module of rank one
which implies that $\mathrm{k}(\mathbb{A})-l(\mathbb{A})=1$.
Suppose that $\mathrm{Irr}_\mathcal{K}(\mathbb{A})=\{\hat{U}_1,\hat{U}_2,\cdots,\hat{U}_{t+1}\}$
and $\mathrm{Irr}_k(\mathbb{A})=\{\bar{M}_1,\bar{M}_2,\cdots,\bar{M}_t\}$ 
for some positive integer $t$.

By \cite[Theorem 1.1]{HZ19}, $l(A)$ is equal to $2$ or $3$.
Suppose that $l(A)=3$. 
In this case, we have $P=C_P(D)$
by \cite[Theorem 1.1 (i) and Proposition 2.3 (a)]{HZ19}.
Then by Corollary \ref{equivariant for bbA-module},
every simple $\bar{\mathbb{A}}$-module is $P$-stable.
In particular, every simple $\bar{\mathbb{A}}$-module can be uniquely extended to 
a simple $\bar{A}$-module.
This implies that $\bar{A}$ and $\bar{\mathbb{A}}$ have the same number of isomorphism classes of simple modules.
In particular, we have $l(\mathbb{A})=3$ and then ${\rm k}(\mathbb{A})=4$.

Now suppose that $l(A)$ is equal to $2$.
In this case, the quotient group $P/C_P(D)$ is a group of order $2$ 
by \cite[Theorem 1.1 (ii) and Proposition 2.3 (b)]{HZ19}.
Then by Corollary \ref{equivariant for bbA-module} again,
the quotient group $P/C_P(D)$ can act on the set ${\rm Irr}_k(\mathbb{A})$.
By Clifford Theory, $l(\mathbb{A})$ has three possibilities: $2$, or $3$, or $4$. 
Assume that $l(\mathbb{A})$ equals $2$.
Then ${\rm k}(\mathbb{A})$ equals $3$.
Set $\sum\limits_{a=1}^{3}n_a[\hat{U}_a]$ is a generator of $L^0(\mathbb{A})$,
where $n_1$ and $n_2$ and $n_3$ are integers.
Since there is an isometry between $L^0(\mathbb{A})$ and $L^0(\U L)$,
we have $$\langle\sum\limits_{a=1}^{3}n_a[\hat{U}_a],\sum\limits_{a=1}^{3}n_a[\hat{U}_a]\rangle_{\mathbb{A}}
=\sum\limits_{a=1}^{3}n_a^2=4.$$
It is easy to see that all but one of the coefficients are zero and 
then $2[\hat{U}_a]$ belongs to $L^0(\mathbb{A})$ for some $1\leq a\leq 3$.
This is impossible.
Assume that $l(\mathbb{A})$ equals $4$.
In this subcase, without loss of generality,
we can assume that $\bar{M}_1$ and $\bar{M}_2$ are permuted by $P/C_P(D)$ 
and $\bar{M}_3$ and $\bar{M}_4$ are permuted by $P/C_P(D)$, respectively.
Set $\mathbb{A}_{C_P(D)}$ to be $\mathop\bigoplus\limits_{uD\in C_P(D)/D}\mathbb{A}u$.
Then for any $1\leq a\leq 4$,
the simple $\bar{\mathbb{A}}$-module $\bar{M}_a$ can be uniquely extended to
a simple $\bar{\mathbb{A}}_{C_P(D)}$-module, which we still denote by $\bar{M}_a$.
Here, $\bar{\mathbb{A}}_{C_P(D)}$ is the $k$-algebra 
$k\mathop\otimes\limits_\mathcal{O}\mathbb{A}_{C_P(D)}$.
Therefore, we can get that
${\rm Irr}_k(A)=\{\bar{A}\mathop\otimes\limits_{\bar{\mathbb{A}}_{C_P(D)}}\bar{M}_1,
\bar{A}\mathop\otimes\limits_{\bar{\mathbb{A}}_{C_P(D)}}\bar{M}_3\}$.
Note that ${\rm dim}_k(\bar{A}\mathop\otimes\limits_{\bar{\mathbb{A}}_{C_P(D)}}\bar{M}_1)
=2{\rm dim}_k(\bar{M}_1)$ and
${\rm dim}_k(\bar{A}\mathop\otimes\limits_{\bar{\mathbb{A}}_{C_P(D)}}\bar{M}_3)
=2{\rm dim}_k(\bar{M}_3)$.
This means that all simple $\bar{A}$-modules have even dimensions.
But by \cite[Proposition (44.9)]{Thevenaz},
there exists some simple $\bar{A}$-module of odd dimension, a contradiction.

In conclusion, we have $l(\mathbb{A})=3$ and then ${\rm k}(\mathbb{A})=4$.

\end{proof}

\begin{rmk}\rm
Another possible way to calculate ${\rm k}(\mathbb{A})$	is to analyse the coefficients
of a generator $\sum\limits_{a=1}^{t+1}n_a[\hat{U}_a]$ of $L^0(\mathbb{A})$
through the equations
$\langle\sum\limits_{a=1}^{t+1}n_a[\hat{U}_a],
\sum\limits_{a=1}^{t+1}n_a[\hat{U}_a]\rangle_{\mathbb{A}}=
\sum\limits_{a=1}^{t+1}n_a^2=4$.
This is exactly the approach taken when calculating $|{\rm Irr}(b)|$ for the block $b$
with a Klein four defect group (see \cite[Corollary 12.2.5]{L'book 2}).
This way is easy and efficient for the block 
due to a well-known fact of the block that 
an irreducible ordinary character is a $\mathbb{Z}$-linear combination of
some characters of projective modules if and only if it belongs to a block of a trivial defect group.
It is reasonable to speculate that 
a simple $\hat{\mathbb{A}}$-module $\hat{U}$ belongs to ${\rm Pr}_\mathcal{O}(\mathbb{A})$
if and only if the hyperfocal subgroup is trivial for the hyperfocal subalgebra $\mathbb{A}$
in general.
But this seems not easy to prove for us.
\end{rmk}

\begin{order}\rm
Let $\{T_1,T_2,T_3\}$ be a set of representatives of 
the isomorphism classes of simple $kL$-modules.	
Moreover, we can set $T_1$ to be the trivial simple $kL$-module $k$.
Denote by $\Omega=\Omega_{kL}$ the Heller operator of $kL$.
By \cite[Corollary 7.2.11]{L'book 2},
there are six isomorphism classes of indecomposable $kL$-modules of dimension $2$,
denoted by $T_\mathrm{i}^\mathrm{j}$ with distinct $\mathrm{i}$ and $\mathrm{j}$ in $\{1,2,3\}$.
Furthermore, the indecomposable $kL$-module $T_\mathrm{i}^\mathrm{j}$ has composition series 
$T_\mathrm{j}$, $T_\mathrm{i}$
and $\Omega(T_\mathrm{i}^\mathrm{j})=T_\mathrm{j}^\mathrm{l}$,
where $\mathrm{l}\in\{1,2,3\}-\{\mathrm{i,\mathrm{j}}\}$.
Since the two algebras $\mathbb{A}$ and $\mathbb{A}_\delta$ are Morita equivalent,
by Lemma \ref{k(A) and l(A)},
we can set $\{\bar{W}_1,\bar{W}_2,\bar{W}_3\}$ to be a set of representatives of 
the isomorphism classes of simple $\bar{\mathbb{A}}_\delta$-modules.
Now we can state the following main theorem in this subsection,
which  yields the Morita equivalence classes and the derived equivalence classes of the hyperfocal subalgebra $\mathbb{A}$ with a Klein four group.
\end{order}

\begin{thm}\label{MT of Klein four}
	Keep the notation as above.
Assume that the hyperfocal subgroup $D$ is a Klein four group.
Then the hyperfocal subalgebra $\mathbb{A}$ is Morita equivalent to either $\U\mathrm{A}_4$
or the principal block algebra of $\U\mathrm{A}_5$.
In particular, the hyperfocal subalgebra $\mathbb{A}$ is Rickard equivalent to $\U\mathrm{A}_4$.	
\end{thm}

\begin{proof}
For any nontrivial subgroup $Q$ of $D$, it is clear that 
the quotient group $D/Q$ is either trivial or a cyclic group of order $2$.
In each case, the primitive Dade $D$-interior algebra $S(Q)$ in the paragraph 5.6 is just $k$.
Then the primitive Dade $D$-interior algebra $S$ has to be the trivial $D$-interior algebra $\U$ 
and 
the primitive $L$-interior algebra $\mathbb{A}^\prime$ in the paragraph 5.12 
is just $\mathbb{A}_\delta$.
By applying Theorem \ref{MT},
we get that a stable equivalence of Morita type between $\mathbb{A}_\delta$ and $\U L$ 
which is induced by the restriction.
By Corollary \ref{dim of bbA and simple mod},
there is at least one of the simple $\bar{\mathbb{A}}_\delta$-modules of odd dimension.
Therefore, the proof of \cite[Proposition 12.2.9]{L'book 2} still applies to 
to our setting and then exactly one of the following statments holds.
\vspace{0.25cm}

\noindent (i) There is an integer $n$ such that for $1\leq\mathrm{i}\leq 3$ 
$$\mathrm{Res}_{kL}^{\bar{\mathbb{A}}_\delta}(\bar{W}_\mathrm{i})\cong\Omega^n(T_\mathrm{i}).$$

\noindent (ii) There is an integer $n$ such that
$$\mathrm{Res}_{kL}^{\bar{\mathbb{A}}_\delta}(\bar{W}_1)\cong\Omega^n(T_1),$$
$$\mathrm{Res}_{kL}^{\bar{\mathbb{A}}_\delta}(\bar{W}_2)\cong\Omega^n(T_3^2),$$
$$\mathrm{Res}_{kL}^{\bar{\mathbb{A}}_\delta}(\bar{W}_3)\cong\Omega^n(T_2^3).$$

In the case (i), 
we have $\bar{\mathbb{A}}_\delta\mathop\otimes\limits_{kL}T_\mathrm{i}\cong
\Omega^{-n}_{\bar{\mathbb{A}}_\delta}(\bar{W}_{\mathrm{i}})$ for any $1\leq\mathrm{i}\leq 3$.
Therefore by \cite[Theorem 4.14.10]{L'book 1},
the $\mathbb{A}_\delta$-$\U L$-bimodule 
$M=\Omega_{\mathbb{A}_\delta\mathop\otimes\limits_{\mathcal{O}}(\mathcal{O}L)^\circ}^n(\mathbb{A}_\delta)$
and its dual induce a Morita equivalence between $\mathbb{A}_\delta$ and $\U L$.
Hence, the $\mathbb{A}$-$\U L$-bimodule $\mathbb{A}j\mathop\otimes\limits_{\mathbb{A}_\delta}M$ and 
its dual $M^*\mathop\otimes\limits_{\mathbb{A}_\delta}j\mathbb{A}$ induce 
a Morita equivalence between $\mathbb{A}$ and $\U L$.

In the case (ii), we denote by $B_0(\mathrm{A}_5)$ the principal block algebra of $\U\mathrm{A}_5$.
Since $L$ is isomorphic to ${\rm A}_4$,
we can view $L$ as a subgroup of ${\rm A}_5$.
The same argument in the last paragraph of the proof of \cite[Theorem 12.1.2]{L'book 2}
shows that there is an indecomposable summand $N^\prime$ of 
the $\mathbb{A}_\delta$-$B_0(\mathrm{A}_5)$-bimodule $\mathbb{A}_\delta\mathop\otimes\limits_{\mathcal{O}L}B_0(\mathrm{A}_5)$ 
such that it and its dual induce a stable equivalence of Morita type between 
$\mathbb{A}_\delta$ and $B_0(\mathrm{A}_5)$ and 
the $\mathbb{A}_\delta$-$B_0(\mathrm{A}_5)$-bimodule 
$N=\Omega_{\mathbb{A}_\delta\mathop\otimes\limits_{\mathcal{O}}B_0(\mathrm{A}_5)^\circ}^n(N^\prime)$ and its dual
induce a Morita equivalence between $\mathbb{A}_\delta$ and $B_0(\mathrm{A}_5)$.
Hence, the $\mathbb{A}$-$B_0(\mathrm{A}_5)$-bimodule $\mathbb{A}j\mathop\otimes\limits_{\mathbb{A}_\delta}N$ and 
its dual $N^*\mathop\otimes\limits_{\mathbb{A}_\delta}j\mathbb{A}$ induce 
a Morita equivalence between $\mathbb{A}$ and $B_0(\mathrm{A}_5)$.

Finally, by \cite[Theorem 12.4.1]{L'book 2}, $\U L$ and $B_0(\mathrm{A}_5)$ are Rickard equivalent and 
then the hyperfocal subalgebra $\mathbb{A}$ is Rickard equivalent to $\U L$,
which is just $\U\mathrm{A}_4$.
\end{proof}

\begin{order}\rm
	In order to consider Brou$\acute{\mathrm{e}}$'s abelian defect group conjecture,
	we can assume that the defect group $P$ of the block $b$ is abelian
	and denote the Brauer correspondent of the block $b$ in $N_G(P)$ by $b_0$.
	Then it is well-known that the restriction of $N_G(P_\gamma)$ to $N_G(D_\delta)$ 
	induces an isomorphism from $N_G(P_\gamma)/C_G(P)$ to $E_\mathfrak{h}$.
	We identify $E_\mathfrak{h}$ with $N_G(P_\gamma)/C_G(P)$.
	Then $P=D\times R$ with $R=C_P(E_\mathfrak{h})$.
	As a corollary of Theorem \ref{MT of Klein four},
	we can get the Morita equivalence classes of the block algebra $\U Gb$ with abelian defect group $P$ and Klein four hyperfocal subgroup $D$ and then
	demonstrate that Brou$\acute{{\rm e}}$'s abelian defect group conjecture holds in this case.
\end{order}

\begin{prop}\label{Broue conj for Klein 4}
Keep the assumption as above.
The block algebra $\U Gb$ is Morita equivalent to either $\U(\mathrm{A}_4\times R)$ 
or the principal block algebra of $\U(\mathrm{A}_5\times R)$.
In particualr, Brou$\acute{e}$'s abelian defect group conjecture holds 
if the hyperfocal subgroup is a Klein four group.
\end{prop}

\begin{proof}
By the structure theory of source algebras with normal defect groups,
the source algebra of the block $b_0$ is isomorphic to $\U(L\times R)$ as $P$-interior algebras.
Denote $|\mathrm{IBr}(b)|$ by $l(b)$.
By \cite[Theorem 1.1 (i)]{HZ19},
we have $l(b)=3=|E_\mathfrak{h}|$.
Then by \cite[Theorem]{W05},
there is an isotypy between the blocks $b$ and $b_0$.
Recall that $A$ and $\mathbb{A}$ are a source algebra and a hyperfocal subalgebra of the block $b$, respectively.
Then by \cite[Proposition 4.8]{HZZ24},
there is an $\U$-algebra isomorphism $A\cong\mathbb{A}\mathop\otimes\limits_{\mathcal{O}}\U R$.
It is clear that as $\U$-algebras,
$\U(\mathrm{A}_4\times R)$ is isomorphic to $\U\mathrm{A}_4\mathop\otimes\limits_{\mathcal{O}}\U R$
and 
the principal block algebra of $\U(\mathrm{A}_5\times R)$ is isomorphic to 
$B_0(\mathrm{A}_5)\mathop\otimes\limits_{\mathcal{O}}\U R$.
Then the statements of the  proposition follow from Theorem \ref{MT of Klein four}.
\end{proof}

On the other hand, we can use Theorem \ref{MT of Klein four} to
obtain some information about 
$\mathrm{Irr}_\mathcal{K}(\mathbb{A})$ and $\mathrm{Irr}_\mathcal{K}(A)$.
We will first verify the forward direction of the KLN conjecture
for the hyperfocal subgroup $D$ being Klein four group.
We need the following two lemmas.

\begin{lem}\label{dim of irre A_delta module}
The dimension of every simple $\hat{\mathbb{A}}_\delta$-module is coprime to $2$.
\end{lem}

\begin{proof}
We set $\{\hat{U}_1,\hat{U}_2,\hat{U}_3,\hat{U}_4\}$ to be a set of representatives of 
the isomorphism classes of simple $\hat{\mathbb{A}}_\delta$-modules.

Suppose that $\mathbb{A}_\delta$ is Morita equivalent to $\U L$.
Then we are in the case (i) of Theorem \ref{MT of Klein four}.
In particular, the dimension of every simple $\bar{\mathbb{A}}_\delta$-module
$\bar{W}_\mathrm{i}$ is coprime to $2$ by \cite[Theorems 7.2.10 (ii) and 7.2.1 (ii)]{L'book 2}.
Since $\mathbb{A}_\delta$ is Morita equivalent to $\U L$,
we can assume that $d_{\mathbb{A}_\delta}([\hat{U}_4])=\sum\limits_{\mathrm{i}=1}^3[\bar{W}_\mathrm{i}]$ and
$d_{\mathbb{A}_\delta}([\hat{U}_\mathrm{i}])=[\bar{W}_\mathrm{i}]$
for any $\mathrm{i}\in\{1,2,3\}$.
By the definition of the decomposition map $d_{\mathbb{A}_\delta}$,
we have
$\mathrm{dim}_\mathcal{K}(\hat{U}_4)=
\sum\limits_{\mathrm{i}=1}^3\mathrm{dim}_k\bar{W}_\mathrm{i}$ and
$\mathrm{dim}_\mathcal{K}(\hat{U}_\mathrm{i})=\mathrm{dim}_k\bar{W}_\mathrm{i}$
for any $\mathrm{i}\in\{1,2,3\}$.
In particular, every simple $\hat{A}_\delta$-module has dimension coprime to $2$ in this case.

Suppose that $\mathbb{A}_\delta$ is Morita equivalent to 
the principal block $B_0(\mathrm{A}_5)$ of $\mathrm{A}_5$.
Then we are in the case (ii) of Theorem \ref{MT of Klein four}.
In particular, $\mathrm{dim}_k(\bar{W}_1)$ is coprime to $2$ and 
$\mathrm{dim}_k(\bar{W}_2)=\mathrm{dim}_k(\bar{W}_3)=2$ by \cite[Corollary 7.2.11]{L'book 2}.
Since $\mathbb{A}_\delta$ is Morita equivalent to 
the principal block $B_0(\mathrm{A}_5)$ of $\mathrm{A}_5$,
of which the decomposition matrix is well-known,
we can assume that
$d_{\mathbb{A}_\delta}([\hat{U}_1])=[\bar{W}_1]$ and 
$d_{\mathbb{A}_\delta}([\hat{U}_4])=[\bar{W}_1]+[\bar{W}_2]+[\bar{W}_3]$ and
$d_{\mathbb{A}_\delta}([\hat{U}_\mathrm{i}])=[\bar{W}_1]+[\bar{W}_\mathrm{i}]$
for any $\mathrm{i}\in\{2,3\}$.
Hence, we still obtain that every simple $\hat{A}_\delta$-module has dimension coprime to $2$ 
in this case.
We are done.
\end{proof}

\begin{lem}\label{dim of char of bbA and Adelta}
Let $U$ be a finitely generated $\mathbb{A}_\delta$-module
which is free as an $\U$-module.
Then the $\U$-rank of $U$ is coprime to $2$ if and only if
the $\U$-rank of $\mathbb{A}j\mathop\otimes\limits_{\mathbb{A}_\delta}U$ is coprime to $2$.
\end{lem}

\begin{proof}
	First, note that $\mathbb{A}j\mathop\otimes\limits_{\mathbb{A}_\delta}U$ is $\U$-free 
	since the $\mathbb{A}$-$\mathbb{A}_\delta$-bimodule $\mathbb{A}j$ 
	induces a Morita equivalence between $\mathbb{A}$ and $\mathbb{A}_\delta$.
	So the notation $\mathrm{rank}_\mathcal{O}(\mathbb{A}j\mathop\otimes\limits_{\mathbb{A}_\delta}U)$ makes sense.
	Denote by $m_\delta$ the multiplicity of $\mathbb{A}$ at $\delta$ 
	(see \cite[Definition 2.2]{P81}).
	Then by \cite[Propositions 3.3 and 3.5]{P00},
	we have $m_\delta$ coprime to $2$.
	Let $j^\prime$ be another element of $\delta$.
	Obviously, we have $j^\prime\mathbb{A}j$ isomorphic to $\mathbb{A}_\delta$
	as $\U D$-$\mathbb{A}_\delta$-bimodules.
	Since $D_\delta$ is the unique local pointed group of $D$ on $\mathbb{A}$,
	there is an $\U D$-module decomposition
	\begin{align}\label{dec of AjXWdelta}
		\mathbb{A}j\mathop\otimes\limits_{\mathbb{A}_\delta}U\cong
	U^{m_\delta}\oplus({\rm e}\mathbb{A}j\mathop\otimes\limits_{\mathbb{A}_\delta}U)	
	\end{align}
	for some idempotent ${\rm e}$ of $\mathbb{A}^D$ with $\mathrm{Br}_D({\rm e})=0$.
	Here, $U^{m_\delta}$ denotes a direct sum of $m_\delta$-copies of $U$.
	Suppose that ${\rm e}$ is primitive.
	By Rosenberg's lemma, we have ${\rm e}\in\mathrm{Tr}_Q^D(\mathbb{A}^Q)$ for some proper subgroup $Q$ of $D$.
	By \cite[Theorem (23.1)]{Thevenaz}, 
	there is a primitive idempotent ${\rm e}^\prime$ of $\mathbb{A}^Q$ such that 
	${\rm e}=\mathrm{Tr}_Q^D({\rm e}^\prime)$
	and ${\rm e}^\prime\cdot ({\rm e}^\prime)^v=({\rm e}^\prime)^v\cdot{\rm e}^\prime=0$ 
	for any $v\in D-Q$.
	For any $\mathfrak{a}\in {\rm e}\mathbb{A}j$ and any $x\in U$,
	the map sending $\mathfrak{a}\mathop\otimes\limits_{\mathbb{A}_\delta}x$ to
	$\sum\limits_{uQ\in D/Q}u\mathop\otimes\limits_{\mathcal{O}Q}
	({\rm e}^\prime u^{-1}\mathfrak{a}\mathop\otimes\limits_{\mathbb{A}_\delta}x)$
	determines an isomorphism of $\U D$-modules
	${\rm e}\mathbb{A}j\mathop\otimes\limits_{\mathbb{A}_\delta}U\cong
	\U D\mathop\otimes\limits_{\mathcal{O}Q}({\rm e}^\prime\mathbb{A}j
	\mathop\otimes\limits_{\mathbb{A}_\delta}U)$.
	Since ${\rm e}\mathbb{A}j\mathop\otimes\limits_{\mathbb{A}_\delta}U$ is isomorphic to
	a direct summand of $\mathbb{A}j\mathop\otimes\limits_{\mathbb{A}_\delta}U$
	as $\U D$-modules, ${\rm e}\mathbb{A}j\mathop\otimes\limits_{\mathbb{A}_\delta}U$ is also
	a free $\U$-module of finite rank.
	Furthermore, it is clear that 
	$\U D\mathop\otimes\limits_{\mathcal{O}Q}({\rm e}^\prime\mathbb{A}j
	\mathop\otimes\limits_{\mathbb{A}_\delta}U)$ is a direct sum of 
	$|D/Q|$-copies of ${\rm e}^\prime\mathbb{A}j\mathop\otimes\limits_{\mathbb{A}_\delta}U$
	as $\U$-modules.
	Hence, ${\rm e}^\prime\mathbb{A}j\mathop\otimes\limits_{\mathbb{A}_\delta}U$
	is a free $\U$-module of finite rank and 
	then $\mathrm{rank}_\mathcal{O}({\rm e}\mathbb{A}j
	\mathop\otimes\limits_{\mathbb{A}_\delta}U)=|D/Q|\cdot
	\mathrm{rank}_\mathcal{O}({\rm e}^\prime\mathbb{A}j\mathop\otimes\limits_{\mathbb{A}_\delta}U)$,
	which is divided by $2$ since $Q$ is a proper subgroup of $D$.
	
	In general, by decomposing ${\rm e}$ as a sum of pairwise orthogonal primitive idempotents of 
	$\mathbb{A}^D$,
	we still have 
	$\mathrm{rank}_\mathcal{O}({\rm e}\mathbb{A}j\mathop\otimes\limits_{\mathbb{A}_\delta}U)$
	divided by $2$.
	Hence, by the decomposition (\ref{dec of AjXWdelta}),
	we have $m_\delta\mathrm{rank}_\mathcal{O}(U)\equiv
	\mathrm{rank}_\mathcal{O}(\mathbb{A}j
	\mathop\otimes\limits_{\mathbb{A}_\delta}U)~(\mathrm{mod}~2)$.
	Then the lemma follows from $m_\delta$ being coprime to $2$.
\end{proof}

\begin{rmk}\label{general dim of char of bbA and Adelta}\rm
In fact, Lemma \ref{dim of char of bbA and Adelta} is true for any prime $p$ 
and any hyperfocal subgroup $D$
since the proof above does not depend on the assumption that
$D$ is a Klein four group.	
\end{rmk}

Combining Lemmas \ref{dim of irre A_delta module} and \ref{dim of char of bbA and Adelta}
with the Morita equivalence between $\mathbb{A}$ and $\mathbb{A}_\delta$,
we can get the following result about the dimensions of simple $\hat{\mathbb{A}}$-modules.

\begin{prop}\label{KLN conj for klein 4}
The dimension of every simple $\hat{\mathbb{A}}$-module is coprime to $2$.
In particular, the forward direction of the KLN conjecture is true in this case.
\end{prop}

\begin{order}\rm
Set $\mathrm{Irr}_\mathcal{K}(\mathbb{A})=\{\hat{U}_1,\hat{U}_2,\hat{U}_3,\hat{U}_4\}$ 
and $\mathrm{Irr}_k(\mathbb{A})=\{\bar{W}_1,\bar{W}_2,\bar{W}_3\}$ 
if no confusion arises.
By the proof of Lemma \ref{dim of irre A_delta module}, 
we can describe the decomposition map $d_\mathbb{A}$ of $\mathbb{A}$ as follows.
When $\mathbb{A}$ is Morita equivalent to $\U{\rm A}_4$, then
for any $\mathrm{i}\in\{1,2,3\}$,
\begin{align}\label{decomp map morita to L}
d_\mathbb{A}([\hat{U}_\mathrm{i}])=[\bar{W}_\mathrm{i}]~\mathrm{and}~
d_\mathbb{A}([\hat{U}_4])=\sum_{\mathrm{i}=1}^3[\bar{W}_\mathrm{i}] 	
\end{align}
When $\mathbb{A}$ is Morita equivalent to the principal block $B_0(\mathrm{A}_5)$ of $\mathrm{A}_5$, then
\begin{align}\label{decomp map morita to principal block}
	d_\mathbb{A}([\hat{U}_1])=[\bar{W}_1]~\mathrm{and}~
	d_\mathbb{A}([\hat{U}_4])=\sum_{\mathrm{i}=1}^3[\bar{W}_\mathrm{i}]~\mathrm{and}~
	d_\mathbb{A}([\hat{U}_\mathrm{i}])=[\bar{W}_1]+[\bar{W}_\mathrm{i}] 	
\end{align}
for any $\mathrm{i}\in\{2,3\}$.
It is clear that the conjugation action of $P$ on $\mathbb{A}$ can 
induce the conjugation actions of $P$ on $\mathrm{Irr}_\mathcal{K}(\mathbb{A})$ 
and $\mathrm{Irr}_k(\mathbb{A})$, respectively.
\end{order}

\begin{order}\rm
On the other hand, since $L$ is isomorphic to ${\rm A}_4$,
we can view $L$ as a subgroup of ${\rm A}_5$ such that 
$D$ equals $\{(1),(12)(34),(13)(24),(14)(23)\}$ and 
$E_\mathfrak{h}$ equals $\{(1),(123),(132)\}$.
We denote $C_P(D)$ by $P_0$.
It is clear that $P/P_0$ is isomorphic to a $2$-subgroup of 
the automorphism group ${\rm Aut}(D)$ of $D$.
Then $P/P_0$ is either trivial or a cyclic group of order $2$.
We can define an action of $P/P_0$ on ${\rm A}_5$ which stabilizes $L$ as follows.
When $P/P_0$ has order $2$, the action of $P/P_0$ on ${\rm A}_5$ is given by 
the conjugation action of the transposition $(12)$ in ${\rm S}_5$. 
Here, ${\rm S}_5$ denotes the symmetric group on $5$ letters.
So it can induce an action of $P$ on ${\rm A}_5$ through the canonical map
$P\longrightarrow P/P_0$ such that this action stabilizes 
$L$ and the principal block $B_0({\rm A}_5)$.
Moreover, set $P/P_0=\langle u_0P_0\rangle$ when $P/P_0$ is nontrivial.
Then $T_1$ is the unique $P$-stable $kL$-simple module.
$T_2$ and $T_3$ are both $P_0$-stable $kL$-simple modules
and ${_{u_0}}T_2=T_3$. 
Note that in both cases,
the actions of $P/P_0$ on $L=D\rtimes E_\mathfrak{h}$ are the same as 
the conjugation actions of $P/P_0$ on $L$.
Now we can get the following $\Delta(P)$-equivariant property of 
the Morita equivalences occurring in Theorem \ref{MT of Klein four}.
\end{order}

\begin{prop}\label{Morita C_P(D)-stable for Klein 4}
Keep the notation as above.
Then the Morita equivalences occurring in Theorem \ref{MT of Klein four}	
are both $\Delta(P)$-equivariant, namely that 
the bimodules induce the Morita equivalences occurring in Theorem \ref{MT of Klein four}
are $\Delta(P)$-equivariant.
\end{prop}

\begin{proof}
We will borrow the notation from the paragraph 4.9 and
 the proof of Theorem \ref{MT of Klein four} and Lemma \ref{action P on E_h}.
Fix an element $u$ of $P$.
By Lemma \ref{action P on E_h},
there is an element $a_{u,\delta}$ of $A_\delta^\times$ such that 
${^{a_{u,\delta}}}\mathfrak{e}_{\tilde{g}}=\mathfrak{e}_{{^u}\tilde{g}}$
for any $\tilde{g}\in E_{\mathfrak{h}}$.

First, we assume that the hyperfocal subalgebra $\mathbb{A}$ is Morita equivalent to $\U L$.
Then the $\mathbb{A}$-$\U L$-bimodule $\mathbb{A}j\mathop\otimes\limits_{\mathbb{A}_\delta}M$ 
and its dual induce a Morita equivalence between $\mathbb{A}$ and $\U L$,
where $M=\Omega_{\mathbb{A}_\delta\mathop\otimes\limits_{\mathcal{O}}(\mathcal{O}L)^\circ}^n(\mathbb{A}_\delta)$.
Then it is easy to check that
the conjugation action of $a_{u,\delta}$ on $\mathbb{A}_\delta$ can induce an 
$\mathbb{A}_\delta$-$\U L$-bimodule isomorphism  between $\mathbb{A}_\delta$ and ${_{(u,u)}}\mathbb{A}_\delta$.
By the construction of the Heller operator $\Omega$,
we can get an $\mathbb{A}_\delta$-$\U L$-bimodule isomorphism $\vartheta:M\longrightarrow{_{(u,u)}}M$.
Now we can define a map 
$$\varTheta:\mathbb{A}j\mathop\otimes\limits_{\mathbb{A}_\delta}M\longrightarrow
{_{(u,u)}}(\mathbb{A}j\mathop\otimes\limits_{\mathbb{A}_\delta}M),~~~~
aj\mathop\otimes\limits_{\mathbb{A}_\delta}m\mapsto u^{-1}aja_{u,\delta}\mathop\otimes\limits_{\mathbb{A}_\delta}\vartheta(m),$$
for any $aj\in\mathbb{A}j$  and any $m\in M$.
Then it is routine to check that
this map is well-defined and gives an $\mathbb{A}$-$\U L$-bimodule isomorphism.
In particular, this Morita equivalence is $\Delta(P)$-equivariant.

The remaining case is where $\mathbb{A}$ is Morita equivalent to the principal block 
$B_0(\mathrm{A}_5)$ of $\mathrm{A}_5$.
We have the analogue notation ${_{(u,u)}}U^\prime$ 
and ${_{(u,u)}}M^\prime$ for any 
$\mathbb{A}_\delta$-$B_0(\mathrm{A}_5)$-bimodule $U^\prime$ 
and any $\mathbb{A}$-$B_0(\mathrm{A}_5)$-bimodule $M^\prime$, respectively.
By \cite[Theorem 4.14.2]{L'book 1},
the $\mathbb{A}_\delta$-$B_0(\mathrm{A}_5)$-bimodule $N^\prime$,
which is an indecomposable direct summand of 
$\mathbb{A}_\delta\mathop\otimes\limits_{\mathcal{O}L}B_0(\mathrm{A}_5)$, 
is unique up to isomorphism.
Similarly, the conjugation actions of $a_{u,\delta}$ on $\mathbb{A}_\delta$
and $u$ on $B_0({\rm A}_5)$ can induce
an isomorphism between $\mathbb{A}_\delta\mathop\otimes\limits_{\mathcal{O}L}B_0(\mathrm{A}_5)$ and
${_{(u,u)}}(\mathbb{A}_\delta\mathop\otimes\limits_{\mathcal{O}L}B_0(\mathrm{A}_5))$
as $\mathbb{A}_\delta$-$B_0(\mathrm{A}_5)$-bimodules.
Therefore, by the uniqueness of $N^\prime$,
we have an $\mathbb{A}_\delta$-$B_0(\mathrm{A}_5)$-bimodule isomorphism
$\vartheta^\prime: N\longrightarrow{_{(u,u)}}N$ with the same argument above
since $N=\Omega_{\mathbb{A}_\delta\mathop\otimes\limits_{\mathcal{O}}B_0(\mathrm{A}_5)^\circ}^n(N^\prime)$.
Similar to the construction of map $\varTheta$,
we can get an $\mathbb{A}_\delta$-$B_0(\mathrm{A}_5)$-bimodule isomorphism
$\varTheta^\prime: \mathbb{A}j\mathop\otimes\limits_{\mathbb{A}_\delta}N\longrightarrow
{_{(u,u)}}\mathbb{A}j\mathop\otimes\limits_{\mathbb{A}_\delta}N$,
sending $aj\mathop\otimes\limits_{\mathbb{A}_\delta}n$ to 
$u^{-1}aja_{u,\delta}\mathop\otimes\limits_{\mathbb{A}_\delta}\vartheta^\prime(n)$
for any $aj\in\mathbb{A}j$  and any $n\in N$.
We are done.
\end{proof}

\begin{order}\rm
Keep the notation as above.
For any subgroup $Z$ of $P$ containing $D$,
we denote $\mathop\bigoplus\limits_{zD\in Z/D}\mathbb{A}z$ by $\mathbb{A}_Z$.	
So $A$ is just $\mathbb{A}_P$.
Set $\hat{\mathbb{A}}_Z$ to be $\mathcal{K}\mathop\otimes\limits_{\mathcal{O}}\mathbb{A}_Z$.
It is clear that $\hat{\mathbb{A}}_Z$ is a split semi-simple $\K$-algebra.
Recall that a simple $\hat{\mathbb{A}}_Z$-module $\hat{X}$ is called \emph{covering}
a simple $\hat{\mathbb{A}}$-module $\hat{U}_\mathrm{i}$ if
$\mathrm{Hom}_{\hat{\mathbb{A}}}
(\mathrm{Res}^{\hat{\mathbb{A}}_Z}_{\hat{\mathbb{A}}}(\hat{X}),\hat{U}_\mathrm{i})\neq 0$.
Suppose that a simple $\hat{\mathbb{A}}$-module $\hat{U}_\mathrm{i}$ can be extended to 
$\hat{\mathbb{A}}_Z$.
We still denote an extension of $\hat{U}_\mathrm{i}$ by $\hat{U}_\mathrm{i}$
if no confusion arises throughout this subsection.
For any $\mu\in\mathrm{Irr}(Z/D)$,
we identify it with its representation and 
then obtain a simple $\hat{\mathbb{A}}_Z$-module, denoted by $\hat{U}_\mathrm{i}\mu$
covering the simple $\hat{\mathbb{A}}$-module $\hat{U}_\mathrm{i}$.
Here, the representation $\rho_\mu$ of $\hat{U}_\mathrm{i}\mu$ is defined as
$\rho_\mu(\mathfrak{a}z)=\rho(\mathfrak{a}z)\mathop\otimes\limits_\mathcal{K}\mu(z)$ 
for any $\mathfrak{a}\in\mathbb{A}$
and any $z\in Z$ with $\rho$ being the representation of 
the simple $\hat{\mathbb{A}}_Z$-module $\hat{U}_\mathrm{i}$.
In fact, by Clifford Theory,
$\{\hat{U}_\mathrm{i}\mu\mid\mu\in\mathrm{Irr}(Z/D)\}$ is the set of all pairwise nonisomorphic 
simple $\hat{\mathbb{A}}_Z$-modules covering the simple $\hat{\mathbb{A}}$-module $\hat{U}_\mathrm{i}$
and 
$\{\hat{U}_\mathrm{i}\mu\mid\mu\in\mathrm{Irr}(Z/D)~\mathrm{with}~\mu(1)=1\}$
is the set of all distinct extensions of the simple $\hat{\mathbb{A}}$-module $\hat{U}_\mathrm{i}$
to $\hat{\mathbb{A}}_Z$.
At last, we denote $\hat{A}\mathop\otimes\limits_{\hat{\mathbb{A}}_Z}\hat{Y}$ by 
$\mathrm{Ind}_Z^P(\hat{Y})$ for any $\hat{\mathbb{A}}_Z$-module $\hat{Y}$.
We adopt the analogue notation above for simple $\bar{\mathbb{A}}_Z$-modules.
\end{order}

\begin{order}\rm
There are two possibilities on $|l(A)|$ (see \cite[Theorem 1.1]{HZ19}).
Let us recall some notation and results from \cite{HZ19}.
Suppose that $P_0=P$, namely, $D$ is in the center of $P$.
Then $N_G(P_\gamma)$ controls fusion of the block $b$ 
and $|l(A)|=3=|l(\mathbb{A})|$.
Furthermore, by the proof of \cite[Lemma 2.4]{HZ19},
we can identify $E_\mathfrak{h}$ with the inertial quotient $N_G(P_\gamma)/PC_G(P)$
through the restriction
and then $P=D\times C_P(E_\mathfrak{h})$.
We denote $C_P(E_\mathfrak{h})$ by $R$ in this case.
The other case is that $P_0<P$.
In this case, $|l(A)|=2$ and $|P:P_0|=2$.
Furthermore,
$E_\mathfrak{h}$ can be identified with a subgroup of 
$\mathrm{Aut}(P_0)$ and then $P_0=D\times C_{P_0}(E_\mathfrak{h})$ with this identification.
We denote $C_{P_0}(E_\mathfrak{h})$ by $R_0$ in this case.
\end{order}

Now we can get the structures about $\mathrm{Irr}_\mathcal{K}(A)$ and $\mathrm{Irr}_k(A)$
as follows.

\begin{prop}\label{structure of characters for klein 4}
Keep the notation as above.
Then exactly one of the two following statements holds.

\noindent (i) 
Suppose that $P_0=P$.
Then every simple $\hat{\mathbb{A}}$-module and 
every simple $\bar{\mathbb{A}}$-module are both $P$-stable
and can be extended to $\hat{A}$ and $\bar{A}$, respectively.
Therefore, we have 
$$\mathrm{Irr}_\mathcal{K}(A)=\{\hat{U}_\mathrm{i}\mu\mid\mathrm{i}=1,2,3,4~\mathrm{and}~
\mu\in\mathrm{Irr}(R)\}$$ and
$$\mathrm{Irr}_k(A)=\{\bar{W}_\mathrm{i}\mid\mathrm{i}=1,2,3\}.$$

\noindent (ii)
Suppose that $P_0<P$.
Then $\bar{W}_1$ is $P$-stable and 
the stabilizer of $\bar{W}_2$ in $P$ is $P_0$.
So the action of $P$ can permute $\bar{W}_2$ and $\bar{W}_3$.
Moreover, $\bar{W}_1$ can be extended to $\bar{A}$ and 
$\bar{W}_\mathrm{i}$ can be extended to $\bar{\mathbb{A}}_{P_0}$
for $\mathrm{i}=1,2$.
For $\mathrm{Irr}_\mathcal{K}(A)$,
we can obtain that
$\hat{U}_1$ and $\hat{U}_4$ are $P$-stable and 
the stabilizer of $\hat{U}_2$ in $P$ is $P_0$.
So the action of $P$ can permute $\hat{U}_2$ and $\hat{U}_3$.
Moreover, $\hat{U}_1$ and $\hat{U}_4$ can be extended to $\hat{A}$ and
$\hat{U}_\mathrm{i}$ can be extended to $\hat{\mathbb{A}}_{P_0}$
for $\mathrm{i}=1,2$.
In conclusion, we have
$$\mathrm{Irr}_\mathcal{K}(A)=\{\hat{U}_\mathrm{i}\mu\mid\mathrm{i}=1,4~\mathrm{and}~
\mu\in\mathrm{Irr}(P/D)\}\bigcup
\{\mathrm{Ind}_{P_0}^P(\hat{U}_\mathrm{i}\mu)\mid\mathrm{i}=2~\mathrm{or}~3~\mathrm{and}~
\mu\in\mathrm{Irr}(R_0)\}$$
and
$$\mathrm{Irr}_k(A)=\{\bar{W}_1,\mathrm{Ind}_{P_0}^P(\bar{W}_2)
\cong\mathrm{Ind}_{P_0}^P(\bar{W}_3)\}.$$
\end{prop}

\begin{proof}
For the case (i), by Proposition \ref{Morita C_P(D)-stable for Klein 4},
it is easy to see that every simple $\bar{\mathbb{A}}$-module is $P$-stable.
Then the extendibility is well known since $P$ is a $2$-group.
The description of $\mathrm{Irr}_k(A)$ can be deduced from this extendibility.
Let us consider simple $\hat{A}$-modules.
Obviously, the decomposition map $d_\mathbb{A}$ commutes with the action of $P$.
Then by the equations 
(\ref{decomp map morita to L}) and (\ref{decomp map morita to principal block}),
we can easily obtain that every simple $\hat{\mathbb{A}}$-module is $P$-stable. 
By Proposition \ref{KLN conj for klein 4},
$\mathrm{dim}_\mathcal{K}(\hat{U}_\mathrm{i})$ is coprime to $2$
for any $\mathrm{i}\in\{1,2,3,4\}$.
Then by \cite[Proposition 2.4]{KLN},
every simple $\hat{\mathbb{A}}$-module can be extended to $\hat{A}$.
Hence, the description of $\mathrm{Irr}_\mathcal{K}(A)$ follows from Clifford Theory
and the fact that $P/D$ is isomorphic to $R$.

For the case (ii),
by Proposition  \ref{Morita C_P(D)-stable for Klein 4},
every simple $\hat{\mathbb{A}}$-module and 
every simple $\bar{\mathbb{A}}$-module are both $P_0$-stable.
Moreover, the action of $P$ can permute the two simple $\bar{\mathbb{A}}$-module
$\bar{W}_2$ and $\bar{W}_3$.
Then the description of the action of $P$ on ${\rm Irr}_\mathcal{K}(\mathbb{A})$
can be similarly obtained by the equations  
(\ref{decomp map morita to L}) and (\ref{decomp map morita to principal block}).
With the same argument above, 
$\hat{U}_1$ and $\hat{U}_4$ can be both  extended to $\hat{A}$.
Note that the proof of \cite[Proposition 2.4]{KLN} is still valid if
we replace $P$ and $\hat{A}$ there by $P_0$ and $\hat{\mathbb{A}}_{P_0}$.
Then $\hat{U}_2$ and $\hat{U}_3$ can be both extended to $\hat{\mathbb{A}}_{P_0}$.
Similarly, the description of $\mathrm{Irr}_\mathcal{K}(A)$ follows from Clifford Theory
and the fact that $P_0/D$ is isomorphic to $R_0$.
\end{proof}
\vskip 1cm
		
\subsection{Cyclic group case}
\quad\, 
In this subsection, we will assume the hyperfocal subgroup $D$ is a nontrivial cyclic group.
In this case 
the hyperfocal quotient inertial $E_\mathfrak{h}$ is a nontrivial cyclic group of order divides $p-1$
and acts freely on $D-\{1\}$.
Then we can apply Theorem \ref{MT} to this case and get a stable equivalence of Morita type
between $\mathbb{A}$ and $\U L$. 
It is well-known that this stable equivalence of Morita type can be lifted to a 
Rickard equivalence by adopting the method in the proof of \cite[Theorem 11.12.1]{L'book 2}.
In the following, we will investigate the equivariant property of this Rickard equivalence.

\begin{order}\rm
Since the hyperfocal subgroup $D$ is cyclic,
by \cite[Theorem 3]{W14}, $N_G(P_\gamma)$ controls fusion of the block $b$.
So we can identify $E_\mathfrak{h}$ with the inertial quotient $E=N_G(P_\gamma)/PC_G(P)$
and then $P=D\rtimes C_P(E)$. 
We denote $C_P(E)$ by $R$.
Then the conjutation action of $R$ on $E$ is trivial.
Moreover, by \cite[Theorem 1]{W14}, $l(b)=|\mathrm{IBr}(b)|=|E|$.
Then $l(A)$ also equals $|E|$.
On the other hand, since $k\mathop\otimes\limits_{\mathcal{O}}\mathbb{A}$ is symmetric and  
stably equivalent to $kL$,
which is a serial algebra in this case, 
by \cite[Proposition 11.6.1]{L'book 2},
we have $l(\mathbb{A})=|\mathrm{IBr}(L)|=|E|$.
Hence, we have $l(\mathbb{A})=l(A)$ equal to the inertial index $|E|$ of the block.
It can be deduced from this equality that $P$ acts trivially on $\mathrm{Irr}_k(\mathbb{A})$
and then every simple $\bar{\mathbb{A}}$-module can be extended to $\bar{A}$.
\end{order}

\begin{order}\rm
Let $\mathfrak{i}$ and $\varsigma$ be primitive idempotents of $\mathbb{A}$ and $\U L$,
respectively.
For any $u\in P$, we have the twisted $\mathbb{A}$-$\U L$-bimodule ${_{(u,u)}}(\mathbb{A}\mathfrak{i}\mathop\otimes\limits_{\mathcal{O}}\varsigma\U L)$
(see the paragraph 4.9).
It is clear that the map sending 
$\mathfrak{a}\mathfrak{i}^{u^{-1}}\mathop\otimes\limits_{\mathcal{O}}\varsigma^{u^{-1}}x$ to 
$(\mathfrak{a}^u)\mathfrak{i}\mathop\otimes\limits_{\mathcal{O}}\varsigma x^u$ 
defines an isomorphism of $\mathbb{A}$-$\U L$-bimodules
from $\mathbb{A}\mathfrak{i}^u\mathop\otimes\limits_{\mathcal{O}}\varsigma^u\U L$ to 
${_{(u,u)}}(\mathbb{A}\mathfrak{i}\mathop\otimes\limits_{\mathcal{O}}\varsigma\U L)$,
where $\mathfrak{a}\in\mathbb{A}$ and $x\in L$.
By the paragraph above, $P$ acts trivially on $\mathrm{Irr}_k(\mathbb{A})$.
It is clear that $\mathbb{A}\mathfrak{i}^{u^{-1}}$ is isomorphic to $\mathbb{A}\mathfrak{i}$
as left $\mathbb{A}$-modules.
Similarly, $P$ acts on $\mathrm{IBr}(L)$ trivially and then
$\varsigma^{u^{-1}}\U L$ is isomorphic to $\varsigma\U L$ as right $\U L$-modules.
In conclusion, ${_{(u,u)}}(\mathbb{A}\mathfrak{i}\mathop\otimes\limits_{\mathcal{O}}\varsigma\U L)$
is isomorphic to $\mathbb{A}\mathfrak{i}\mathop\otimes\limits_{\mathcal{O}}\varsigma\U L$
for any $u\in P$.
Hence, we have the following equivariant property for projective $\mathbb{A}$-$\U L$-bimodules.
\end{order}

\begin{lem}\label{equivariant for proj. bimod}
Every finitely generated projetive $\mathbb{A}$-$\U L$-bimodule is $\Delta(P)$-stable.
Namely, if $\mathbb{P}$ is a finitely generated projective $\mathbb{A}$-$\U L$-bimodule
and $u$ belongs to $P$, then ${_{(u,u)}}\mathbb{P}$ is isomorphic to $\mathbb{P}$
as $\mathbb{A}$-$\U L$-bimodules.
\end{lem}

\begin{order}\rm
Let $M_\cdot=(M_n,\partial_n)$ be a bounded complex of $\mathbb{A}$-$\U L$-bimodules.
For any $u\in P$, we can define the twisted complex ${_{(u,u)}}(M_\cdot)$
with $({_{(u,u)}}(M_\cdot))_n={_{(u,u)}}(M_n)$ and the morphism from 
${_{(u,u)}}(M_\cdot)_n$ to ${_{(u,u)}}(M_\cdot)_{n-1}$ is just $\partial_n$.
The complex $M_\cdot$ is called $P$-\emph{invariant} if
$M_\cdot$ is isomorphic to ${_{(u,u)}}(M_\cdot)$ in 
$\mathcal{D}^b(\mathbb{A}\mathop\otimes\limits_{\mathcal{O}}(\U L)^\circ)$.
Furthermore, if the complex $M_\cdot$ induces a derived equivalence 
between $\mathbb{A}$ and $\U L$,
then this derived equivalence is called $P$-\emph{equivariant}.
We refer to \cite[Definition 2.5]{M99} for these two definitions.
\end{order}

\begin{thm}\label{MT for cyclic case}
There is a $P$-invariant $2$-term Rickard  complex $\mathbb{M}_\cdot$ of 
$\mathbb{A}$-$\U L$-bimodules.
In particular,
there is a $P$-equivariant Rickard equivalence between $\mathbb{A}$ and $\U L$.
\end{thm}

\begin{proof}
Recall from the paragraph 5.16 that we have an indecomposable 
$\mathbb{A}$-$\U L$-bimodule
$M_{\mathrm{st}}$ that induces a stable equivalence of Morita type between $\mathbb{A}$ and $\U L$.
By Proposition \ref{equivariant of M_st}, 
the two modules $M_{\mathrm{st}}$ and ${_{(u,u)}}M_{\mathrm{st}}$ are isomorphic to each other
as $\mathbb{A}$-$\U L$-bimodules for any $u\in P$.
Fix an isomoprhism $\varphi_{\mathrm{st}}$ between these two bimodules.

Let $\mathbb{P}$ be a projective cover of the bimodule $M_{\mathrm{st}}$
as $\mathbb{A}$-$\U L$-bimodules and $\pi$ a corresponding surjective homomorphism
from $\mathbb{P}$ to $M_{\mathrm{st}}$.
By \cite[Theorem 1.1]{HZ25}, the hyperfocal subalgebra $\mathbb{A}$ is symmetric.
Then by a slight generalization of the proof of \cite[Theorem 11.12.1]{L'book 2},
there is a direct summand $\mathbb{Q}$ of $\mathbb{P}$, 
uniquely determined by $M_{\mathrm{st}}$ up to isomorphism,
such that the following $2$-term complex $\mathbb{M}_\cdot$ of $\mathbb{A}$-$\U L$-bimodules
	\begin{eqnarray*}
	\xymatrix{
		\mathbb{Q}\ar[r]^{\mathrm{Res}^{\mathbb{P}}_{\mathbb{Q}}(\pi)}  &
		M_{\mathrm{st}}
	}
\end{eqnarray*}
is a Rickard complex of $\mathbb{A}$-$\U L$-bimodules
(see the proof of \cite[Theorem 5.24]{HZZ24}).
For simplicity, 
we still denote $\mathrm{Res}^{\mathbb{P}}_{\mathbb{Q}}(\pi)$ by $\pi$. 
Obviously, $\pi:{_{(u,u)}}\mathbb{P}\rightarrow{_{(u,u)}}M_{\mathrm{st}}$ is still 
a surjective homomorphism of $\mathbb{A}$-$\U L$-bimodules
and then ${_{(u,u)}}\mathbb{P}$ is a projective cover of ${_{(u,u)}}M_{\mathrm{st}}$.
Due to the facts that $M_{\mathrm{st}}$ is isomorphic to ${_{(u,u)}}M_{\mathrm{st}}$ and
projective covers are unique up to isomorphism, 
there is an isomorphism $\psi_{\mathrm{st}}$ of $\mathbb{A}$-$\U L$-bimodules
between $\mathbb{P}$ and ${_{(u,u)}}\mathbb{P}$ such that the following diagram commutes
	\begin{eqnarray}\label{comm. diag for P}
	\xymatrix{
	\mathbb{P}\ar[r]^{\pi}\ar[d]_{\psi_{\mathrm{st}}}  & M_{\mathrm{st}}\ar[d]_{\mathrm{\varphi_{\mathrm{st}}}}\\
	{_{(u,u)}}\mathbb{P}\ar[r]^\pi & {_{(u,u)}}M_{\mathrm{st}}.
	}
\end{eqnarray}

By the construction of $\mathbb{P}$ (see \cite[Theorem 11.9.1]{L'book 1}),
it is clear that every indecomposable direct summand of $\mathbb{P}$ has multiplicity one
and $\mathbb{P}$ is isomorphic to $(\mathbb{A}\mathop\otimes\limits_{\mathcal{O}}(\U L)^\circ)\mathfrak{f}$
as $\mathbb{A}$-$\U L$-bimodules
for some idempotent $\mathfrak{f}$ of $\mathbb{A}\mathop\otimes\limits_{\mathcal{O}}(\U L)^\circ$.
We can write ${_{(u,u)}}\mathbb{P}$ as ${_{(u,u)}}\mathbb{Q}\oplus\mathbb{Q}^\prime$
for some projective $\mathbb{A}$-$\U L$-bimodule $\mathbb{Q}^\prime$.
By Lemma \ref{equivariant for proj. bimod},
$\mathbb{Q}$ is isomorphic to ${_{(u,u)}}\mathbb{Q}$ as $\mathbb{A}$-$\U L$-bimodules.
Then we have $\psi_{\mathrm{st}}(\mathbb{Q})$ isomorphic to 
${_{(u,u)}}\mathbb{Q}$ as $\mathbb{A}$-$\U L$-bimodules.
Therefore, by \cite[Lemma (44.7)]{Thevenaz},
we have ${_{(u,u)}}\mathbb{P}=\psi_{\mathrm{st}}(\mathbb{Q})\oplus\mathbb{Q}^\prime$
as $\mathbb{A}$-$\U L$-bimodules.
We can assume that ${_{(u,u)}}\mathbb{P}=(\mathbb{A}\mathop\otimes\limits_{\mathcal{O}}(\U L)^\circ)\mathfrak{f}$.
Then the two decompositions ${_{(u,u)}}\mathbb{P}={_{(u,u)}}\mathbb{Q}\oplus\mathbb{Q}^\prime$ and
${_{(u,u)}}\mathbb{P}=\psi_{\mathrm{st}}(\mathbb{Q})\oplus\mathbb{Q}^\prime$ 
yields two orthogonal decompositions of the idempotent $\mathfrak{f}$ as follows
$$\mathfrak{f}=\mathfrak{f}_u+\mathfrak{f}^\prime_u~\mathrm{and}~
\mathfrak{f}=\mathfrak{f}_{\mathrm{st}}+\mathfrak{f}^\prime_{\mathrm{st}}$$
with ${_{(u,u)}}\mathbb{Q}=(\mathbb{A}\mathop\otimes\limits_{\mathcal{O}}(\U L)^\circ)\mathfrak{f}_u$
and $\psi_{\mathrm{st}}(\mathbb{Q})=
(\mathbb{A}\mathop\otimes\limits_{\mathcal{O}}(\U L)^\circ)\mathfrak{f}_{\mathrm{st}}$ and
$(\mathbb{A}\mathop\otimes\limits_{\mathcal{O}}(\U L)^\circ)\mathfrak{f}_u^\prime=\mathbb{Q}^\prime=
(\mathbb{A}\mathop\otimes\limits_{\mathcal{O}}(\U L)^\circ)\mathfrak{f}_{\mathrm{st}}^\prime$.
By the equation 
$(\mathbb{A}\mathop\otimes\limits_{\mathcal{O}}(\U L)^\circ)\mathfrak{f}_u^\prime=
(\mathbb{A}\mathop\otimes\limits_{\mathcal{O}}(\U L)^\circ)\mathfrak{f}_{\mathrm{st}}^\prime$,
we can get that $\mathfrak{f}^\prime_u\mathfrak{f}^\prime_{{\rm st}}=\mathfrak{f}^\prime_u$
and $\mathfrak{f}^\prime_{{\rm st}}\mathfrak{f}^\prime_u=\mathfrak{f}^\prime_{{\rm st}}$.
Then it is easy to check that
$\mathfrak{f}_u\mathfrak{f}_{\mathrm{st}}^\prime=0=\mathfrak{f}_{\mathrm{st}}^\prime\mathfrak{f}_u$
and 
$\mathfrak{f}_{\mathrm{st}}\mathfrak{f}_u^\prime=0=\mathfrak{f}_u^\prime\mathfrak{f}_{\mathrm{st}}$.
Hence, 
$$\mathfrak{f}_u=\mathfrak{f}\mathfrak{f}_u=\mathfrak{f}_{\mathrm{st}}\mathfrak{f}_u=
\mathfrak{f}_{\mathrm{st}}(\mathfrak{f}_u+\mathfrak{f}_u^\prime)=
\mathfrak{f}_{{\rm st}}\mathfrak{f}=\mathfrak{f}_{\mathrm{st}}.$$
So we have $\psi_{\mathrm{st}}(\mathbb{Q})={_{(u,u)}}\mathbb{Q}$.
By restricting from $\mathbb{P}$ to $\mathbb{Q}$,
we can deduce from the commutative diagram (\ref{comm. diag for P}) the following 
diagram commuting
	\begin{eqnarray}\label{comm. diag for Q}
	\xymatrix{
		\mathbb{Q}\ar[r]^{\pi}\ar[d]_{\psi_{\mathrm{st}}}  & M_{\mathrm{st}}\ar[d]_{\mathrm{\varphi_{\mathrm{st}}}}\\
		{_{(u,u)}}\mathbb{Q}\ar[r]^\pi & {_{(u,u)}}M_{\mathrm{st}},
	}
\end{eqnarray}
which implies that the complex $\mathbb{M}_\cdot$ is isomorphic to ${_{(u,u)}}\mathbb{M}_\cdot$
as complexes of $\mathbb{A}$-$\U L$-bimodules.
In particular, the complex $\mathbb{M}_\cdot$ is $P$-invariant.
We are done.
\end{proof}

\begin{order}\rm
Keep the notation as above.
Obviously, the complex $\mathcal{K}\mathop\otimes\limits_{\mathcal{O}}\mathbb{M}_\cdot$ 
induces a derived equivalence between $\hat{\mathbb{A}}$ and $\K L$,
which are both semi-simple $\mathcal{K}$-algebras.
Then for any simple $\hat{\mathbb{A}}$-module $\hat{U}$,
it is well-known that $\hat{U}[m]$ is isomorphic to 
$(\mathcal{K}\mathop\otimes\limits_{\mathcal{O}}\mathbb{M}_\cdot)\mathop\otimes\limits_{\mathcal{K}L}\Xi$
in $\mathcal{D}^b(\hat{\mathbb{A}})$
for some simple $\K L$-module $\Xi$ and some integer $m$.
Here, $[m]$ denotes the shift functor of $\mathcal{D}^b(\hat{\mathbb{A}})$.
Therefore,
$$\mathrm{End}_{\mathcal{D}^b(\hat{\mathbb{A}})}(\hat{U})\cong
\mathrm{End}_{\mathcal{D}^b(\hat{\mathbb{A}})}(\hat{U}[m])\cong
\mathrm{End}_{\mathcal{D}^b(\mathcal{K}L)}(\Xi)$$
as $\mathcal{K}$-algebras.
But it is well-known that
$$\mathrm{End}_{\mathcal{D}^b(\hat{\mathbb{A}})}(\hat{U})\cong
\mathrm{End}_{\hat{\mathbb{A}}}(\hat{U})~\mathrm{and}~
\mathrm{End}_{\mathcal{D}^b(\mathcal{K}L)}(\Xi)\cong
\mathrm{End}_{\mathcal{K}L}(\Xi)$$ 
as $\mathcal{K}$-algebras.
In conclusion, we can get
$\mathrm{End}_{\hat{\mathbb{A}}}(\hat{U})\cong\mathrm{End}_{\mathcal{K}L}(\Xi)$,
which is just $\mathcal{K}$.
This implies that the field $\mathcal{K}$ is a splitting field for $\hat{\mathbb{A}}$.
Therefore, the set $\mathrm{Irr}_\mathcal{K}(\mathbb{A})$ is an orthonormal basis
of $R_\mathcal{K}(\mathbb{A})$.
\end{order}

\begin{order}\rm
Next, we will investigate the structures of $\mathrm{Irr}_\mathcal{K}(A)$ and $\mathrm{Irr}_\mathcal{K}(\mathbb{A})$.
Set $|D|=p^n$ with $n\geq 1$ and $|E|=e$. 
Clearly, $e$ divides $p-1$.
By Theorem \ref{MT for cyclic case},
$\mathrm{k}(\mathbb{A})=|\mathrm{Irr}_\mathcal{K}(\mathbb{A})|=
|\mathrm{Irr}(L)|=e+\frac{p^n-1}{e}$
and $l(\mathbb{A})=|\mathrm{Irr}_k(\mathbb{A})|=|\mathrm{IBr}(L)|=e$.
Denote by $\mathcal{M}$ a set of representatives of the $E$-conjugacy classes of nontrivial irreducible characters of $D$.
Clearly, $|\mathcal{M}|=\frac{p^n-1}{e}$.
We can identify $\mathrm{Irr}(E)$ with a subset of $\mathrm{Irr}(L)$
through the canonical way.
Then 
$\mathrm{Irr}(L)=\{\mathrm{Ind}_D^L(\chi)\mid\chi\in\mathcal{M}\}\cup
\{\lambda\mid\lambda\in\mathrm{Irr}(E)\}$.
So by the Rickard equivalence obtained in Theorem \ref{MT for cyclic case},
we can set
$\{\hat{U}_\chi,\hat{U}_\lambda\mid\chi\in\mathcal{M},\lambda\in\mathrm{Irr}(E)\}$
to be a set of representatives of the isomorphism classes of 
simple $\hat{\mathbb{A}}$-modules
and $\{\bar{W}_\lambda\mid\lambda\in\mathrm{Irr}(E)\}$ 
to be a set of representatives of the isomorphism classes of 
simple $\bar{\mathbb{A}}$-modules.
By \cite[Proposition 3.2]{KL10}, the decomposition map $d_\mathbb{A}$ is surjective.
At the same time, by Corollary \ref{nonsingular of cartan matrix},
the Cartan matrix of $\mathbb{A}$ is non-singular.
Therefore, by \cite[Proposition 2.2]{KL10},
we have $L^0(\mathbb{A})^\bot=\mathrm{Pr}_\mathcal{O}(\mathbb{A})$.
For the group algebra $\U L$, 
we denote by $R_\mathcal{K}(\U L)$ the set of all generalized characters of $L$.
We adopt the similar notation for $R_k(\U L)$.
\end{order}


\begin{order}\rm
By Theorem \ref{MT for cyclic case},
the Rickard equivalence induced by the complex $\mathbb{M}_\cdot$
can determine two bijections 
$\Phi_\mathcal{K}:R_\mathcal{K}(\mathbb{A})\rightarrow R_\mathcal{K}(\U L)$	and
$\Phi_k:R_k(\mathbb{A})\rightarrow R_k(\U L)$ such that
$\Phi_\mathcal{K}$ is an isometry and the following the diagram commutes
	\begin{eqnarray}\label{comm diag for biject and decom map cyclic}
	\xymatrix{
		R_\mathcal{K}(\mathbb{A})\ar[r]^{\Phi_\mathcal{K}}\ar[d]_{d_\mathbb{A}}  & 
		R_\mathcal{K}(\U L)\ar[d]_{d_{\mathcal{O}L}}\\
		R_k(\mathbb{A})\ar[r]^{\Phi_k}& R_k(\U L).
	}
\end{eqnarray}
Then the restricting of $\Phi_\mathcal{K}$ to $L^0(\mathbb{A})$ induces
a bijective isometry, still denoted by $\Phi_\mathcal{K}$,
between $L^0(\mathbb{A})$ and $L^0(\U L)$.
Furthermore, since $\Phi_\mathcal{K}$ is an isometry,
we can get an bijective isometry
$\Phi_\mathcal{K}: \mathrm{Pr}_\mathcal{O}(\mathbb{A})\rightarrow
\mathrm{Pr}_\mathcal{O}(\U L)$.
Then we can obtain $\mathbb{Z}$-bases of $\mathrm{Pr}_\mathcal{O}(\mathbb{A})$ and
$L^0(\mathbb{A})$ as follows.
\end{order}

\begin{prop}\label{structure of character fo bbA cyclic}
For any $\chi\in\mathcal{M}$ and any $\lambda\in\mathrm{Irr}(E)$,
there are signs $\epsilon_\chi\in\{\pm 1\}$ and $ \epsilon_\lambda\in\{\pm 1\}$ such that
the following holds.

\noindent (i) $\Phi_\mathcal{K}([\hat{U}_\chi])= \epsilon_\chi\chi$ and
$\Phi_\mathcal{K}([\hat{U}_\lambda])= \epsilon_\lambda\lambda$.

\noindent (ii) $\{ \epsilon_\chi[\hat{U}_\chi]-
\sum\limits_{\lambda\in\mathrm{Irr}(E)} \epsilon_\lambda[\hat{U}_\lambda]\mid\chi\in\mathcal{M}\}$ is a $\mathbb{Z}$-basis of $L^0(\mathbb{A})$.
In particular, we have $ \epsilon_\chi= \epsilon_{\chi^\prime}$, denoted by $\epsilon$ for 
any $\chi,\chi^\prime\in\mathcal{M}$.

\noindent (iii) $\{ \epsilon_\lambda[\hat{U}_\lambda]+
\epsilon\sum\limits_{\chi\in\mathcal{M}}[\hat{U}_\chi]\mid\lambda\in\mathrm{Irr}(E)\}$
is a $\mathbb{Z}$-basis of $\mathrm{Pr}_\mathcal{O}(\mathbb{A})$.
\end{prop}

\begin{proof}
	The statements in this proposition can be easily checked 
	from the arguments in the paragraph above except the one 
	$\epsilon_\chi=\epsilon_{\chi^\prime}$.
	We can assume that $|\mathcal{M}|\geq 2$.
	Take two distinct elements $\chi$ and $\chi^\prime$ of $\mathcal{M}$.
	Then $\epsilon_\chi[\hat{U}_\chi]-\epsilon_{\chi^\prime}[\hat{U}_{\chi^\prime}]$
	belongs to $L^0(\mathbb{A})$.
	By the lemma below, we have
	$\epsilon_\chi\mathrm{dim}_\mathcal{K}(\hat{U}_\chi)-
	\epsilon_{\chi^\prime}\mathrm{dim}_\mathcal{K}(\hat{U}_{\chi^\prime})=0$,
	which forces $\epsilon_\chi=\epsilon_{\chi^\prime}$.
\end{proof}

\begin{lem}\label{dim of L^0=0}
Let $\sum\limits_{\chi\in\mathcal{M},\lambda\in\mathrm{Irr}(E)}
(m_\chi[\hat{U}_\chi]+m_\lambda[\hat{U}_\lambda])$
be an element in $L^0(\mathbb{A})$.
Here, $m_\chi$ and $m_\lambda$ are integers.
We have 
$\sum\limits_{\chi\in\mathcal{M},\lambda\in\mathrm{Irr}(E)}
(m_\chi\mathrm{dim}_\mathcal{K}(\hat{U}_\chi)+
m_\lambda\mathrm{dim}_\mathcal{K}(\hat{U}_\lambda))=0$.
\end{lem}

\begin{proof}
It is clear that taking the dimension of any simple $\hat{\mathbb{A}}$-module
can induce a $\mathbb{Z}$-linear map ${\rm dim}_\mathcal{K}$ 
from $R_\mathcal{K}(\mathbb{A})$ to $\mathbb{Z}$.	
Similarly, taking the dimension of any simple $\bar{\mathbb{A}}$-module
can also induce a $\mathbb{Z}$-linear map ${\rm dim}_k$ 
from $R_k(\mathbb{A})$ to $\mathbb{Z}$.		
Let $X$ be a finite dimensional $\hat{\mathbb{A}}$-module.
Set $d_{\mathbb{A}}([X])=\sum\limits_{\lambda\in{\rm Irr}(E)}a_\lambda[\bar{W}_\lambda]$
for some nonnegative integer $a_\lambda$.
By the definition of the decomposition map $d_\mathbb{A}$,
we can get that 
${\rm dim}_\mathcal{K}(X)=\sum\limits_{\lambda\in{\rm Irr}(E)}a_\lambda{\rm dim}_k(\bar{W}_\lambda)$.
This implies that ${\rm dim}_k\circ d_\mathbb{A}={\rm dim}_\mathcal{K}$.
Since $L^0(\mathbb{A})$ is the kernel of the decomposition map $d_\mathbb{A}$,
we complete the proof of this lemma.
\end{proof}

\begin{order}\rm
Now we can verify the forward direction of the KLN conjecture 
when the hyperfocal subgroup $D$ is nontrivial cyclic.
We first recall a well-known fact that there is a simple $\hat{A}$-module 
of dimension coprime to $p$.
Indeed, by \cite[Proposition 6.11.11]{L'book 2},
the $p$-part of the dimension of a simple $\hat{A}$-module equals to the height
of its corresponding irreducible ordinary character in the block $b$
through the Morita equivalence induced by the bimodule $i\U G$.
\end{order}

\begin{lem}\label{dim coprime to p cyclic}
There is at least one simple $\hat{\mathbb{A}}$-module of dimension coprime to $p$.	
\end{lem}

\begin{proof}
	This can be easily deduced from the argument above and \cite[Proposition 2.3 (iv)]{KLN}.	
\end{proof}

\begin{prop}\label{KLN conj cyclic}
The dimension of every simple $\hat{\mathbb{A}}$-module is coprime to $p$ 
when the hyperfocal subgroup $D$ is nontrivial cyclic.
In particular, the forward direction of the KLN conjecture is true in this case.
\end{prop}

\begin{proof}
Since the hyperfocal subgroup $D$ is assumed to be nontrivial,
the order $e$ of $E$ can not be $1$.
Then by Proposition \ref{structure of character fo bbA cyclic} (iii),
$\epsilon_\lambda[\hat{U}_\lambda]-\epsilon_\mu[\hat{U}_\mu]$ 
belongs to $\mathrm{Pr}_\mathcal{O}(\mathbb{A})$
for any two distinct elements $\lambda$ and $\mu$ of $\mathrm{Irr}(E)$.
Since $\mathbb{A}$ is projective as left $\U D$-module,
every finitely generated projective $\mathbb{A}$-module has $\U$-rank divided by $|D|$.
Hence, we have $p$ dividing 
$\epsilon_\lambda\mathrm{dim}_\mathcal{K}(\hat{U}_\lambda)-
\epsilon_\mu\mathrm{dim}_\mathcal{K}(\hat{U}_\mu)$.
In particular, $\mathrm{dim}_\mathcal{K}(\hat{U}_\lambda)$ is coprime to $p$ if and only if
$\mathrm{dim}_\mathcal{K}(\hat{U}_\mu)$ is coprime to $p$.
Suppose that $\mathrm{dim}_\mathcal{K}(\hat{U}_\lambda)$ is divided by $p$ 
for some $\lambda\in\mathrm{Irr}(E)$.
So is $\mathrm{dim}_\mathcal{K}(\hat{U}_\mu)$ for any $\mu\in\mathrm{Irr}(E)$.
Now by Proposition \ref{structure of character fo bbA cyclic} (ii) and
Lemma \ref{dim of L^0=0},
$\epsilon\mathrm{dim}_\mathcal{K}(\hat{U}_\chi)-
\sum\limits_{\lambda\in\mathrm{Irr}(E)}\epsilon_\lambda\mathrm{dim}_\mathcal{K}(\hat{U}_\lambda)=0$.
Hence, $\mathrm{dim}_\mathcal{K}(\hat{U}_\chi)$ is divided by $p$ 
for any $\chi\in\mathcal{M}$.
In conclusion, every simple $\hat{\mathbb{A}}$-module has dimension divided by $p$,
which contradicts Lemma \ref{dim coprime to p cyclic}.
So $\mathrm{dim}_\mathcal{K}(\hat{U}_\lambda)$ is coprime to $p$ for any $\lambda\in\mathrm{Irr}(E)$.
By Proposition \ref{structure of character fo bbA cyclic} (iii),
there is at least one element $\chi$ in $\mathcal{M}$ such that
$\mathrm{dim}_\mathcal{K}(\hat{U}_\chi)$ is coprime to $p$.
We denote it by $\chi_0$.
Now we can assume that $\mathcal{M}$ has at least two elements.
For any other $\chi\in\mathcal{M}$ different from $\chi_0$,
by Proposition \ref{structure of character fo bbA cyclic} (ii),
$\epsilon([\hat{U}_{\chi_0}]-[\hat{U}_\chi])$ belongs to $L^0(\mathbb{A})$.
In particular, by Lemma \ref{dim of L^0=0},
these two simple modules $\hat{U}_{\chi_0}$ and $\hat{U}_\chi$ has the same dimension.
Then we are done.
\end{proof}

\begin{order}\rm
Note that the structure of a $\mathbb{Z}$-basis of $\mathrm{Pr}_\mathcal{O}(\mathbb{A})$
in Proposition \ref{structure of character fo bbA cyclic} plays an important role 
in the proof of Proposition \ref{KLN conj cyclic}.
Furthermore, by this structure,
we can also calculate decomposition numbers of $\mathbb{A}$ and
get that these numbers are either $1$ or $0$ in the following.
This result can be regarded as a `hyperfocal decomposition number' version
of the classical fact that decomposition numbers of blocks with nontrivial cyclic defect groups
are either $1$ or $0$.
By \cite[\S 2]{KL10},
this is equivalent to showing that $[\mathcal{K}\mathop\otimes\limits_{\mathcal{O}}U]=
\sum\limits_{\theta\in\mathrm{Irr}(L)}a_{\theta}[\hat{U}_\theta]$
with $a_\theta$ equal to $1$ or $0$ 
for any projective indecomposable $\mathbb{A}$-module $U$.
One key technique in the context of blocks with nontrivial cyclic defect groups 
is making use of $\U$-pure submodules.
An $\mathbb{A}$-submodule $U^\prime$ of an $\mathbb{A}$-module $U$ is $\U$-\emph{pure}
in $U$ if it is a direct summand of $U$ as $\U$-modules (see \cite[Section 4.2]{L'book 1}).
Our proof here is similar to the one of \cite[Theorem 11.10.5]{L'book 2}.
We need the following two lemmas which are analogues of 
\cite[Lemma 11.10.3]{L'book 2} and \cite[Proposition 11.10.4]{L'book 2}.
\end{order}

\begin{lem}\label{lemma 11.10.3}
Let $U$ be a finitely generated $\U$-free $\mathbb{A}$-module such that
$\bar{U}=k\mathop\otimes\limits_{\mathcal{O}}U$ is nonprojective indecomposable.
Then the image $[\hat{U}]$ of $\hat{U}=\mathcal{K}\mathop\otimes\limits_{\mathcal{O}}U$ 
in $R_\mathcal{K}(\mathbb{A})$ is not contained in $\mathrm{Pr}_\mathcal{O}(\mathbb{A})$.
\end{lem}	

\begin{proof}
Suppose that $[\hat{U}]$ is contained in ${\rm Pr}_\mathcal{O}(\mathbb{A})$.
By the argument in the proof of Proposition \ref{KLN conj cyclic},
$|D|$ divides the dimension ${\rm dim}_\mathcal{K}(\hat{U})$ of $\hat{U}$.
Recall from Section $5$ that 
there is an $\U$-algebra $\mathbb{A}^\prime$ Morita equivalent to $\mathbb{A}$ such that
it has a subalgebra isomorphic to $\U L$ and  a stable equivalence of Morita type 
between $\mathbb{A}^\prime$ and this subalgebra induced by induction and restriction exists.
We identify this subalgebra with $\U L$.
We denote by $U^\prime$ the $\mathbb{A}^\prime$-module 
corresponding to $U$ via the Morita equivalence.
Then the $\mathbb{A}^\prime$-module $U^\prime$ satisfies the same properties as 
the $\mathbb{A}$-module $U$.
In particular, $|D|$ also divides the $\U$-rank ${\rm rank}_\mathcal{O}(U^\prime)$ of $U^\prime$.	
Let $V^\prime$ be up to isomorphism the unique indecomposable nonprojective
direct summand of ${\rm Res}_{\mathcal{O}L}^{\mathbb{A}^\prime}(U^\prime)$.
Hence, the $\U$-rank ${\rm rank}_\mathcal{O}(V^\prime)$ of $V^\prime$ is also divided by $|D|$.
On the other hand, by \cite[Proposition 4.14.6]{L'book 1},
$k\mathop\otimes\limits_{\mathcal{O}}V^\prime$ remains indecomposable and nonprojective.
This implies that $k\mathop\otimes\limits_{\mathcal{O}}V^\prime$ 
is a uniserial nonprojective $kL$-module.
It is well-known that every uniserial nonprojective $kL$-module 
has dimension strictly less that $|D|$.
So we have a contradiction.
\end{proof}

\begin{lem}\label{prop 11.10.4}
Let $U$ be a projective indecomposable $\mathbb{A}$-module.
Denote $\mathcal{K}\mathop\otimes\limits_{\mathcal{O}}U$ by $\hat{U}$.
Suppose that there are nonzero finite dimensional $\hat{\mathbb{A}}$-modules
$\hat{U}_1$ and $\hat{U}_2$	such that $[\hat{U}]=[\hat{U}_1]+[\hat{U}_2]$.
Then none of $[\hat{U}_1]$ or $[\hat{U}_2]$ is contained in $\mathrm{Pr}_\mathcal{O}(\mathbb{A})$.
\end{lem}

\begin{proof}
Since $\hat{\mathbb{A}}$ is semi-simple,
we have $\hat{U}\cong\hat{U}_1\oplus\hat{U}_2$ as $\hat{\mathbb{A}}$-modules.
Then by \cite[Theorem 4.16.4]{L'book 1},
there is an $\U$-pure submodule $U_1$ of $U$ such that
$\hat{U}_1$ is isomorphic to $\mathcal{K}\mathop\otimes\limits_{\mathcal{O}}U_1$ as $\hat{\mathbb{A}}$-modules.
By \cite[Proposition 4.2.6]{L'book 1},
$k\mathop\otimes\limits_{\mathcal{O}}U_1$ can be viewed as a submodule of $k\mathop\otimes\limits_{\mathcal{O}}U$
which is a projective indecomposable $\bar{\mathbb{A}}$-module.
Since $\mathbb{A}$ is a symmetric $\U$-algebra,
the socle of $k\mathop\otimes\limits_{\mathcal{O}}U$ is simple
and $k\mathop\otimes\limits_{\mathcal{O}}U_1$ can not be projective
due to the indecomposability of $k\mathop\otimes\limits_{\mathcal{O}}U$. 
At the same time, this implies that 
the socle of $k\mathop\otimes\limits_{\mathcal{O}}U_1$ is also simple.
In particular, $k\mathop\otimes\limits_{\mathcal{O}}U_1$ is indecomposable.
By Lemma \ref{lemma 11.10.3}, $[\hat{U}_1]$ is not contained in $\mathrm{Pr}_\mathcal{O}(\mathbb{A})$.
\end{proof}

\begin{prop}\label{decomp number of bbA cyclic}
With the notation as in Proposition \ref{structure of character fo bbA cyclic}.
Let $U$ be a projective indecomposable $\mathbb{A}$-module.
Then either $[\hat{U}]=[\hat{U}_\lambda]+[\hat{U}_{\lambda^\prime}]$ 
for some $\lambda,\lambda^\prime\in\mathrm{Irr}(E)$ such that
$\epsilon_\lambda\neq\epsilon_{\lambda^\prime}$, or
$[\hat{U}]=[\hat{U}_\lambda]+\sum\limits_{\chi\in\mathcal{M}}[\hat{U}_\chi]$
for some $\lambda\in\mathrm{Irr}(E)$ such that $\epsilon_\lambda=\epsilon$.
In particular, all decomposition numbers of $\mathbb{A}$ are either $1$ or $0$.
\end{prop}

\begin{proof}
Due to the structure of $\mathrm{Pr}_\mathcal{O}(\mathbb{A})$ described in
Proposition \ref{structure of character fo bbA cyclic} (iii),
by Lemmas \ref{lemma 11.10.3} and \ref{prop 11.10.4},
the proof is similar to the one of \cite[Theorem 11.10.5]{L'book 2}. 	
\end{proof}

\begin{order}\rm
Now we will give an explicit description of $\mathrm{Irr}_\mathcal{K}(A)$
by the Clifford theoretic relationship between the representation theory of 
the source algebra $A$ and the hyperfocal subalgebra $\mathbb{A}$.
This description can be obtained once it can be proved that the isometry $\Phi_\mathcal{K}$
in the commutative diagram (\ref{comm diag for biject and decom map cyclic}) 
commutes with the conjugation actions of $P$ on 
$\mathrm{Irr}_\mathcal{K}(\mathbb{A})$ and $\mathrm{Irr}(L)$.
This fact seems to be established by directly applied \cite[Theorem 2.6]{M99} to our case
since the isometry $\Phi_\mathcal{K}$ is induced by 
the $P$-invariant Rickard complex $\mathbb{M}_\cdot$.
However, the validity of Theorem 2.6 in \cite{M99} depends on an assumption that 
the symmetrizing forms of $A$ and $\U(P\rtimes E)$ are $R$-invariant symmetrizing forms
for $\mathbb{A}$ and $\U L$,
which means that there is an $\mathbb{A}$-$\mathbb{A}$-bimodule
($\U L$-$\U L$-bimodule, resp.) isomorphism between $\mathbb{A}$ ($\U L$, resp.) and 
$\mathbb{A}^*$ ($(\U L)^*$, resp.) commuting with the conjugation action of $R$
(see \cite[5.1.A]{M'book}).
This assumption obviously holds for $\U L$ and $\U(P\rtimes E)$. 
For $\mathbb{A}$ and $A$, 
the proof of Theorem 1.1 in \cite{HZ25} implicitly demonstrates the validity of this assumption
(see \cite[Section 4]{HZ25}).
In conclusion, the isometry $\Phi_\mathcal{K}$ in the commutative diagram 
(\ref{comm diag for biject and decom map cyclic}) commutes with the conjugation actions of $P$.
\end{order}

\begin{order}\rm
We need some more notation to describe the structure of ${\rm Irr}_\mathcal{K}(A)$.
Let $\mathcal{M}_R$ be a set of representatives of the $E\times R$-conjugacy classes of 
nontrivial irreducible characters of $D$. 
Without loss of generality, we can assume that $\mathcal{M}_R$ is contained in $\mathcal{M}$.
For any $\theta\in\mathrm{Irr}(L)$,
set $R_\theta$ to be the subgroup of $R$ such that
$D\rtimes R_\theta$ is the stabilizer of $\theta$ under the conjugation action of $P$.
It is clear that $C_P(D)$ is contained in $D\rtimes R_\theta$ 
and then $R/R_\theta$ is cyclic for any $\theta\in\mathrm{Irr}(L)$.
Let us borrow the notation from the paragraph 6.15.
Moreover, since $Z=D\rtimes(Z\cap R)$ for any subgroup $Z$ of $P$ containing $D$,
we denote $\mathop\bigoplus\limits_{zD\in Z/D}\mathbb{A}z$ and the induced module
$\hat{A}\mathop\otimes\limits_{\hat{\mathbb{A}}_Z}\hat{Y}$
by $\mathbb{A}_{Z\cap R}$ and $\mathrm{Ind}_{Z\cap R}^R(\hat{Y})$
instead of $\mathbb{A}_Z$ and $\mathrm{Ind}_Z^P(\hat{Y})$
for any subgroup $Z$ of $P$ containing $D$ and 
any $\hat{\mathbb{A}}_Z$-module $\hat{Y}$, respectively.
Then we can obtain the structure of $\mathrm{Irr}_\mathcal{K}(A)$ as follows.
\end{order}

\begin{prop}\label{structure of characters for cyclic}
Keep the notation as above.	
The following statements hold.

\noindent (i) For any $\chi\in\mathcal{M}_R$,
the stabilizer of $\hat{U}_\chi$ under the conjugation action of $P$ is just $R_\chi$.
So $\hat{U}_\chi$ is $P$-stable if and only if $\chi$ is $P$-stable.
Moreover, for any $\chi\in\mathcal{M}_R$,
the simple $\hat{\mathbb{A}}$-module $\hat{U}_\chi$ can be extended to 
a simple $\hat{\mathbb{A}}_{R_\chi}$-module,
denoted by $\hat{U}_\chi^{\mathrm{ext}}$.

\noindent (ii) For any $\lambda\in\mathrm{Irr}(E)$,
the simple $\hat{\mathbb{A}}$-module $\hat{U}_\lambda$ is $P$-stable and 
then it can be extended to a simple $\hat{A}$-module, 
denoted by $\hat{U}_\lambda^{\mathrm{ext}}$.

\noindent (iii) 
$\mathrm{Irr}_\mathcal{K}(A)=
\{\mathrm{Ind}_{R_\chi}^R(\hat{U}_\chi^{\mathrm{ext}}\mu_\chi)\mid\chi\in\mathcal{M}_R,
\mu_\chi\in\mathrm{Irr}(R_\chi)\}
\bigcup\{\hat{U}_\lambda^{\mathrm{ext}}\mu\mid
\lambda\in\mathrm{Irr}(E), \mu\in\mathrm{Irr}(R)\}$.
\end{prop}

\begin{proof}
By the arguments in the paragraph 6.35,
the isometry $\Phi_\mathcal{K}$ in the 
commutative diagram (\ref{comm diag for biject and decom map cyclic})
preserves the actions of $P$ on $\mathrm{Irr}_\mathcal{K}(\mathbb{A})$ and $\mathrm{Irr}(L)$.
Then the statements on stabilizers of simple $\hat{\mathbb{A}}$-modules are clear.
For any $\chi\in\mathcal{M}_R$,
since $\mathrm{dim}_\mathcal{K}(\hat{U}_\chi)$ is coprime to $p$,
the same argument in the proof of Proposition \ref{structure of characters for klein 4} shows that
$\hat{U}_\chi$ can be extended to $\hat{\mathbb{A}}_{R_\chi}$.
For any $\lambda\in\mathrm{Irr}(E)$,
since $\hat{U}_\lambda$ is $P$-stable and has dimension coprime to $p$,
by \cite[Proposition 2.4]{KLN},
$\hat{U}_\lambda$ can be extended to $\hat{A}$.
This shows (i) and (ii).
The statement (iii) just follows from \cite[Theorem 1.6]{F09} or \cite[Proposition 2.3]{KLN}.
\end{proof}

As a conesquence, we can get the following decomposition numbers of the block
$b$ with a nontrivial cyclic hyperfocal subgroup.

\begin{cor}\label{decomp number of block cyclic}
With the notation as above,
all decomposition numbers of the block $b$ are either $p^l$ or $0$
with $1\leq p^l\leq|R/C_R(D)|$
when the hyperfocal subgroup $D$ is nontrivial cyclic.	
In particular,
when the hyperfocal subgroup $D$ is nontrivial cyclic and central in $P$,
all decomposition numbers of the block $b$ are either $1$ or $0$.
\end{cor}

\begin{proof}
Since the block algebra $\U Gb$ is Morita equivalent to the source algebra $A$,
it suffices to show that all decomposition numbers of the source algebra $A$
satisfy this property.
As before, we can get these numbers by calculating the image $[\hat{U}_A]$
of $\hat{U}_A$ in $R_\mathcal{K}(A)$
of any projective indecomposable  $A$-module $U_A$ 
with $\hat{U}_A=\mathcal{K}\mathop\otimes\limits_{\mathcal{O}}U_A$.
Since every simple $\bar{\mathbb{A}}$-module can be extended to $\bar{A}$
(see the paragraph 6.18),
then every  projective indecomposable $A$-module has the form
$A\mathfrak{i}$, which is isomorphic to $A\mathop\otimes\limits_{\mathbb{A}}\mathbb{A}\mathfrak{i}$
for some primitive idempotent $\mathfrak{i}$ of $\mathbb{A}$.
Now fix a primitive idempotent $\mathfrak{i}$ of $\mathbb{A}$.
By Proposition \ref{decomp number of bbA cyclic},
we can assume that as $\hat{\mathbb{A}}$-modules,
$\hat{\mathbb{A}}\mathfrak{i}\cong\hat{U}_\lambda\oplus\hat{U}_{\lambda^\prime}$
for some suitable distinct $\lambda,\lambda^\prime\in\mathrm{Irr}(E)$ or
$\hat{\mathbb{A}}\mathfrak{i}\cong\hat{U}_\mu\oplus
(\mathop\bigoplus\limits_{\chi\in\mathcal{M}}\hat{U}_\chi)$ for some suitable $\mu\in\mathrm{Irr}(E)$.
Fix these notation.
Then as $\hat{A}$-modules,
$\hat{A}\mathfrak{i}\cong\hat{A}\mathop\otimes\limits_{\hat{\mathbb{A}}}\mathbb{A}\mathfrak{i}\cong
\mathrm{Ind}_1^R(\hat{U}_\lambda)\oplus\mathrm{Ind}_1^R(\hat{U}_{\lambda^\prime})$
or
$\hat{A}\mathfrak{i}\cong\hat{A}\mathop\otimes\limits_{\hat{\mathbb{A}}}\mathbb{A}\mathfrak{i}\cong
\mathrm{Ind}_1^R(\hat{U}_\mu)\oplus
(\mathop\bigoplus\limits_{\chi\in\mathcal{M}}\mathrm{Ind}_1^R(\hat{U}_\chi))$.
By Proposition \ref{structure of characters for cyclic},
as $\hat{A}$-modules,
$\mathrm{Ind}_1^R(\hat{U}_\nu)\cong\mathop\bigoplus\limits_{\vartheta\in\mathrm{Irr}(R)}
\hat{U}_\nu^{\mathrm{ext}}\vartheta$ for any $\nu\in\mathrm{Irr}(E)$ 
and
$\mathrm{Ind}_1^R(\hat{U}_\chi)\cong\mathop\bigoplus\limits_{\mu_\chi\in\mathrm{Irr}(R_\chi)}
\mathrm{Ind}_{R_\chi}^R(\hat{U}_\chi^{\mathrm{ext}}\mu_\chi)$ 
for any $\chi\in\mathcal{M}$.
When two elements $\chi$ and $\chi^\prime$ of $\mathcal{M}$ 
belong to the same $R$-orbit,
the two $\hat{A}$-modules $\mathrm{Ind}_1^R(\hat{U}_\chi)$ 
and $\mathrm{Ind}_1^R(\hat{U}_{\chi^\prime})$ are isomorphic to each other.
Since $C_R(D)$ acts trivially on $\mathrm{Irr}_\mathcal{K}(\mathbb{A})$,
it is easy to get the information on decomposition numbers of $A$
as stated in this corollary.
\end{proof}

	\end{document}